\documentclass[11pt,a4paper]{amsart}
\usepackage{geometry}
\usepackage{mathrsfs} 
\usepackage{enumitem}
\usepackage{amssymb}
\usepackage{hyperref}
\usepackage{mathtools}
\usepackage{tikz}
\usetikzlibrary{positioning,cd,decorations.pathreplacing}

\usepackage[
  backend=biber,
  style=alphabetic,
  doi=false
]{biblatex}
\usepackage{cleveref}

\usepackage{subcaption}

\usepackage{mathabx}

\numberwithin{equation}{section}

\theoremstyle{plain}
\newtheorem{theorem}[equation]{Theorem}
\newtheorem{corollary}[equation]{Corollary}
\newtheorem{lemma}[equation]{Lemma}
\newtheorem{proposition}[equation]{Proposition}
\newtheorem{claim}[equation]{Claim}

\theoremstyle{definition}
\newtheorem{definition}[equation]{Definition}

\newtheorem*{convention}{Convention}
\newtheorem{example}[equation]{Example}

\theoremstyle{remark}
\newtheorem{remark}[equation]{Remark}
\newtheorem{question}{Question}

\AddToHook{env/theorem/begin}{\crefalias{equation}{theorem}}
\AddToHook{env/lemma/begin}{\crefalias{equation}{lemma}}
\AddToHook{env/corollary/begin}{\crefalias{equation}{corollary}}
\AddToHook{env/proposition/begin}{\crefalias{equation}{proposition}}
\AddToHook{env/claim/begin}{\crefalias{equation}{claim}}
\AddToHook{env/definition/begin}{\crefalias{equation}{definition}}
\AddToHook{env/assumption/begin}{\crefalias{equation}{assumption}}
\AddToHook{env/convention/begin}{\crefalias{equation}{convention}}
\AddToHook{env/example/begin}{\crefalias{equation}{example}}
\AddToHook{env/remark/begin}{\crefalias{equation}{remark}}
\AddToHook{env/question/begin}{\crefalias{equation}{question}}

\crefname{theorem}{Theorem}{Theorems}
\crefname{lemma}{Lemma}{Lemmas}
\crefname{corollary}{Corollary}{Corollaries}
\crefname{proposition}{Proposition}{Propositions}
\crefname{claim}{Claim}{Claims}
\crefname{definition}{Definition}{Definitions}
\crefname{assumption}{Assumption}{Assumptions}
\crefname{convention}{Convention}{Conventions}
\crefname{example}{Example}{Examples}
\crefname{remark}{Remark}{Remarks}
\crefname{question}{Question}{Questions}

\usepackage{stackengine}
\stackMath

\DeclareMathOperator{\UD}{UD}
\DeclareMathOperator{\D}{D}
\DeclareMathOperator{\UC}{UC}
\DeclareMathOperator{\Stab}{Stab}
\DeclareMathOperator{\B}{Big}
\DeclareMathOperator{\diam}{diam}
\DeclareMathOperator{\CAT}{CAT}
\DeclareMathOperator{\sgn}{sgn}

\newcommand{\GBG}{\mathbb{B}}
\newcommand{\NZ}{\mathbb{Z}_{\ge 0}}

\newcommand{\Chat}{\mathord{\widehat{\mathcal{C}{}}}}

\newcommand{\EO}{%
  \mathord{\stackon[-9.4pt]{\mathcal{G}{}}{\widehat{\phantom{\mathcal{G}{}}}}}%
}
\newcommand{\EOW}{\EO^{\raisebox{-0.6ex}{$\scriptstyle\mathfrak S$}}}
\newcommand{\EOLP}{\EO^{\raisebox{-0.6ex}{$\scriptstyle{\mathfrak{S}_{\mathrm{LP}}}$}}}

\newcommand{\GBGdisc}{\mathbb{B}^{\mathrm{disc}}}
\newcommand{\PGBGdisc}{\mathbb{P}^{\mathrm{disc}}}

\newcommand{\sfC}{\mathsf{C}}
\newcommand{\sfD}{\mathsf{D}}

\newcommand{\sfK}{\mathsf{K}}
\newcommand{\sfL}{\mathsf{L}}

\newcommand{\sfS}{\mathsf{St}}
\newcommand{\sfT}{\mathsf{T}}

\newcommand{\Graf}{\mathsf{\Gamma}}

\newcommand{\sfe}{\mathsf{e}}

\newcommand{\sfu}{\mathsf{u}}
\newcommand{\sfv}{\mathsf{v}}
\newcommand{\sfw}{\mathsf{w}}

\newcommand{\sfz}{\mathsf{z}}

\newcommand{\bfc}{\mathbf{c}}

\newcommand{\bfn}{\mathbf{n}}

\newcommand{\bfx}{\mathbf{x}}
\newcommand{\bfy}{\mathbf{y}}
\newcommand{\bfz}{\mathbf{z}}

\newcommand{\bfE}{\mathbf{E}}

\newcommand{\bfone}{\mathbf{1}}

\newcommand{\bbA}{\mathbb{A}}
\newcommand{\bbF}{\mathbb{F}}

\newcommand{\bbP}{\mathbb{P}}

\newcommand{\frG}{\mathfrak{G}}
\newcommand{\frS}{\mathfrak{S}}
\newcommand{\frSLP}{\mathfrak{S}_{\mathrm{LP}}}
\newcommand{\frF}{\mathfrak{F}}
\newcommand{\frFLP}{\mathfrak{F}_{\mathrm{LP}}}
\newcommand{\frg}{\mathfrak{g}}

\renewcommand\bar\overline
\renewcommand\tilde\widetilde

\newcommand{\con}{\mathrm{con}}
\newcommand{\free}{\mathrm{free}}

\begin{document}
\title[Hierarchical geometry and RAAGs in graph braid groups]{Hierarchical geometry and right-angled Artin groups in graph braid groups}

\author{Byung Hee An}
\email{anbyhee@knu.ac.kr}
\address{Department of Mathematics Education, Kyungpook National University, Daegu, South Korea 41566}
\author{Sangrok Oh}
\email{sangrokoh.math@gmail.com}
\address{Innovation Center for MathScience Research \& Education, Pusan National University, Busan, South Korea 46241}
\author{Jihoon Park}
\email{ttsiug@knu.ac.kr}
\address{Department of Mathematics Education, Kyungpook National University, Daegu, South Korea 41566}


%\date{\today}
\begin{abstract}
For the unordered discrete configuration space $\mathrm{UD}_n(\mathsf{\Gamma})$ of $n$ particles on a connected finite graph $\mathsf{\Gamma}$, we construct an explicit factor system on its universal cover. Its factors are encoded by legal pairs, namely subgraphs equipped with particle distributions. The nesting, orthogonality, and product regions in the resulting hierarchically hyperbolic group (HHG) structure admit explicit descriptions in terms of configuration-space geometry, and we show that this structure satisfies the additional properties needed for constructing and obstructing subgroups isomorphic to right-angled Artin groups (RAAGs).

Using sufficiently subdivided models, we apply this hierarchy to graph braid groups. We give a finite combinatorial formula for the maximal rank of a free abelian subgroup and show that every RAAG occurs as an undistorted subgroup of some graph braid group. For graph $2$-braid groups, we obtain stronger restrictions: every RAAG subgroup has bipartite defining graph, and the embedding problem is characterized by an induced-subgraph condition in the expanded core graph of the hierarchy. For the RAAG defined by the four-vertex path, this condition is equivalent to a finite graphical criterion on the underlying graph.
\end{abstract}

\maketitle
%\setcounter{tocdepth}{1}
%\tableofcontents

\section{Introduction}
For a connected finite simplicial graph \(\Graf\), let \(\UC_n(\Graf)\) denote the unordered configuration space of \(n\) distinct points in \(\Graf\), and write \(\GBG_n(\Graf)=\pi_1(\UC_n(\Graf))\) for the corresponding graph \(n\)-braid group; see \Cref{Def:UC_n}. Graph braid groups have been studied from both algebraic and geometric viewpoints; see, for example, \cite{Abr00,CW,FS05,Gen21GBG,Oh22,AnO26Oncertain}. Their relationship with right-angled Artin groups (RAAGs) is particularly prominent:
Crisp and Wiest proved that graph braid groups embed into RAAGs \cite{CW}, while Sabalka proved that every RAAG embeds into some graph braid group using his construction of halo graphs \cite{Sab07Embedding}. Beyond these embedding results, the problem of determining which graph braid groups are themselves RAAGs has been studied in \cite{FS05,KKP12}, and the large-scale geometries of graph braid groups and RAAGs have been compared in \cite{Fer12,Oh22}.

This relationship is closely tied to the cubical geometry of configuration spaces. Let \(\UD_n(\Graf)\) denote the unordered discrete configuration space, and set \(\GBGdisc_n(\Graf)=\pi_1(\UD_n(\Graf))\); see \Cref{Def:UD_n}.
For every connected finite graph \(\Graf\) with \(n\leq\#V(\Graf)\), independently of any subdivision hypothesis, \(\UD_n(\Graf)\) is a compact special cube complex in the sense of Haglund--Wise \cite{HW08}; see \cite{CW,Gen21GBG}.
If \(\Graf\) is sufficiently subdivided for \(n\) (roughly speaking, if it contains sufficiently many bivalent vertices along paths between non-bivalent vertices and around cycles, with the required numbers depending on \(n\); see \Cref{Def:SufficientlySubdivided}), then \(\UD_n(\Graf)\) is a strong deformation retract of \(\UC_n(\Graf)\) \cite{Abr00,KKP12,PS12}, and hence \(\GBGdisc_n(\Graf)\cong\GBG_n(\Graf)\).

The compact special cubical structure places \(\GBGdisc_n(\Graf)\) within the theory of hierarchically hyperbolic groups (HHGs). Introduced by Behrstock--Hagen--Sisto, this axiomatic theory includes mapping class groups and fundamental groups of compact special cube complexes \cite{BHS17I,BHS19II,HS20}. Its distance formula recovers the large-scale metric from projections to hyperbolic spaces, its realization theorem characterizes which collections of projection data come from points, and orthogonal domains determine standard product regions \cite{BHS17I,BHS19II}. An HHG structure also gives rise to a natural boundary \cite{DHS17,DHS20}. Together, these features have made HHG structures effective tools for studying product geometry, quasiflats, rank, and stable subgroups; see, for example, \cite{ABD21,BHS21,HRSS25}.

In \cite{OP25,OPRAAG}, the last two authors developed an approach based on the nesting and orthogonality relations among HHG domains for constructing and obstructing RAAG embeddings. At the heart of this approach is the \emph{expanded core graph}, whose vertices represent axial directions supported on \emph{nesting-minimal unbounded domains}, namely the minimal elements among the unbounded domains under nesting, and whose edges record orthogonality; it is designed as an HHG counterpart of the extension graph of a RAAG \cite{KK13}.
To make the orthogonality recorded by this graph algebraically effective, the last two authors introduced a class \(\Xi\) of HHGs satisfying three structural properties: suitable domain stabilizers inherit HHG structures, powers of elements decompose into factors strongly fully supported on their active orthogonal domains, and strongly fully supported elements on orthogonal domains commute; see \Cref{Subsection:HHG}.
For an HHG in \(\Xi\), the expanded core graph governs both constructions and obstructions for RAAG embeddings and, when the HHG has rank at most two, gives a complete criterion.

The purpose of this paper is twofold: to construct an explicit HHG structure for graph braid groups and establish its membership in \(\Xi\), and to use this structure, together with general HHG results, to study RAAG embeddings into graph braid groups. We call this structure the \emph{legal-pair hierarchy} and denote it by 
\[\bigl(\GBGdisc_n(\Graf),\frSLP\bigr).\]
It is defined directly from configuration subspaces determined by subgraphs of \(\Graf\) and distributions of particles among their components.

The HHG structures arising from rich families on compact special cube complexes are known to belong to \(\Xi\) \cite[Proposition~3.9]{OPRAAG}. For our applications, we seek a hierarchy whose domains and orthogonality admit explicit descriptions in terms of configuration-space geometry. The legal-pair factor system need not coincide with either canonical rich-family construction; see \Cref{Rem:RichFamilyVersusLegalPairs,Ex:RichFamilyNotLegalPair}.
We therefore prove directly that the resulting hierarchy belongs to \(\Xi\). This allows us to apply the RAAG-embedding results of \cite{OP25,OPRAAG}, formulated in terms of its expanded core graph, which we denote by \(\EOLP\).

The legal-pair hierarchy \((\GBGdisc_n(\Graf),\frSLP)\) and its expanded core graph \(\EOLP\) are constructed from a chosen discrete model of \(\Graf\). Nevertheless, the information used in the group-theoretic statements below is stable under subdivision. Once \(\Graf\) is sufficiently subdivided, the nontrivial connected components of \(\EOLP\) are independent, up to graph isomorphism, of the chosen subdivision. The graph-theoretic conditions appearing in the remaining applications are also preserved under subdivision. Consequently, the applications can be stated directly for the usual graph braid group \(\GBG_n(\Graf)\), without imposing a subdivision hypothesis on the original graph \(\Graf\).

We first present these group-theoretic conclusions, postponing the construction of the hierarchy, the description of its domains, the concrete description of \(\EOLP\), and the precise subdivision-stability statement until afterward.

\begin{convention}
We denote by \(\bbA(\Lambda)\) the RAAG defined by a finite simplicial graph \(\Lambda\) and by \(\bbF_r\) the free group of rank \(r\geq 1\).
We use \emph{RAAG subgroup} as shorthand for a subgroup isomorphic to a RAAG. Such a subgroup is \emph{undistorted} if its inclusion into the ambient group is a quasi-isometric embedding.
For two groups \(G\) and \(H\), we write \(H<G\) if \(G\) contains a subgroup isomorphic to \(H\). For graphs \(\Lambda\) and \(\Omega\), we write \(\Lambda<\Omega\) if \(\Lambda\) is isomorphic to an induced subgraph of \(\Omega\). 
\end{convention}

\subsection{Main applications}
% Our applications investigate RAAG subgroups of graph braid groups from several complementary directions. We combine the legal-pair hierarchy, general results about HHGs, the RAAG-embedding machinery of \cite{OP25,OPRAAG}, and direct arguments in configuration spaces to construct undistorted RAAG subgroups and to obtain obstructions and, in some cases, complete criteria for RAAG embeddings. 
We begin with two basic families of RAAGs: free abelian groups and direct products of free groups. A vertex of \(\Graf\) is called \emph{essential} if it has degree at least \(3\). 

\begin{theorem}[Free abelian rank formula, \Cref{Thm:RankGBG}]\label{Thm:IntroRankGBG}
Let \(\Graf\) be a connected finite graph and let \(n\geq 2\). Then the maximal rank of a free abelian subgroup of \(\GBG_n(\Graf)\) is
\[\max\left\{ p+q\ \middle|\
\begin{array}{c}
\text{\(p\) embedded cycles and \(q\) essential vertices in \(\Graf\)}\\
\text{can be chosen pairwise disjoint, and }p+2q\leq n
\end{array}
\right\}.\]
\end{theorem}

Since \(\Graf\) has only finitely many embedded cycles and essential vertices, the displayed maximum can be computed by finite enumeration.

Writing \(r\) for the displayed maximum, the lower bound is realized by an explicit undistorted subgroup \(\mathbb Z^r\). Related constructions of free abelian subgroups already appear in earlier work; compare, for instance, \cite[Proposition~3.6]{AnO26Oncertain}. The new point here is the matching upper bound, which follows from the general rank bound for HHGs. 
In particular, this determines the maximal rank exactly and allows us to characterize the pairs \((\Graf,n)\) for which every free abelian subgroup of \(\GBG_n(\Graf)\) has rank at most two; see \Cref{Thm:IntroRankTwoGBG}.

A slight variation of the construction realizing the lower bound also recovers the stable-range theorem of Jankiewicz--Schreve \cite{JS25GBG}. Let \(m(\Graf)\) be the number of essential vertices of \(\Graf\), and let \(m_3(\Graf)\) be the number of trivalent vertices. If \(m(\Graf)>0\) and \(n\geq2m(\Graf)+m_3(\Graf)\), then \(\GBG_n(\Graf)\) contains an undistorted subgroup isomorphic to \(\bbF_2^{\,m(\Graf)}\); see \Cref{Ex:LegalPairProductSubgroups}\eqref{Item:StableRangeLegalPair}.

We also strengthen Sabalka's universal RAAG-embedding result by proving that the resulting RAAG subgroups can be chosen undistorted.

\begin{corollary}[Universal undistorted RAAG embeddings, \Cref{Cor:EveryRAAGUndistortedGBG}]\label{Cor:IntroEveryRAAGUndistortedGBG}
Let \(\Lambda\) be a finite simplicial graph. For every integer \(n\geq\chi(\Lambda)\), there exists a connected finite graph \(\Graf\) such that \(\bbA(\Lambda)\) is isomorphic to an undistorted subgroup of \(\GBG_n(\Graf)\).
\end{corollary}

For \(n\geq2\), the proof starts with Sabalka's halo graph associated to a proper \(\chi(\Lambda)\)-coloring and applies our undistortion result for the corresponding Artin loop braids; see \Cref{Prop:UndistortedRAAGsFromGeneralizedHaloGraphs,Cor:EveryRAAGUndistortedGBG}. The case \(n>\chi(\Lambda)\) is obtained by adjoining stationary pendant vertices, while the case \(n=\chi(\Lambda)=1\) follows since \(\bbA(\Lambda)\) is free.

We now specialize to graph \(2\)-braid groups. Their legal-pair hierarchies have rank at most two---that is, no three unbounded domains are pairwise orthogonal; see \Cref{Thm:IntroRankTwoGBG} below. Combining the rank-two RAAG-embedding criterion of \cite[Theorem~1.3]{OPRAAG} with an explicit analysis of the legal-pair hierarchy, we obtain more explicit restrictions on their RAAG subgroups. We first prove the following uniform obstruction.

\begin{theorem}[Graph \(2\)-braid obstruction, \Cref{Thm:GBG2BipartiteObstruction}]\label{Thm:IntroGBG2Bipartite}
Let \(\Graf\) be a connected finite graph. If \(\Lambda\) is a finite simplicial graph and \(\bbA(\Lambda)<\GBG_2(\Graf)\), then \(\Lambda\) is bipartite.
\end{theorem}

Conversely, if \(\Lambda\) is bipartite, equivalently \(\chi(\Lambda)\leq 2\), then a proper two-coloring of \(\Lambda\) can be used in the generalized halo construction. Hence every bipartite RAAG embeds as an undistorted subgroup of some graph \(2\)-braid group. Thus \Cref{Thm:IntroGBG2Bipartite} is sharp when \(\Graf\) is allowed to vary.

For a fixed connected finite graph \(\Graf\), the induced-subgraph equivalence in \Cref{Thm:IntroRankTwoGBG}, stated below, reduces the embedding problem for \(\bbA(\Lambda)\) to determining whether \(\Lambda\) occurs as an induced subgraph of \(\EOLP\). This reduction does not by itself give a finite decision procedure, since \(\EOLP\) need not be finite. For particular defining graphs, however, the induced-subgraph condition can sometimes be translated into a finite graphical condition on \(\Graf\), without first constructing \(\EOLP\). We carry this out completely for the RAAG defined by the path graph on four vertices.

\begin{theorem}[Graphical \(P_4\)-criterion, \Cref{Thm:GraphicalP4Criterion}]\label{Thm:IntroGraphicalP4Criterion}
Let \(\Graf\) be a connected finite graph, and let \(P_4\) be the four-vertex path. The following are equivalent:
\begin{enumerate}[label=\textup{(\arabic*)}, ref=\arabic*]
\item \(\bbA(P_4)<\GBG_2(\Graf)\);
% \item \(P_4\) embeds as an induced subgraph of \(\EOLP\);
\item \(\Graf\) contains either three pairwise disjoint embedded cycles, or four distinct embedded cycles \(\sfC_a,\sfC_b,\sfC_c,\sfC_d\subseteq\Graf\) satisfying
\[\sfC_a\cap\sfC_b=\sfC_b\cap\sfC_c=\sfC_c\cap\sfC_d=\varnothing,\quad
\sfC_a\cap\sfC_c\neq\varnothing,\quad
\sfC_a\cap\sfC_d\neq\varnothing,\quad
\sfC_b\cap\sfC_d\neq\varnothing.\]
\end{enumerate}
\end{theorem}

Since a finite graph has only finitely many embedded cycles, each of the two conditions in \textup{(2)} can be checked by finite enumeration. Hence whether \(\bbA(P_4)<\GBG_2(\Graf)\) is algorithmically decidable; see \Cref{Cor:AlgorithmicP4Embeddability}. Moreover, the same graphical criterion characterizes the embedding of the RAAG defined by any path on at least four vertices and, equivalently, whether \(\GBG_2(\Graf)\) contains every RAAG defined by a finite forest; see \Cref{Cor:PathForestRAAGsGBG2}.

\subsection{The legal-pair hierarchy}

We now turn to the legal-pair hierarchy underlying the preceding results and explain how its domains are constructed from lifts of certain discrete \(n\)-configuration subspaces of \(\Graf\). This construction does not require \(\Graf\) to be sufficiently subdivided for \(n\). We assume throughout that \(2\leq n<\#V(\Graf)\); the excluded case \(n=\#V(\Graf)\) is degenerate, since \(\UD_n(\Graf)\) is then a point.

A \emph{legal pair} \((\sfL,\bfn)\) of norm \(n\) consists of a possibly disconnected subgraph \(\sfL\subseteq\Graf\), together with a distribution \(\bfn\colon\pi_0(\sfL)\to\NZ\) of the \(n\) particles among its connected components. A non-singleton component is not allowed to be completely occupied, while a singleton component carries exactly one particle. 
Such a legal pair determines a subcomplex \(\UD_{\bfn}(\sfL)\subseteq\UD_n(\Graf)\), called the \emph{legal-pair subcomplex associated to \((\sfL,\bfn)\)}, whose inclusion is a local isometry and whose fundamental group is a product of discrete graph braid groups on the components of \(\sfL\). A \emph{legal-pair lift} is a connected component of the full preimage of a legal-pair subcomplex in \(\widetilde{\UD}_n(\Graf)\).

This direct construction is not supplied by either canonical rich-family choice for a compact special cube complex, since the legal-pair factor system need not agree with either one. Moreover, the proof of Berlyne's description of rich-family factors in terms of graphical subgroups \cite[Theorem~5.1.10]{Ber21} passes from hyperplanes to their underlying edge labels, losing the component data of the remaining particles. The resulting identification with configuration-space lifts does not hold in general.
In fact, the minimal rich family can omit natural legal-pair lifts, whereas the maximal rich family can contain factors not represented by legal pairs; see \Cref{Rem:RichFamilyVersusLegalPairs,Ex:RichFamilyNotLegalPair}. We leave open whether the legal-pair factor system arises from a suitable intermediate rich family.

A second issue concerns basepoints and lifts. In the familiar Salvetti model for a RAAG, the subgroup associated to an induced subgraph is realized by a standard sub-Salvetti complex; see, for example, \cite[\S2]{Char07}. Since the Salvetti complex has a unique vertex, every such subcomplex contains the ambient basepoint and hence has a canonical based inclusion. A legal-pair subcomplex \(\UD_{\bfn}(\sfL)\), however, need not contain the ambient base configuration. Thus the downstairs subcomplex does not by itself determine a hierarchical domain: its lifts may belong to distinct parallelism classes, and identifying its fundamental group with a subgroup requires a choice of base configuration and connecting path.
We therefore fix standard base configurations and connecting paths and use them to choose standard representatives; see \Cref{Subsection:StandardPaths}. This allows us to describe parallelism by a double-coset condition and to give explicit criteria for nesting and orthogonality. We also prove directly that legal-pair lifts are closed under intersections, parallel transport, and gate projections.

\begin{theorem}[The legal-pair hierarchy]\label{Thm:IntroLegalPairFactorSystem}
Let \(\Graf\) be a connected finite graph and let \(2\leq n<\#V(\Graf)\). Then the lifts of the subcomplexes \(\UD_{\bfn}(\sfL)\), where \((\sfL,\bfn)\) ranges over the legal pairs of norm \(n\), form a factor system in \(\widetilde{\UD}_n(\Graf)\).

Moreover, the associated HHG structure obtained after discarding singleton factors---called the legal-pair hierarchy and denoted by \((\GBGdisc_n(\Graf),\frSLP)\)---belongs to \(\Xi\).
\end{theorem}

See \Cref{Prop:LCsubcomplexes_of_UD_n,Thm:GBGHHG,Thm:GBGLegalPairXi}.
Using the standard paths, the domains are represented by cosets of discrete graph braid subgroups. Parallelism of legal-pair lifts, and nesting and orthogonality of the corresponding domains, are described in \Cref{lemma:parallelism,Lem:NestingOrthogonality}.
Of particular importance are the nesting-minimal unbounded domains. Their classification is the main input for the rank formula and, in the two-particle case, for the more refined RAAG-embedding obstructions.

\begin{theorem}[Nesting-minimal unbounded domains]\label{Thm:IntroMinimalDomainsGBG}
For a connected finite graph \(\Graf\) and \(2\leq n<\#V(\Graf)\), let \((\GBGdisc_n(\Graf),\frSLP)\) be the HHG structure given in \Cref{Thm:IntroLegalPairFactorSystem}.
A domain is nesting-minimal unbounded if and only if it admits a representative legal-pair lift whose underlying legal pair has exactly
one non-singleton component and that component is either
\begin{enumerate}
\item an embedded cycle carrying a positive number of particles; or
\item a tripod, that is, a subgraph isomorphic to \(\sfK_{1,3}\), carrying two particles.
\end{enumerate}
Moreover, the hyperbolic space associated to every nesting-minimal unbounded domain is a quasi-line.
\end{theorem}

This combines \Cref{Thm:MinimalUnboundedDomainsGBG,Cor:CoreEqualsExpandedCore}; the former gives the equivalent coset formulation in terms of standard representatives (\Cref{Def:StandardRepresentativeLegalPair}) and constrained parts (\Cref{Def:ConstrainedandFree}).

Since the hyperbolic space associated to every nesting-minimal unbounded domain is a quasi-line, the expansion used in the general definition of the expanded core graph adds no additional vertices in this case; see \Cref{Def:ExpandedCoreGraph}. Thus the vertices of \(\EOLP\) are precisely the nesting-minimal unbounded domains. Together with the explicit orthogonality criterion for legal-pair domains, the preceding classification therefore describes \(\EOLP\) in terms of cycle and tripod supports, particle distributions, and particle motions.
To relate this description to the usual graph braid group, we choose a model of \(\Graf\) that is sufficiently subdivided for \(n\). The following proposition compares the non-isolated parts of the expanded core graphs under further subdivision.

\begin{proposition}[Subdivision stability, \Cref{Prop:SubdivisionStabilityExpandedCore}]
Fix \(n\geq2\), and let \(\Graf\) be a connected finite graph that is sufficiently subdivided for \(n\), and let \(\Graf'\) be a further subdivision of \(\Graf\). Let \(\EO\) and \(\EO'\) denote the expanded core graphs of the legal-pair hierarchies on \(\GBGdisc_n(\Graf)\) and \(\GBGdisc_n(\Graf')\), respectively.

The full subgraphs of \(\EO\) and \(\EO'\) spanned by their non-isolated vertices are isomorphic. Consequently, the nontrivial connected components are independent, up to graph isomorphism, of the chosen sufficiently subdivided model.
\end{proposition}

The proof identifies non-isolated domains by matching their cyclic subgroups under the isomorphism induced by subdivision. This correspondence need not extend to a bijection of the full vertex sets satisfying the same condition on cyclic subgroups; see \Cref{Ex:SubdivisionCreatesIsolatedDomain}.

The free abelian rank formula in \Cref{Thm:IntroRankGBG} now follows from the classification above. Pairwise orthogonal nesting-minimal unbounded domains have pairwise disjoint cycle and tripod supports, with particle costs \(1\) and \(2\), respectively. Conversely, any such disjoint collection whose total particle cost is at most \(n\) determines a pairwise orthogonal collection of nesting-minimal unbounded domains. The general HHG rank bound gives the upper bound, while the corresponding legal-pair subcomplex gives the matching undistorted free abelian subgroup.

The description of \(\EOLP\) also explains the generalized halo construction in \Cref{Prop:UndistortedRAAGsFromGeneralizedHaloGraphs}. Each supporting cycle \(\ell_v\) determines a nesting-minimal unbounded domain \(U_v\), and the corresponding Artin loop braid is strongly fully supported on \(U_v\). The defining intersection conditions give \(U_v\bot U_w\) if and only if \(\{v,w\}\in E(\Lambda)\).
Thus the domains \(U_v\) realize \(\Lambda\) as an induced subgraph of \(\EOLP\). Since the legal-pair hierarchy belongs to \(\Xi\) by \Cref{Thm:IntroLegalPairFactorSystem}, the general RAAG construction then shows that sufficiently large common powers of the Artin loop braids generate an undistorted subgroup isomorphic to \(\bbA(\Lambda)\). In particular, applying this argument to Sabalka's halo graphs gives the universal undistorted realization stated in \Cref{Cor:IntroEveryRAAGUndistortedGBG}.

We now turn to the obstruction side. We determine explicitly, in terms of \(\Graf\) and \(n\), when the legal-pair hierarchy has rank at most two, and hence when the rank-two RAAG-embedding criterion of \cite{OPRAAG} applies.

\begin{theorem}[Rank-two classification for graph braid groups, \Cref{Thm:RankTwoGBG}]\label{Thm:IntroRankTwoGBG}
Let \(\Graf\) be a connected finite graph and let \(n\geq 2\). The legal-pair hierarchy associated to any sufficiently subdivided model of \(\Graf\) has rank at most \(2\) if and only if one of the following holds:
\begin{enumerate}
\item \(n=2\);
\item \(n=3\), and \(\Graf\) contains no three pairwise disjoint embedded cycles;
\item \(n=4\), and whenever \(\ell_1,\ell_2\subseteq\Graf\) are disjoint embedded cycles, the union \(\ell_1\cup\ell_2\) contains every essential vertex of \(\Graf\); 
\item \(n=5\), and every embedded cycle in \(\Graf\), if one exists, contains all but at most one essential vertex of \(\Graf\);
\item \(n\geq6\), and \(\Graf\) contains at most two essential vertices.
\end{enumerate}
In each of these cases, for every finite simplicial graph \(\Lambda\),
\[\bbA(\Lambda)<\GBG_n(\Graf)\qquad\Longleftrightarrow\qquad \Lambda<\EOLP,\] 
where \(\EOLP\) is the expanded core graph of the legal-pair hierarchy on any chosen sufficiently subdivided model of \(\Graf\).
\end{theorem}

In the rank-two cases above, the theorem gives a stronger form of model independence: whether a given finite graph \(\Lambda\) occurs as an induced subgraph of \(\EOLP\) is independent of the chosen sufficiently subdivided model.

When \(n=2\), the particle-permutation homomorphism gives a bipartite coloring of \(\EOLP\); see \Cref{Prop:GBG2CoreGraphBipartite}. Together with the induced-subgraph equivalence above, this yields the graph \(2\)-braid obstruction in \Cref{Thm:IntroGBG2Bipartite}.

To obtain the more precise \(P_4\)-criterion, we describe every non-isolated vertex of \(\EOLP\) by a cycle coordinate \((g,\sfC,\sfC^\perp)\), where \(g\in\GBGdisc_2(\Graf)\), \(\sfC\subseteq\Graf\) is an embedded cycle, and \(\sfC^\perp\) is a component of \(\Graf\setminus\sfC\) containing an embedded cycle; see \Cref{Lem:NonisolatedCycleDomains,Def:CycleCoordinates}. The adjacency criterion for these coordinates translates the condition \(P_4<\EOLP\) into the finite graphical condition on \(\Graf\) in \Cref{Thm:IntroGraphicalP4Criterion}; see \Cref{Lem:CycleCoordinateOrthogonality,Prop:P4SourcesGBG2,Lem:P4CycleExtraction}.

\Cref{Fig:IntroGammaToCore} illustrates the case of three pairwise disjoint cycles in \Cref{Thm:IntroGraphicalP4Criterion}.

% \Cref{Fig:IntroGammaToCore} illustrates the case of three pairwise disjoint cycles in \Cref{Thm:IntroGraphicalP4Criterion}. The element \(\sigma\) is supplied by the last assertion of \Cref{Lem:CycleCoordinateOrthogonality}; its nontrivial particle permutation ensures that the two endpoints of the displayed \(P_4\) are distinct.

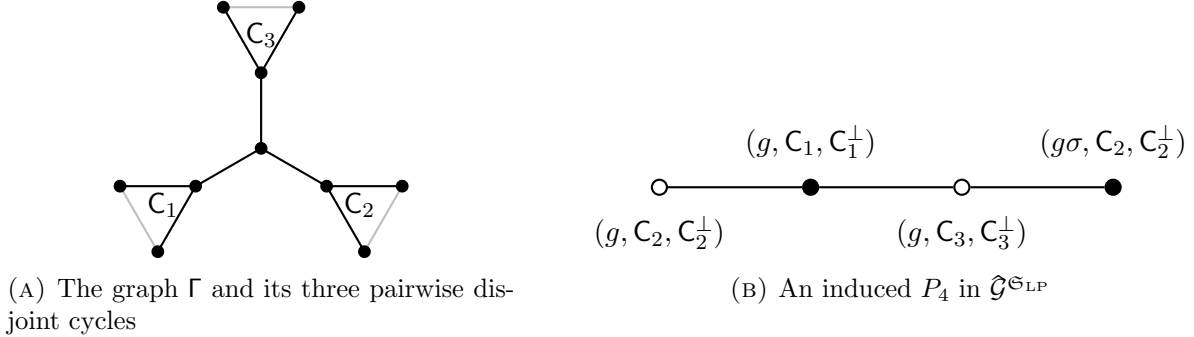
\begin{figure}[ht]
\centering
\subcaptionbox{The graph \(\Graf\) and its three pairwise disjoint cycles%
  \label{Fig:IntroGamma}}[0.42\textwidth]{%
\begin{tikzpicture}[baseline=-.5ex]
\draw[fill, thick] (0,0) circle (2pt);
\foreach \i in {0,120,240} {
\begin{scope}[rotate=\i]
\draw[lightgray, thick] (0,0) ++(90:1) ++(60:1) -- ++(-1,0);
\draw[fill, thick] (0,0) -- (90:1) circle (2pt) -- ++(60:1) circle (2pt) ++(-1,0) circle (2pt) -- ++(-60:1);
\end{scope}
}
\draw (210:1.5) node {\(\sfC_1\)} (-30:1.5) node {\(\sfC_2\)} (90:1.5) node {\(\sfC_3\)};
\end{tikzpicture}
}
\hfill
\subcaptionbox{An induced \(P_4\) in \(\EOLP\)%
  \label{Fig:IntroCoreGraph}}[0.54\textwidth]{%
\begin{tikzpicture}[scale=1.0,
    wht/.style={circle,draw,thick,fill=white,inner sep=2pt},
    blk/.style={circle,draw,thick,fill=black,inner sep=2pt}]
  \node[wht] (a) at (0,0)    {};
  \node[blk] (b) at (2,0) {};
  \node[wht] (c) at (4,0)  {};
  \node[blk] (d) at (6,0) {};
  \draw[thick] (a) -- (b) -- (c) -- (d);
  \node[below=7pt] at (a) {\((g,\sfC_2,\sfC_2^\perp)\)};
  \node[above=7pt] at (b) {\((g,\sfC_1,\sfC_1^\perp)\)};
  \node[below=7pt] at (c) {\((g,\sfC_3,\sfC_3^\perp)\)};
  \node[above=7pt] at (d) {\((g\sigma,\sfC_2,\sfC_2^\perp)\)};
\end{tikzpicture}}
\caption{From \(\Graf\) to \(\EOLP\), for \(n=2\), where \(\sfC_i^{\perp}=\Graf\setminus\sfC_i\).
Edges record orthogonality, and the two vertex classes come from the particle-permutation homomorphism. The four vertices span an induced \(P_4\); here \(\sigma\) is a braid with nontrivial particle permutation, so it changes the class and the two ends are distinct vertices with the same cycle datum.}
\label{Fig:IntroGammaToCore}
\end{figure}

\subsection{Further questions}\label{Subsection:FurtherQuestions}

The legal-pair hierarchy and our RAAG-embedding results suggest questions at two levels. We first consider the hierarchy itself and further applications of HHG theory, and then turn to questions about the subgroup structure of graph braid groups.

The explicit description of the legal-pair hierarchy suggests that other objects associated to an HHG structure may admit direct combinatorial descriptions.

\begin{question}[The HHS boundary]
Can the HHS boundary of the legal-pair hierarchy be described directly in terms of subgraphs of \(\Graf\), particle distributions, and particle motions? How do the resulting boundaries compare under subdivision? Can such a description be used to characterize Morse elements of graph braid groups?
\end{question}

Our application of HHG theory to generalized halo graphs raises a separate quantitative issue. Applying the general RAAG-embedding theorem for HHGs gives undistortion after passing to sufficiently large powers of the associated Artin loop braids. Sabalka's original halo construction, however, proves injectivity for the subgroup generated by their squares.

\begin{question}[Squares in halo embeddings]
Let \(\Lambda\) be a finite simplicial graph equipped with a proper \(n\)-coloring for \(n\ge 2\), and let \(\Graf\) be a \(\Lambda\)-halo graph with respect to this coloring in the sense of Sabalka.
For each \(v\in V(\Lambda)\), let \(\gamma_v\) be its Artin loop braid. Is the subgroup \[\langle\gamma_v^2\mid v\in V(\Lambda)\rangle<\GBG_n(\Graf)\] 
undistorted? More generally, can one effectively determine an exponent \(d\) for which \(\langle\gamma_v^d\mid v\in V(\Lambda)\rangle<\GBG_n(\Graf)\) is undistorted?
\end{question}

We now turn to questions about RAAG subgroups of graph braid groups that can be formulated independently of the chosen HHG structure. The preceding corollary shows that, for every finite simplicial graph \(\Lambda\), the least particle number needed to realize \(\bbA(\Lambda)\) is at most \(\chi(\Lambda)\). This upper bound is sharp when \(\chi(\Lambda)\leq2\): graph \(1\)-braid groups are free, while \Cref{Thm:IntroGBG2Bipartite}, together with the generalized halo construction, characterizes the RAAGs that embed in graph \(2\)-braid groups when the ambient graph is allowed to vary.

\begin{question}[Particle number and chromatic number]
For a finite simplicial graph \(\Lambda\), is the least integer
\(n\) for which \(\bbA(\Lambda)<\GBG_n(\Graf)\) for some connected finite graph \(\Graf\) equal to \(\chi(\Lambda)\)?
More modestly, is there a function \(f\colon\mathbb N\to\mathbb N\) such that
\[\bbA(\Lambda)<\GBG_n(\Graf)\quad\Longrightarrow\quad\chi(\Lambda)\leq f(n)\]
for all \(n\geq2\), all finite simplicial graphs \(\Lambda\), and all connected finite graphs \(\Graf\)?
\end{question}

The RAAG-embedding problem is also natural algorithmically for an arbitrary number of particles. The rank-two RAAG-embedding criterion makes the two-particle case the first one in which it reduces to a concrete induced-subgraph problem.

\begin{question}[Algorithmic RAAG embeddability]
Given a finite simplicial graph \(\Lambda\) and a connected finite graph \(\Graf\), is there an algorithm that determines whether \(\bbA(\Lambda)<\GBG_2(\Graf)\)? Equivalently, can one decide whether \(\Lambda\) occurs as an induced subgraph of the expanded core graph of a sufficiently subdivided model of \(\Graf\)? More generally, does such an algorithm exist for each fixed \(n\geq2\)?
\end{question}

For \(r\geq3\), let \(C_r\) denote the cycle graph on \(r\) vertices. The graphical criterion in \Cref{Thm:IntroGraphicalP4Criterion} answers the algorithmic question when \(\Lambda=P_4\), while \Cref{Thm:IntroGBG2Bipartite} rules out every odd cycle graph. The first remaining case is \(C_4\), for which \(\bbA(C_4)\cong\bbF_2\times\bbF_2\).

\begin{question}[Even-cycle criteria]
For each \(m\geq2\), is there a finite graphical criterion on
\(\Graf\) characterizing when
\(\bbA(C_{2m})<\GBG_2(\Graf)\)?
\end{question}

\subsection{Outline of the paper}

In \Cref{Section:Preliminaries}, we recall the necessary background on factor systems, hierarchical hyperbolicity, RAAG embeddings, and graph braid groups. In \Cref{Section:LegalPairFactorSystem}, we construct the legal-pair factor system. In \Cref{Section:LegalPairHierarchy}, we describe the resulting hierarchy and prove that it belongs to \(\Xi\). In \Cref{Section:ApplicationtoGBG}, we classify the nesting-minimal unbounded domains, establish subdivision stability, compute the rank of the legal-pair hierarchy, and construct undistorted RAAG subgroups using generalized halo graphs.
In \Cref{Section:Graph2BraidGroups}, we specialize to graph \(2\)-braid groups and analyze their RAAG subgroups using cycle coordinates for the expanded core graph.

\subsection*{Acknowledgements}
The first and third authors were supported by Samsung Science and Technology Foundation under Project Number SSTF-BA2022-03.

\section{Preliminaries}\label{Section:Preliminaries}

\begin{convention}
\begin{enumerate}
\item
Unless explicitly stated otherwise, all graphs are assumed to be simplicial, and ambient graphs for graph braid groups are assumed to have at least one edge. For topological graph braid groups, the simpliciality assumption entails no loss of generality, since every finite graph admits a simplicial subdivision with the same underlying topological space. For discrete configuration spaces, however, the chosen simplicial subdivision is part of the cubical model. We adopt the simplicial convention throughout for uniformity with the defining graphs of RAAGs considered in this paper.

\item
Ambient graphs for graph braid groups and their subgraphs are typeset in sans serif, as in \(\Graf\) and \(\sfL\); labeled vertices and edges are similarly
written as \(\sfv\) and \(\sfe\). In the cycle-coordinate and \(P_4\)-configuration arguments, \(\sfC\) and \(\sfC^\perp\) denote distinguished cycles and selected complementary components, while other cycles and tripods are denoted by Greek letters such as \(\ell\) and \(\tau\). Defining graphs of RAAGs are written in ordinary italic type, as in \(\Lambda\).

Configuration spaces are typeset in upright Roman letters, as in \(\UD_n(\Graf)\) and \(\UC_n(\Graf)\). Graph braid groups, RAAGs, and free groups are typeset in blackboard bold, as in \(\GBG_n(\Graf)\), \(\bbA(\Lambda)\), and \(\bbF_r\), while generic groups and subgroups are denoted by italic letters such as \(G\) and \(H\).

\item
For \(r,s\geq1\), the standard symbols \(K_r\) and \(K_{r,s}\) denote the complete graph on \(r\) vertices and the complete bipartite graph with parts of sizes \(r\) and \(s\), respectively. We write \(\overline K_r\) for the discrete graph on \(r\) vertices. When one of these graphs occurs as an ambient or supporting graph for a graph configuration space, we use the sans-serif versions \(\sfK_r\) and \(\sfK_{r,s}\). 
\item For cubical maps between nonpositively curved cube complexes, we use the standard full-link criterion: a cubical map is a local isometry if it is locally injective and the induced map on each vertex link is injective with image a full subcomplex.
\end{enumerate}
\end{convention}

\subsection{\texorpdfstring{\(\CAT(0)\)}{CAT(0)} cube complexes and factor systems}\label{Subsection:FactorSystems}

We follow the standard cubical conventions of \cite[\S2]{BHS17I} and the references therein. In particular, we use the notions of hyperplanes, carriers, combinatorial hyperplanes, convex subcomplexes, gates, and parallelism. We recall only the definitions and properties needed for the factor-system construction below.

We will repeatedly use the following standard consequence of covering-space theory and \cite[Proposition~II.4.14]{BH}.

\begin{lemma}[Lifts of locally convex subcomplexes]\label{Lem:IntersectionOfLifts}
Let \(X\) be a compact nonpositively curved cube complex, and let \(p\colon\widetilde X\rightarrow X\) be its universal covering map. If \(A\subseteq X\) is a connected locally convex subcomplex, i.e., \(A\) is connected and the embedding \(A\hookrightarrow X\) is a local isometry, then every lift of \(A\) is convex in \(\widetilde X\).

Moreover, if \(A,B\subseteq X\) are connected locally convex subcomplexes and lifts \(\widetilde A,\widetilde B\subseteq \widetilde X\) have nonempty intersection, then \(\widetilde A\cap\widetilde B\) is a lift of a connected component of \(A\cap B\).
\end{lemma}

Let \(X\) be a \(\CAT(0)\) cube complex. For a nonempty convex subcomplex \(K\subseteq X\) and a vertex \(x\in X^{(0)}\), there is a unique closest vertex \(\frg_K(x)\in K\), called the \emph{gate of \(x\) in \(K\)}.
For a subcomplex \(Y\subseteq X\), the \emph{gate projection} of \(Y\) to \(K\), denoted by \(\frg_K(Y)\), is the subcomplex spanned by the vertices \(\frg_K(y)\), where \(y\in Y^{(0)}\); if \(Y\) is convex, then \(\frg_K(Y)\) is also convex. 
Two convex subcomplexes \(Y_1,Y_2\subseteq X\) are called \emph{parallel}, written \(Y_1\parallel Y_2\), if they are crossed by the same set of hyperplanes. By \cite[Lemma~2.4]{BHS17I}, parallel convex subcomplexes occur as the boundary fibers of a cubical product region.

\begin{definition}[Factor system]\label{Def:factor system}
Let \(X\) be a \(\CAT(0)\) cube complex. A collection \(\frF\) of subcomplexes of \(X\) is called a \emph{factor system} if it satisfies the following conditions:
\begin{enumerate}
\item\label{Item:Elements_of_FS}
Every element of \(\frF\) is a nonempty convex subcomplex of \(X\), and \(X\in\frF\).
\item\label{Item:CombinatorialHyperplane}
Every nontrivial (that is, non-singleton) convex subcomplex parallel to a combinatorial hyperplane of \(X\) belongs to \(\frF\).
\item\label{Item:Delta}
There exists \(\Delta\geq 1\) such that every vertex of \(X\) belongs to at most \(\Delta\) elements of \(\frF\).
\item\label{Item:Projections}
There exists \(\xi\geq 0\) such that, for every \(F,F'\in\frF\), either \(\frg_F(F')\in\frF\), or \(\diam\bigl(\frg_F(F')\bigr)<\xi\).
\end{enumerate}
\end{definition}

For any cube complex \(Y\), its \emph{crossing graph} \(\mathcal C_0(Y)\) and \emph{contact graph} \(\mathcal C(Y)\) both have the hyperplanes of \(Y\) as their vertices. Two hyperplanes are adjacent in \(\mathcal C_0(Y)\) if they cross, and they are adjacent in \(\mathcal C(Y)\) if their carriers intersect. Thus \(\mathcal C_0(Y)\) is a spanning subgraph of \(\mathcal C(Y)\).

\begin{definition}[Factored contact graph]
Let \(X\) be a \(\CAT(0)\) cube complex with a factor system \(\frF\), and let \(F\in\frF\). The \emph{factored contact graph} \(\Chat F\) is obtained from the contact graph of \(F\) as follows. For each parallelism class of \(F'\in\frF\) which is parallel to a proper subcomplex of \(F\) and which is either parallel to a combinatorial hyperplane or has diameter at least \(\xi\), one adds a cone vertex joined to every vertex corresponding to a hyperplane crossing \(F'\).
\end{definition}

\begin{remark}
If \Cref{Def:factor system}\eqref{Item:Projections} holds for some \(\xi_0\geq0\), then it also holds for every \(\xi\geq\xi_0\). Hence, when passing to the associated HHG structure, we enlarge \(\xi\), if necessary, so that \(\xi\geq1\), and, when \(X\) is not a point, discard singleton factors; compare \cite[Remark~8.15]{BHS17I} and \cite[Definition~10.10]{DHS17}. The resulting collection is again a factor system: the first three axioms are preserved, and a gate projection removed in this process is a singleton and hence has diameter \(0<\xi\).

We always use this reduced factor system in the associated HHS construction. This removal ensures that singleton factors do not determine hierarchical domains or contribute cone vertices to factored contact graphs.
\end{remark}

We will use both the hyperbolicity of \(\Chat F\) and the fact that every factor coned off in its construction has bounded image. A factor system on a \(\CAT(0)\) cube complex determines an HHS structure; if a group \(G\) acts properly and cocompactly while preserving the factor system, this gives an HHG structure on \(G\) \cite{BHS17I,HS20}.
In particular, the universal cover of a compact special cube complex admits a deck-invariant factor system \cite[Proposition~B and Remark~13.2]{BHS17I}. Hence, if \(X\) is a compact special cube complex, then \(\pi_1 X\) admits an HHG structure. We stress that the legal-pair factor system constructed below is a specific factor system adapted to graph braid groups; its factor-system axioms and the additional structural properties needed below will be verified directly.

\subsection{Hierarchical terminology and the class \texorpdfstring{\(\Xi\)}{Xi}}\label{Subsection:HHG}

All hierarchical structures used in this paper arise from the factor-system construction recalled above. We therefore do not review the full HHS axioms, and record only the terminology and structural results needed below. We refer to \cite{BHS19II,DHS17} for hierarchically hyperbolic spaces and groups.

Let \((G,\frS)\) be a hierarchically hyperbolic group (HHG). A domain \(U\in\frS\) is called \emph{unbounded} if \(\diam(\mathcal C U)=\infty\).
For \(g\in G\), its \emph{bigset} is
\[\B(g)=\left\{ U\in\frS \mid \diam_{\mathcal C U}\bigl(\pi_U(\langle g\rangle x)\bigr)=\infty \right\},\]
where \(x\in G\) is any basepoint. It follows immediately from the definition that 
\[\B(g^k)=\B(g)\qquad\text{for every nonzero integer }k.\]
The elements of \(\B(g)\) are called the \emph{active domains} of \(g\). We say that \(g\) is \emph{elliptic} if its orbits in \(G\) are bounded, and \emph{axial} if the orbit map \(n\mapsto g^n x\) is a quasi-isometric embedding for some \(x\in G\).

The results of \cite[Proposition~6.4 and Lemma~6.7]{DHS17}, \cite[Theorem~3.1]{DHS20}, and \cite[Lemma~1.21]{AB23} give the axial--elliptic dichotomy and the corresponding behavior on active domains. The precise uniform formulation below was recorded in \cite[Proposition~2.19]{OP25} and recalled in \cite[Proposition~2.7]{OPRAAG}.

\begin{proposition}[Axial--elliptic dichotomy]\label{Prop:AxialEllipticDichotomy}
An element \(g\in G\) is elliptic if and only if \(\B(g)=\varnothing\), if and only if \(g\) has finite order.

If \(g\) has infinite order, then \(g\) is axial. Moreover,
\begin{itemize}
\item
\(g\) permutes \(\B(g)\), which is a nonempty finite pairwise orthogonal subset of \(\frS\);
\item
there exists \(M=M(\frS)>0\) such that, for every \(U\in\B(g)\), the element \(g^M\) fixes \(U\) and acts loxodromically on \(\mathcal C U\).
\end{itemize}
\end{proposition}

We next recall the product-region terminology used in the definition of the class \(\Xi\). Let \((X,\frS)\) be an HHS. For each domain \(U\in\frS\), there is a coarsely defined product map
\[\phi_U\colon\mathbf F_U\times\mathbf E_U\longrightarrow X\]
whose image is the standard product region \(\mathbf P_U\), where \((\mathbf F_U,\frS_U)\) and \((\mathbf E_U,\frS_U^\perp)\) are the standard nesting and orthogonality factors associated to \(U\), respectively; see \cite[\S5]{BHS19II}.

When the HHS structure arises from a factor system on a \(\CAT(0)\) cube complex \(X\), we use the cubical model of \cite[Remark~13.5]{BHS17I}. Namely, after choosing a factor \(F\in U\), we identify the standard nesting factor \(\mathbf F_U\) with \(F\). The standard orthogonality factor \(\mathbf E_U\) is modeled by the convex cube complex parametrizing the parallel copies of \(F\), and the product map is a cubical isometric embedding with convex image. In this model, nesting can be witnessed by parallel representatives \(F_U\in U\) and \(F_V\in V\) with \(F_U\subseteq F_V\), while orthogonality can be witnessed by representatives occurring as distinct factors of a common cubical product region; see \cite[\S8.3]{BHS17I}. Thus \(\mathbf P_U\) is realized as a convex cubical product subcomplex of \(X\).

We use the metric orthogonal stabilizers introduced in \cite[Definition~3.1]{AB23} and the strengthened \(\mathbf F_U\) stabilizers property introduced in \cite{OPRAAG}.

\begin{definition}[Metric orthogonal stabilizers and strong support]\label{Def:OurGUandFU}
For \(U\in\frS\), let \(G_U\leq\Stab_G(U)\) be the subgroup consisting of the elements which stabilize every \(\mathbf F_U\)-fiber \(\phi_U(\mathbf F_U\times\{e\})\) for \(e\in\mathbf E_U\).
Equivalently, \(G_U\) is the kernel of the induced action of \(\Stab_G(U)\) on \(\mathbf E_U\), viewed as a metric space.
We call \(G_U\) the \emph{metric orthogonal stabilizer} of \(U\). An axial element \(g\in G\) is said to be \emph{strongly fully supported on} an unbounded domain \(U\in\frS\) if \(g\in G_U\) and \(U\in\B(g)\).
\end{definition}

By \cite[Lemma~3.3]{OPRAAG}, if \(g\in G_U\) and \(V\bot U\) is unbounded, then \(g\) fixes \(V\) and acts trivially on \(\mathcal C V\). In particular, if \(g\) is strongly fully supported on \(U\), then \(\B(g)=\{U\}\).

The class needed for the obstruction results is the following. Notice that clean containers are not part of its definition; clean containers define the subclass \(\Xi_{\mathrm{cc}}\) in \cite{OPRAAG}.

\begin{definition}[The class \(\Xi\)]\label{Def:Xiclass}
We say that an HHG \((G,\frS)\) satisfies the \emph{\(\mathbf F_U\) stabilizers property} if, for every \(U\in\frS\), the natural restricted action \(G_U\curvearrowright(\mathbf F_U,\frS_U)\) makes \(G_U\) an HHG with underlying HHS \((\mathbf F_U,\frS_U)\).

The class \(\Xi\) of HHGs introduced in \cite[Definition~3.7]{OPRAAG} is the collection of HHGs \((G,\frS)\) satisfying the \(\mathbf F_U\) stabilizers property and the following two additional properties:
\begin{enumerate}
\item \emph{Orthogonal decomposition property.} For every infinite-order element \(g\in G\), writing \(\B(g)=\{U_1,\dots,U_k\}\), there exist \(N>0\) and axial elements \(h_1,\dots,h_k\in G\) such that
\[g^N=h_1\cdots h_k,\]
where \(h_i\) is strongly fully supported on \(U_i\) for every \(i\).

\item \emph{Commutative property.} Any two axial elements strongly fully supported on orthogonal unbounded domains commute.
\end{enumerate}
\end{definition}

We refer to \cite[\S3]{OPRAAG} for the development of these properties. In particular, the one-domain HHG structures on hyperbolic groups and the rich-family HHG structures on compact special groups belong to \(\Xi\); see \cite[Proposition~3.9]{OPRAAG}. The legal-pair hierarchy constructed below is not assumed to be one of these rich-family structures, so its membership in \(\Xi\) will be verified separately in \Cref{Thm:GBGLegalPairXi}.

\subsection{RAAG embeddings and the expanded core graph}
\label{Subsection:RAAGEmbeddings}

A \emph{right-angled Artin group} is the group associated to a finite graph \(\Lambda\), defined by
\[\bbA(\Lambda)=\langle v\in V(\Lambda)\mid [v,w]=1\text{ whenever }\{v,w\}\in E(\Lambda)\rangle.\]
The graph \(\Lambda\) is called the \emph{defining graph} of \(\bbA(\Lambda)\). 

We first recall geometric irredundancy and the construction theorem for HHGs in the class \(\Xi\) used below.

Let \((G,\frS)\) be an HHG. A finite collection of axial elements \(g_1,\dots,g_m\in G\) is \emph{geometrically irredundant} if, for any distinct \(g_i,g_j\), either \(\B(g_i)\neq\B(g_j)\), or there exists \(U\in\B(g_i)\cap\B(g_j)\) such that the limit sets of \(g_i\) and \(g_j\) in \(\partial\mathcal C U\) are disjoint.

\begin{proposition}[Undistorted RAAGs in \(\Xi\), {\cite[Proposition~3.13]{OPRAAG}}]\label{Prop:RAAGEmbeddingTheorem}
Let \((G,\frS)\in\Xi\), and let \(g_1,\dots,g_m\in G\) be a geometrically irredundant collection of axial elements. Suppose that each \(g_i\) is strongly fully supported on an unbounded domain \(U_i\in\frS\). Then there exists \(D>0\) such that, for every integer \(d\geq D\), the subgroup \(\langle g_1^d,\dots,g_m^d\rangle<G\) is an undistorted subgroup isomorphic to \(\bbA(\Lambda)\), where \(\Lambda\) has vertex set \(\{v_1,\dots,v_m\}\), with \(v_i\) corresponding to \(g_i\), and its edges are determined by
\[\{v_i,v_j\}\in E(\Lambda)\quad\Longleftrightarrow\quad U_i\bot U_j \qquad\text{for all distinct \(i,j\).}\]
\end{proposition}

The expanded core graph packages the domains used in this construction.

\begin{definition}[Expanded core graph]\label{Def:ExpandedCoreGraph}
Let \(\frS^\infty=
\{U\in\frS\mid \diam(\mathcal C U)=\infty\}\). An element \(U\in\frS^\infty\) is called \emph{nesting-minimal unbounded} if it is minimal in the poset \((\frS^\infty,\sqsubseteq)\); equivalently, whenever \(U'\sqsubseteq U\) and \(\mathcal C U'\) is unbounded, one has \(U'=U\).

Let \(\frG\subseteq\frS\) be the set of nesting-minimal unbounded domains. The \emph{core graph} \(\mathcal G^{\frS}\) has vertex set \(\frG\), with two vertices adjacent precisely when the corresponding domains are orthogonal.

For each \(U\in\frG\), let \(V_U\) consist of one vertex if \(\mathcal C U\) is a quasi-line, and countably many vertices otherwise. The \emph{expanded core graph} \(\EOW\), or simply \(\EO\) when \(\frS\) is clear, is defined by
\[V(\EOW)=\bigsqcup_{U\in\frG}V_U,\]
where every vertex of \(V_U\) is adjacent to every vertex of \(V_{U'}\) if and only if \(U\bot U'\).
\end{definition}

By \cite[Lemma~4.4]{OPRAAG} and \Cref{Prop:RAAGEmbeddingTheorem}, every finite induced subgraph \(\Lambda<\EO\) determines an undistorted subgroup of \(G\) isomorphic to \(\bbA(\Lambda)\).

Recall that the \emph{rank} of an HHS is the maximal cardinality of a collection of pairwise orthogonal unbounded domains. The obstruction theory of \cite{OPRAAG} gives a complete RAAG-embedding criterion when the rank is at most two.

\begin{theorem}[Rank-two RAAG-embedding criterion, {\cite[Theorem~1.3]{OPRAAG}}]\label{Thm:RankTwoCriterionHHG}
Let \((G,\frS)\in\Xi\) have rank at most \(2\). For every finite graph \(\Lambda\), \(\bbA(\Lambda)<G\) if and only if \(\Lambda<\EO\).
\end{theorem}

\subsection{Graph braid groups}\label{Sec:graphbraidgroup}
We recall configuration spaces of graphs, their discrete counterparts, and the cubical structures and hyperplanes of the latter that will be used in \Cref{Section:LegalPairFactorSystem}.

\begin{definition}[Configuration spaces and graph braid groups]\label{Def:UC_n}
Let \(\Graf\) be a graph and let \(n\geq0\). The \emph{ordered} and \emph{unordered \(n\)-configuration spaces} of \(\Graf\) are, respectively, 
\[\mathrm C_n(\Graf)=\{(x_1,\dots,x_n)\in\Graf^n\mid x_i\neq x_j\text{ for }i\neq j\}\qquad\text{and}\qquad\UC_n(\Graf)=\mathrm C_n(\Graf)/\mathfrak S_n.\]

Choose an ordered configuration \(\bar{\mathbf x}\in\mathrm C_n(\Graf)\), and let \(\mathbf x\in\UC_n(\Graf)\) be its image. The corresponding
\emph{pure graph \(n\)-braid group} and \emph{graph \(n\)-braid group} are, respectively,
\[\bbP_n(\Graf,\bar{\mathbf x})=\pi_1(\mathrm C_n(\Graf),\bar{\mathbf x})
\qquad\text{and}\qquad \GBG_n(\Graf,\mathbf x)=\pi_1(\UC_n(\Graf),\mathbf x).\]
\end{definition}

Note that the previous two groups only depend on the topological type of the ambient graph. The following two groups do depend on the subdivision of the graph.

\begin{definition}[Discrete configuration spaces and discrete graph braid groups]\label{Def:UD_n}
Let \(\Graf\) be a graph and let \(n\geq0\). The \emph{ordered} and \emph{unordered discrete \(n\)-configuration spaces} of \(\Graf\) are
\begin{align*}
\D_n(\Graf)&=
\{(\sigma_1,\dots,\sigma_n)\in\Graf^n\mid \sigma_i\cap\sigma_j=\varnothing \text{ for } i\neq j\}\quad\text{and}\\
\UD_n(\Graf)&=
\{\{\sigma_1,\dots,\sigma_n\}\subset\Graf\mid \sigma_i\cap\sigma_j=\varnothing \text{ for } i\neq j\}
\cong \D_n(\Graf)/\mathfrak S_n,
\end{align*}
where each \(\sigma_i\) is either a vertex or an edge of \(\Graf\), regarded as a closed cell. When basepoint components are understood, their universal covers are denoted by \(\widetilde{\D}_n(\Graf)\) and \(\widetilde{\UD}_n(\Graf)\), respectively.

Choose a collection \(\mathbf x\) of \(n\) distinct vertices of \(\Graf\), and let \(\bar{\mathbf x}\) be an ordering of \(\mathbf x\). The corresponding discrete graph braid groups are
\[\PGBGdisc_n(\Graf,\bar{\mathbf x})=\pi_1(\D_n(\Graf),\bar{\mathbf x})\qquad\text{and}\qquad\GBGdisc_n(\Graf,\mathbf x)=\pi_1(\UD_n(\Graf),\mathbf x).\]
\end{definition}

The free action of \(\mathfrak S_n\) on \(\D_n(\Graf)\) by coordinate permutations makes the quotient map \[q\colon\D_n(\Graf)\longrightarrow\UD_n(\Graf)\] 
a covering map. With the basepoints chosen above, \(q_*\) identifies \(\PGBGdisc_n(\Graf,\bar{\mathbf x})\) with the kernel of the particle-permutation homomorphism \(\GBGdisc_n(\Graf,\mathbf x)\rightarrow\mathfrak S_n\). Likewise, the quotient map \(\mathrm C_n(\Graf)\rightarrow\UC_n(\Graf)\) is a covering map, and the corresponding statement holds for \(\bbP_n(\Graf,\bar{\mathbf x})<\GBG_n(\Graf,\mathbf x)\).

When the basepoints are understood, we suppress them from all four group notations above.

\begin{remark}\label{Rem:GBGDegenerate}
When \(n=0\), all four configuration spaces above consist of the empty configuration and are therefore one-point spaces. When \(n=1\), all four are naturally isomorphic to \(\Graf\). Consequently, all four braid groups are trivial for \(n=0\) and are isomorphic to \(\pi_1(\Graf)\) for \(n=1\).
Accordingly, unless otherwise stated, we assume that \(n\geq2\) for ambient configuration spaces and graph braid groups.
\end{remark}

Both \(\D_n(\Graf)\) and \(\UD_n(\Graf)\) have natural cube-complex structures. A \(k\)-cube of \(\UD_n(\Graf)\) is represented by a collection \(\{\sfe_1,\dots,\sfe_k,\sfv_1,\dots,\sfv_{n-k}\}\), where the \(\sfe_i\) are edges of \(\Graf\), the \(\sfv_j\) are vertices of \(\Graf\), and all these closed cells are pairwise disjoint. Its faces are obtained by replacing some of the \(\sfe_i\) by their endpoints. The cube structure on \(\D_n(\Graf)\) is defined analogously, with the ordering of the cells retained.
In particular, a vertex of \(\UD_n(\Graf)\), which we also call a \emph{vertex configuration}, is an unordered collection of \(n\) distinct vertices of \(\Graf\).

\begin{remark}\label{Rem:ConfigurationSpaceDegenerate}
Since every cube of \(\UD_n(\Graf)\) has a vertex face, \(\UD_n(\Graf)\) is nonempty if and only if it contains a vertex configuration, equivalently, if and only if \(n\leq\#V(\Graf)\). For \(n\geq1\), it is a point if and only if
\(n=\#V(\Graf)\).

If \(\Graf\) is connected and \(1\leq n<\#V(\Graf)\), then \(\UD_n(\Graf)\) is connected. Indeed, using at least one unoccupied vertex, one can move particles one at a time along a spanning tree of \(\Graf\) to pass between any
two vertex configurations. 
Together with
\Cref{Rem:GBGDegenerate}, this shows that the nonempty, non-point case relevant below is precisely \(2\leq n<\#V(\Graf)\).
\end{remark}

Write \(\pi_0(\Graf)\) for the set of connected components of \(\Graf\), and let \(\bfn:\pi_0(\Graf)\to\NZ\) be a function. We define
\begin{equation}\label{Eq:ProductofGBG}
\UD_{\bfn}(\Graf)=\prod_{\Graf_i\in \pi_0(\Graf)}\UD_{\bfn(\Graf_i)}(\Graf_i),    
\end{equation}
whenever each factor on the right-hand side is nonempty, and write \(\|\bfn\|=\sum_{\Graf_i\in \pi_0(\Graf)}\bfn(\Graf_i)\). 
Equivalently, \(\UD_{\bfn}(\Graf)\) is the subcomplex of \(\UD_{\|\bfn\|}(\Graf)\) consisting of configurations containing exactly \(\bfn(\Graf_i)\) particles in \(\Graf_i\) for every \(\Graf_i\in \pi_0(\Graf)\).
If a vertex configuration \(\mathbf x\) contains exactly \(\bfn(\Graf_i)\) vertices in \(\Graf_i\), we say that \(\mathbf x\) \emph{realizes} \(\bfn\). 
After choosing such a basepoint, we write 
\[\GBGdisc_{\bfn}(\Graf)=\pi_1(\UD_{\bfn}(\Graf),\mathbf x).\]
By the product decomposition in \eqref{Eq:ProductofGBG}, one has 
\[\GBGdisc_{\bfn}(\Graf)\cong\prod_{\Graf_i\in \pi_0(\Graf)}\GBGdisc_{\bfn(\Graf_i)}(\Graf_i).\]

\begin{definition}\label{Def:SufficientlySubdivided}
Let \(\Graf\) be a finite connected graph, and let \(n\geq2\).
We say that \(\Graf\) is \emph{sufficiently subdivided for \(n\)} if the following conditions hold:
\begin{enumerate}
\item if \(\Graf\) is a path, then it has at least \(n+1\) vertices;
\item otherwise,
\begin{enumerate}
\item every path between two distinct non-bivalent vertices contains at least \(n-1\) edges;
\item every cycle contains at least \(n+1\) edges.
\end{enumerate}
\end{enumerate}
\end{definition}

In the path case, this convention is slightly stronger than the usual sufficient-subdivision condition: we require at least \(n+1\) vertices, rather than \(n\), thereby excluding the case in which \(\UD_n(\Graf)\) is a point. For a non-path graph, the remaining conditions already imply \(n<\#V(\Graf)\). Thus every graph sufficiently subdivided for \(n\) satisfies \(n<\#V(\Graf)\), and consequently \(\UD_n(\Graf)\) is nonempty and not a point; see \Cref{Rem:ConfigurationSpaceDegenerate}.

The following result allows us to work with discrete configuration spaces and then pass to the usual graph braid groups using sufficiently subdivided models.

\begin{theorem}[Discrete configuration-space models, \cite{Abr00,KKP12,PS12}]\label{Thm:DiscreteConfigurationRetraction}
Let \(\Graf\) be a connected finite graph that is sufficiently subdivided for \(n\geq2\). Then the inclusions
\[\D_n(\Graf)\hookrightarrow\mathrm C_n(\Graf)\qquad\text{and}\qquad\UD_n(\Graf)\hookrightarrow\UC_n(\Graf)\]
exhibit the discrete configuration spaces as strong deformation retracts of the corresponding continuous configuration spaces. 
Consequently,
\[\PGBGdisc_n(\Graf)\cong\bbP_n(\Graf)\qquad\text{and}\qquad\GBGdisc_n(\Graf)\cong\GBG_n(\Graf).\]
\end{theorem}

We use these isomorphisms together with changes of basepoint, without fixing particular deformation retractions.

Independently of the subdivision condition, discrete configuration spaces of finite graphs are compact special cube complexes.

\begin{theorem}
Suppose that \(\Graf\) is a finite graph and that \(\bfn\colon\pi_0(\Graf)\rightarrow\NZ\) is a particle distribution for which \(\UD_{\bfn}(\Graf)\) is nonempty. Then \(\UD_{\bfn}(\Graf)\) is a compact special cube complex.
Consequently, \(\GBGdisc_{\bfn}(\Graf)=\pi_1(\UD_{\bfn}(\Graf))\) admits an HHG structure arising from a factor system on its universal cover.
\end{theorem}
\begin{proof}
By \eqref{Eq:ProductofGBG} and the closure of compact special cube complexes under finite products, it suffices to consider the case in which \(\Graf\) is connected. Set \(n=\|\bfn\|\). The cases \(n=0\), \(n=1\) and \(n=\#V(\Graf)\) are immediate from \Cref{Rem:GBGDegenerate,Rem:ConfigurationSpaceDegenerate}. We may therefore assume that \(2\leq n<\#V(\Graf)\).
Then \(\UD_{\bfn}(\Graf)=\UD_n(\Graf)\) is connected by \Cref{Rem:ConfigurationSpaceDegenerate}. Since \(\Graf\) is finite, \(\UD_n(\Graf)\) has finitely many cubes and is therefore compact.

Abrams proved that \(\UD_n(\Graf)\) is nonpositively curved \cite{Abr00}, and Crisp--Wiest constructed a cubical local isometry from \(\UD_n(\Graf)\) to a Salvetti complex \cite[Theorem~2]{CW}. Thus \(\UD_n(\Graf)\) is special in the sense of Haglund--Wise \cite{HW08}. Genevois later gave a direct verification of specialness using a special coloring; see \cite[Propositions~3.4 and~3.7 and Remark~3.8]{Gen21GBG}.
Finally, by \cite[Proposition~B and Remark~13.2]{BHS17I}, the universal cover of a compact special cube complex admits an invariant factor system, and the associated hierarchical structure makes its fundamental group an HHG.
\end{proof}

Using the cubical terminology recalled in \Cref{Subsection:FactorSystems}, we next describe the hyperplanes of \(\UD_n(\Graf)\) directly in terms of discrete configurations.
For subgraphs \(\Graf''\subseteq\Graf'\subseteq\Graf\), we write \(\Graf'\setminus\Graf''\) for the subgraph of \(\Graf'\) consisting of all closed cells disjoint from \(\Graf''\). In particular, if \(\Graf'=\Graf\) and \(\Graf''\) is an edge \(\sfe\), then the endpoints of \(\sfe\), as well as every edge incident to either endpoint, are omitted from \(\Graf\setminus\sfe\).

Let \(\bfE=\{\sfe,\sfv_1,\dots,\sfv_{n-1}\}\) be an edge of \(\UD_n(\Graf)\). For each \(\sfL_i\in\pi_0(\Graf\setminus\sfe)\), define
\(\bfn_{\bfE}(\sfL_i)=\#(\bfE\cap \sfL_i)\). Let
\[\Graf_{\bfE}=\bigsqcup_{\substack{\sfL_i\in\pi_0(\Graf\setminus\sfe)\\ \bfn_{\bfE}(\sfL_i)>0}} \sfL_i,\]
and regard \(\bfn_{\bfE}\) as a positive particle distribution on the components of \(\Graf_{\bfE}\).

We write \(H_{\bfE}\) for the hyperplane of \(\UD_n(\Graf)\) dual to \(\bfE\), and \(\mathcal N(H_{\bfE})\) for its carrier, that is, the union of the cubes of \(\UD_n(\Graf)\) crossed by \(H_{\bfE}\).

\begin{lemma}[Hyperplanes in discrete configuration spaces]\label{Lem:Hyperplane_in_UD_n}
Let \(\bfE=\{\sfe,\sfv_1,\dots,\sfv_{n-1}\}\) be an edge of \(\UD_n(\Graf)\).
Then the following are equivalent:
\begin{itemize}
\item an edge \(\bfE'=\{\sfe',\sfv_1',\dots,\sfv_{n-1}'\}\) of \(\UD_n(\Graf)\) is dual to \(H_{\bfE}\);
\item \(\sfe=\sfe'\) and \(\bfn_{\bfE}\equiv\bfn_{\bfE'}\).
\end{itemize}

Consequently, the carrier of \(H_{\bfE}\) has a cubical product decomposition \[\mathcal N(H_{\bfE})\cong \UD_{\bfn_{\bfE}}(\Graf_{\bfE})\times[0,1],\] where the interval factor records the particle moving along \(\sfe\).
Under this decomposition, \(H_{\bfE}\cong \UD_{\bfn_{\bfE}}(\Graf_{\bfE})\times\left\{1/2\right\}\).

Moreover, every hyperplane \(\widetilde H\) of \(\widetilde{\UD}_n(\Graf)\) is a lift of a hyperplane of this form. If \(\widetilde H\) lies over \(H_{\bfE}\), then \(\mathcal N(\widetilde H)\) is a connected component of the inverse image of \(\mathcal N(H_{\bfE})\), and \(\mathcal N(\widetilde H)\cong\widetilde H\times[0,1]\).
\end{lemma}

\begin{proof}
Two edges of a cube complex are dual to the same hyperplane if and only if they are related by a sequence of elementary parallelisms, where two edges are elementarily parallel when they are opposite sides of a square.

Consider a square of \(\UD_n(\Graf)\) having \(\bfE\) as one of its sides.
It is represented by a configuration containing the edge \(\sfe\), another edge \(\sfe'\) disjoint from \(\sfe\), and \(n-2\) fixed vertices. The two opposite sides parallel to \(\bfE\) are obtained by retaining the \(\sfe\)-coordinate and replacing the particle at one endpoint of \(\sfe'\) by the other endpoint. Thus, under an elementary parallelism, the moving edge \(\sfe\) remains unchanged, and each of the other particles remains in the same connected component of \(\Graf\setminus\sfe\).
It follows that if \(\bfE'=\{\sfe',\sfv_1',\dots,\sfv_{n-1}'\}\) is dual to \(H_{\bfE}\), then \(\sfe=\sfe'\), and the numbers of stationary particles in the components of \(\Graf\setminus\sfe\) agree with those of \(\bfE\).

Conversely, suppose that \(\sfe=\sfe'\) and that these particle counts agree. For each component \(\sfL_i\) of \(\Graf\setminus\sfe\), the two collections \(\bfE\cap\sfL_i\) and \(\bfE'\cap\sfL_i\) are vertices of \(\UD_{\bfn_{\bfE}(\sfL_i)}(\sfL_i)\). 
Since \(\sfL_i\) is connected, so is the configuration space \(\UD_{\bfn_{\bfE}(\sfL_i)}(\sfL_i)\). Hence the first configuration can be transformed into the second by a sequence of single-particle moves inside \(\sfL_i\).
Performing these moves successively over all components of \(\Graf\setminus\sfe\), while keeping the \(\sfe\)-coordinate fixed, gives a sequence of squares in \(\UD_n(\Graf)\) whose consecutive \(\sfe\)-edges are opposite sides. Therefore \(\bfE'\) is dual to \(H_{\bfE}\).

The preceding characterization shows that the edges dual to \(H_{\bfE}\) are precisely those obtained by allowing the stationary particles to vary in the components of \(\Graf_{\bfE}\) according to the distribution \(\bfn_{\bfE}\), while retaining the particle moving along \(\sfe\). 
The cubes crossed by \(H_{\bfE}\) are obtained by additionally allowing disjoint stationary particles to move simultaneously. Their union is therefore exactly
\[\UD_{\bfn_{\bfE}}(\Graf_{\bfE})\times\sfe\cong\UD_{\bfn_{\bfE}}(\Graf_{\bfE})\times[0,1].\]
The hyperplane itself is the midcube \(\UD_{\bfn_{\bfE}}(\Graf_{\bfE})\times\left\{1/2\right\}\).

Finally, a covering map of cube complexes preserves cubes, parallelism of edges, and hyperplanes. Hence every hyperplane of \(\widetilde{\UD}_n(\Graf)\) lies over a hyperplane \(H_{\bfE}\) of \(\UD_n(\Graf)\), and its carrier is a connected component of the inverse image of \(\mathcal N(H_{\bfE})\). 
Since \([0,1]\) is simply connected, each connected component of the inverse image of \(\UD_{\bfn_{\bfE}}(\Graf_{\bfE})\times[0,1]\) is a product of a connected component of the inverse image of \(\UD_{\bfn_{\bfE}}(\Graf_{\bfE})\) with \([0,1]\).
Thus \(\mathcal N(\widetilde H)\cong\widetilde H\times[0,1]\).
\end{proof}

\begin{remark}\label{Rem:CombinatorialHyperplanesUD}
In the product decomposition \(\mathcal N(H_{\bfE})\cong\UD_{\bfn_{\bfE}}(\Graf_{\bfE})\times[0,1]\), the two combinatorial hyperplanes associated to \(H_{\bfE}\) are the boundary components \(\UD_{\bfn_{\bfE}}(\Graf_{\bfE})\times\{0\}\) and \(\UD_{\bfn_{\bfE}}(\Graf_{\bfE})\times\{1\}\).
They are obtained by placing the moving particle at one or the other endpoint of \(\sfe\), while retaining the same distribution of the other particles.
\end{remark}

\begin{example}\label{Ex:HyperplaneData}
\Cref{Fig:HyperplaneCarrierDecomposition} illustrates these notions for \(n=4\).
In \Cref{Fig:HyperplaneData} the four particles of \(\bfE\) are red: one of them moves along
\(\sfe\), while the other three are distributed among the three components of
\(\Graf\setminus\sfe\), each of which is labeled by the number of particles it carries.
Let \(\sfL_1\) be the triangle and let \(\sfL_2\) be the \(4\)-cycle, so that
\(\bfn_{\bfE}(\sfL_1)=2\) and \(\bfn_{\bfE}(\sfL_2)=1\); the third component, drawn in
gray, carries no particle and is therefore discarded. Thus
\(\Graf_{\bfE}=\sfL_1\sqcup\sfL_2\) is the black part of the picture, and \(\UD_{\bfn_{\bfE}}(\Graf_{\bfE})\cong\UD_2(\sfL_1)\times\UD_1(\sfL_2)\). \Cref{Fig:HyperplaneCarrier} shows the corresponding product decomposition of the carrier
\(\mathcal N(H_{\bfE})\), given by \Cref{Lem:Hyperplane_in_UD_n}.
\end{example}

\begin{figure}[ht]
\subcaptionbox{The data \((\Graf_{\bfE},\bfn_{\bfE})\)\label{Fig:HyperplaneData}}[0.4\textwidth]{
\(
\begin{tikzpicture}[baseline=-.5ex]
\draw[thick,lightgray, fill] (0,0) -- (-135:1) circle (2pt) -- node[midway,below] {\(0\)} ++(-1,0) circle (2pt) (0,0) -- (135:1) (1,0) -- (2,0);
\draw[thick,fill] (135:1) circle (2pt) -- ++(150:1) circle (2pt) -- ++(-90:1) circle (2pt) -- (135:1);
\draw (135:1) ++(-0.57,0) node {\(2\)} +(-0.75,0) node {\(\sfL_1\)};
\draw[thick,fill] (2,0) circle (2pt) -- ++(45:1) circle (2pt) -- ++(-45:1) circle (2pt) -- ++(-135:1) circle (2pt) -- (2,0);
\draw (2.71,0) node {\(1\)} +(0.75,0.75) node {\(\sfL_2\)};
\draw[thick, fill, red] (0,0) circle (2pt) -- node[midway, below] {\(\sfe\)} (1,0) circle (2pt);
\draw[red,fill=red] (135:1) +(150:1) circle (3pt) +(-150:1) circle (3pt);
\draw[red,fill=red] (2,0) ++(45:1) ++(-45:1) circle (3pt);
\end{tikzpicture}
\)
}
\subcaptionbox{The carrier \(\mathcal N(H_{\bfE})\)\label{Fig:HyperplaneCarrier}}[0.4\textwidth]{
\(
\begin{tikzpicture}[baseline=-.5ex]
\draw[thick,fill=blue!5] (0,0) rectangle (2.6,1.6);
\draw[thick,dashed] (0,0.8) -- (2.6,0.8);
\draw (2.6,1.6) node[right=0.6ex] {\(\times\{1\}\)};
\draw (2.6,0) node[right=0.6ex] {\(\times\{0\}\)};
\draw (2.6,0.8) node[right=0.6ex] {\(H_{\bfE}\)};
\draw[<->] (-0.3,0) -- (-0.3,1.6);
\draw (-0.3,0.8) node[left=0.2ex] {\(\sfe\)};
\draw[<->] (0,-0.35) -- (2.6,-0.35);
\draw (1.3,-0.4) node[below] {\(\UD_{\bfn_{\bfE}}(\Graf_{\bfE})\)};
\end{tikzpicture}
\)
}
\caption{The hyperplane \(H_{\bfE}\) dual to an edge
\(\bfE=\{\sfe,\sfv_1,\dots,\sfv_{n-1}\}\) of \(\UD_n(\Graf)\), for \(n=4\).}
\label{Fig:HyperplaneCarrierDecomposition}
\end{figure}
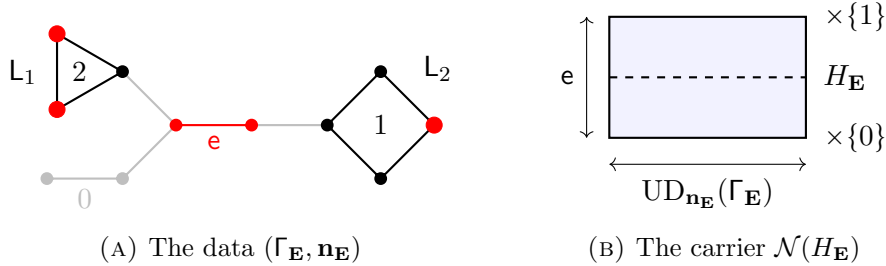

We conclude this subsection with an infinitude criterion for discrete graph braid groups.
A vertex of \(\Graf\) is called \emph{essential} if it has degree at least three. We write \(V_{\mathrm{ess}}(\Graf)\) for the set of essential vertices.

\begin{lemma}[Infinite discrete graph braid groups]\label{Lem:InfiniteGBG}
Let \(\Graf\) be a connected finite graph, and suppose that \(1\leq n\leq m=\#V(\Graf)\). Then \(\GBGdisc_n(\Graf)\) is infinite if and only if one of the following holds:
\begin{enumerate}
\item \(\Graf\) contains an essential vertex and \(\#V(\Graf)\geq n+2\geq4\);
\item \(\Graf\) contains a cycle and \(\#V(\Graf)\geq n+1\geq2\).
\end{enumerate}
\end{lemma}
\begin{proof}
Suppose first that \(\Graf\) contains an embedded cycle \(\ell\) and that \(m\geq n+1\). Write \(q=\#V(\ell)\), and choose \(r=\max\{1,n-(m-q)\}\). 
Then \(1\leq r\leq q-1\), and there are at least \(n-r\) vertices outside \(\ell\). By the full-link criterion, fixing particles at \(n-r\) such vertices gives a locally convex subcomplex \(\UD_r(\ell)\hookrightarrow\UD_n(\Graf)\). 
Since \(q\geq r+1\), we have \(\pi_1(\UD_r(\ell))\cong\mathbb Z\): this is immediate when \(r=1\), and when \(r\geq2\) it follows from \Cref{Thm:DiscreteConfigurationRetraction}.
The inclusion is a local isometry and therefore injective on fundamental groups. Thus \(\GBGdisc_n(\Graf)\) is infinite.

Suppose next that \(\Graf\) contains an essential vertex and that \(m\geq n+2\geq4\). If \(\Graf\) contains an embedded cycle, then the preceding argument applies, since \(m\geq n+1\). Hence we may assume that \(\Graf\) contains no cycle. It is therefore a tree, and we may choose a tripod \(\tau\cong\sfK_{1,3}\) centered at the given essential vertex. Since \(m-4\geq n-2\), we may fix \(n-2\) particles outside \(\tau\).
This gives a locally convex subcomplex \(\UD_2(\tau)\hookrightarrow\UD_n(\Graf)\). The complex \(\UD_2(\tau)\) is a circle, so it again determines an infinite cyclic subgroup of \(\GBGdisc_n(\Graf)\).

Conversely, suppose that neither condition holds. If \(\Graf\) contains a cycle, then the failure of the second condition and the inequality \(n\leq m\) imply that \(n=m\). Hence \(\UD_n(\Graf)\) consists of a single point.

We may therefore assume that \(\Graf\) is a tree. If \(\Graf\) contains an essential vertex, then the failure of the first condition implies that either \(n=1\) or \(m\leq n+1\). If \(n=1\), then \(\UD_1(\Graf)=\Graf\), which is contractible. If \(m=n\), then \(\UD_n(\Graf)\) is a point. If \(m=n+1\), the map which assigns to a configuration its unique unoccupied vertex gives a cubical isomorphism \(\UD_n(\Graf)\cong\Graf\), so the configuration space is again contractible.

Finally, if \(\Graf\) contains neither a cycle nor an essential vertex, then \(\Graf\) is a path, and \(\UD_n(\Graf)\) is contractible. Thus \(\GBGdisc_n(\Graf)\) is trivial in all remaining cases.
\end{proof}

\section{The legal-pair factor system}\label{Section:LegalPairFactorSystem}

Throughout this and the next section, fix a connected finite graph \(\Graf\) and an integer \(n\) satisfying \(2\leq n<\#V(\Graf)\). No subdivision hypothesis is imposed.

In this section, we introduce legal pairs and prove that their lifts form a factor system in \(\widetilde{\UD}_n(\Graf)\). We retain singleton factors in order to obtain exact closure under gate projections; they will be discarded in \Cref{Section:LegalPairHierarchy} before defining the legal-pair hierarchy.

\subsection{Legal pairs}
As in the product decomposition in \eqref{Eq:ProductofGBG}, a configuration space on a disconnected graph depends not only on the graph but also on the distribution of the particles among its connected components. We encode these data by legal pairs.

\begin{definition}[Legal pair]\label{Def:LegalPair}
Let \(\sfL\subseteq\Graf\) be a possibly empty subgraph, not necessarily connected, and let \(\bfn\colon\pi_0(\sfL)\rightarrow\NZ\) be a function
satisfying \[\|\bfn\|=\sum_{\sfL_i\in\pi_0(\sfL)}\bfn(\sfL_i)\leq n.\]
We call \(\|\bfn\|\) the \emph{norm} of the pair \((\sfL,\bfn)\).

The pair \((\sfL,\bfn)\) is called \emph{legal} if, for every \(\sfL_i\in\pi_0(\sfL)\), \(0<\bfn(\sfL_i)\leq \#V(\sfL_i)\), and equality holds if and only if \(\sfL_i\) consists of a single vertex. 
In particular, \((\varnothing,\varnothing)\) is a legal pair. See \Cref{figure:legal_pairs}.
\end{definition}

\begin{figure}[ht]
\subcaptionbox{A legal pair}[0.45\textwidth]{
\(
(\sfL,\bfn)=\begin{tikzpicture}[baseline=-.5ex]
\draw[thick,fill] (0,0) circle (2pt) node[above] {\(1\)};
\draw[thick,fill] (1,0) circle (2pt) -- node[midway,above] {\(1\)} (2,0) circle (2pt);
\begin{scope}[xshift=4cm]
\draw[thick,fill] (0,0) node {\(2\)} (36:1) circle (2pt) -- (108:1) circle (2pt) -- (180:1) circle (2pt) -- (252:1) circle (2pt) -- (-36:1) circle (2pt) -- (36:1);
\end{scope}
\end{tikzpicture}
\)
}
\subcaptionbox{A non-legal pair}[0.45\textwidth]{
\(
(\sfL,\bfn)=\begin{tikzpicture}[baseline=-.5ex]
\draw[thick,fill] (0,0) circle (2pt) node[above] {\(1\)};
\draw[thick,fill] (1,0) circle (2pt) -- node[midway,above] {\(2\)} (2,0) circle (2pt);
\begin{scope}[xshift=4cm]
\draw[thick,fill] (0,0) node {\(1\)} (36:1) circle (2pt) -- (108:1) circle (2pt) -- (180:1) circle (2pt) -- (252:1) circle (2pt) -- (-36:1) circle (2pt) -- (36:1);
\end{scope}
\end{tikzpicture}
\)
}
\caption{An example and a non-example of a legal pair. The value of $\bfn(\sfL_i)$ is indicated by the number on each component.}
\label{figure:legal_pairs}
\end{figure}
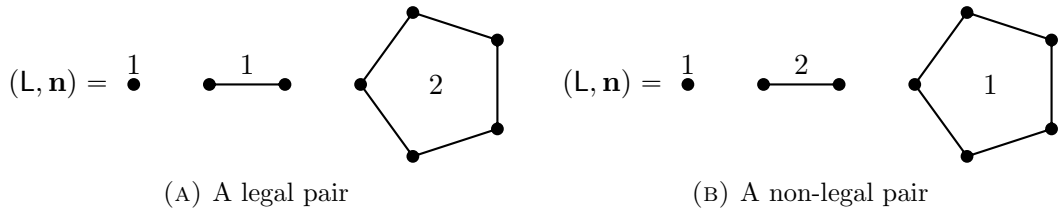

A legal pair \((\sfL,\bfn)\) determines the connected cube complex \(\UD_{\bfn}(\sfL)\) defined in \eqref{Eq:ProductofGBG}. For the empty legal pair, we regard \(\UD_0(\varnothing)=\{\varnothing\}\) as a point.

\begin{remark}[Normalization of particle distributions]\label{Rem:LegalPairNormalization}
The conditions in \Cref{Def:LegalPair} are a normalization convention for configuration spaces on subgraphs. A component carrying no particles may be deleted without changing the corresponding configuration space. If a non-singleton component \(\sfL_i\) carries \(\#V(\sfL_i)\) particles, then \(\UD_{\#V(\sfL_i)}(\sfL_i)=\{V(\sfL_i)\}\).
We therefore replace \(\sfL_i\) by the discrete subgraph \(V(\sfL_i)\), regarded as a union of singleton components, and assign one particle to each of them.

Thus any subgraph equipped with a particle distribution for which the corresponding configuration space is nonempty can be replaced, without changing the corresponding configuration space, by a legal pair. We refer to this replacement as \emph{normalization}.
\end{remark}

\begin{definition}[The relation \(\preceq\)]\label{Def:LegalPairRelation}
Let \((\sfL,\bfn)\) and \((\sfL',\bfn')\) be legal pairs. Suppose that \(\sfL\subseteq \sfL'\), and write \(\sfL^{\mathrm{ext}}=\sfL'\setminus\sfL\), with connected components \(\sfL^{\mathrm{ext}}_k\in\pi_0(\sfL^{\mathrm{ext}})\). 
We write \((\sfL,\bfn)\preceq(\sfL',\bfn')\) if there exists a function \(\bfn^{\mathrm{ext}}:\pi_0(\sfL^{\mathrm{ext}})\to\NZ\) such that
\begin{enumerate}
\item \(\bfn^{\mathrm{ext}}(\sfL^{\mathrm{ext}}_k)\leq\#V(\sfL^{\mathrm{ext}}_k)\) for every \(\sfL^{\mathrm{ext}}_k\in\pi_0(\sfL^{\mathrm{ext}})\), and
\item for every component \(\sfL'_i\in\pi_0(\sfL')\),
\[\bfn'(\sfL'_i)=\sum_{\substack{\sfL_j\in\pi_0(\sfL)\\ \sfL_j\subseteq \sfL'_i}}\bfn(\sfL_j) +\sum_{\substack{\sfL^{\mathrm{ext}}_k\in\pi_0(\sfL^{\mathrm{ext}})\\ \sfL^{\mathrm{ext}}_k\subseteq \sfL'_i}}\bfn^{\mathrm{ext}}(\sfL^{\mathrm{ext}}_k).\]
\end{enumerate}
In this case, we say that \(\bfn^{\mathrm{ext}}\) \emph{witnesses} the relation \((\sfL,\bfn)\preceq(\sfL',\bfn')\). See \Cref{figure:comparable_legal_pairs}.
\end{definition}

By definition, if \((\sfL,\bfn)\preceq(\sfL',\bfn')\) and \(\sfL=\sfL'\), then two functions \(\bfn\) and \(\bfn'\) must be identical.

\begin{figure}[ht]
\subcaptionbox{Comparable legal pairs}{
\(
(\sfL,\bfn)=\left(\begin{tikzpicture}[baseline=-.5ex]
\draw[thick,fill] (0,0) circle (2pt) node[above] {\(1\)};
\draw[thick,fill] (1,0) circle (2pt) -- node[midway,above] {\(1\)} (2,0) circle (2pt);
\begin{scope}[xshift=4cm]
\draw[thick,fill,lightgray] (252:1) -- (-36:1) circle (2pt) -- (36:1) circle (2pt) -- (108:1);
\draw[thick,fill] (0,0) node {\(2\)} (108:1) circle (2pt) -- (180:1) circle (2pt) -- (252:1) circle (2pt);
\end{scope}
\end{tikzpicture}\right)
\preceq
\left(\begin{tikzpicture}[baseline=-.5ex]
\draw[thick,fill] (0,0) circle (2pt) node[above] {\(1\)};
\draw[thick,fill] (1,0) circle (2pt) -- node[midway,above] {\(1\)} (2,0) circle (2pt);
\begin{scope}[xshift=4cm]
\draw[thick,fill] (0,0) node {\(3\)} (36:1) circle (2pt) -- (108:1) circle (2pt) -- (180:1) circle (2pt) -- (252:1) circle (2pt) -- (-36:1) circle (2pt) -- (36:1);
\end{scope}
\end{tikzpicture}\right)=(\sfL',\bfn')
\)
}

\subcaptionbox{Incomparable legal pairs}{
\(
(\sfL,\bfn)=\left(\begin{tikzpicture}[baseline=-.5ex]
\draw[thick,fill] (0,0) circle (2pt) node[above] {\(1\)};
\draw[thick,fill] (1,0) circle (2pt) -- node[midway,above] {\(1\)} (2,0) circle (2pt);
\begin{scope}[xshift=4cm]
\draw[thick,fill] (0,0) node {\(2\)} (36:1) circle (2pt) -- (108:1) circle (2pt) -- (180:1) circle (2pt) -- (252:1) circle (2pt) -- (-36:1) circle (2pt) -- (36:1);
\end{scope}
\end{tikzpicture}\right)
\not\preceq
\left(\begin{tikzpicture}[baseline=-.5ex]
\draw[thick,fill] (0,0) circle (2pt) node[above] {\(1\)};
\draw[thick,fill] (1,0) circle (2pt) -- node[midway,above] {\(1\)} (2,0) circle (2pt);
\begin{scope}[xshift=4cm]
\draw[thick,fill] (0,0) node {\(3\)} (36:1) circle (2pt) -- (108:1) circle (2pt) -- (180:1) circle (2pt) -- (252:1) circle (2pt) -- (-36:1) circle (2pt) -- (36:1);
\end{scope}
\end{tikzpicture}\right)=(\sfL,\bfn')
\)
}
\caption{An illustration of the relation \(\preceq\) and a pair of
legal pairs which are not \(\preceq\)-comparable.}
% \caption{Comparable and incomparable legal pairs.}
\label{figure:comparable_legal_pairs}
\end{figure}
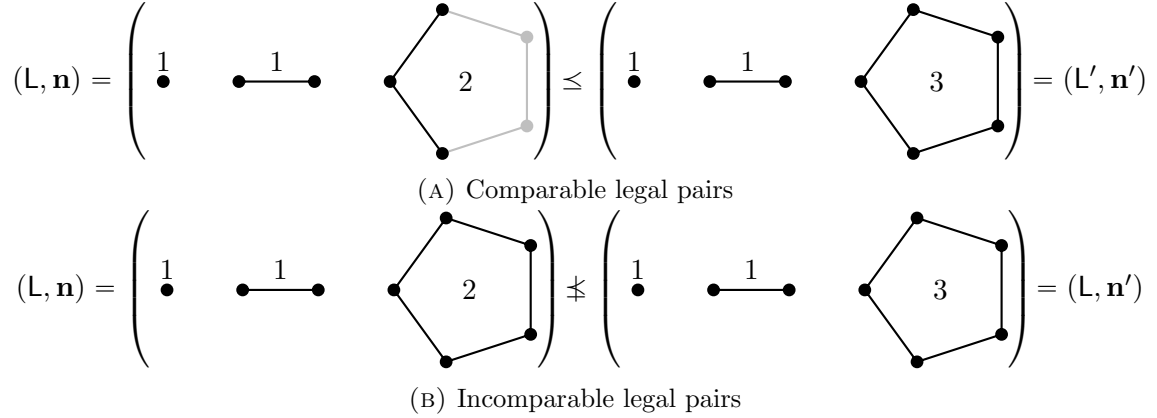

\begin{example}[Nonuniqueness of witnesses]\label{Ex:NonuniquenessOfWitnesses}
When the norms of two \(\preceq\)-comparable legal pairs are different, the witness of the relation need not be unique. Consider the following strict comparison:
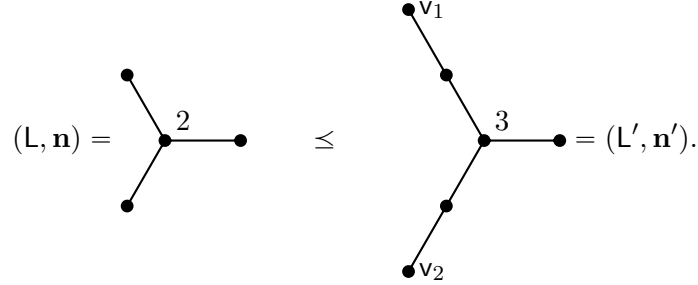
\[
(\sfL,\bfn)=
\begin{tikzpicture}[baseline=-.5ex]
\draw[thick,fill]
(0,0) circle (2pt) node[above right] {\(2\)}
-- (1,0) circle (2pt)
(0,0) -- (120:1) circle (2pt)
(0,0) -- (240:1) circle (2pt);
\end{tikzpicture}
\qquad\preceq\qquad
\begin{tikzpicture}[baseline=-.5ex]
\draw[thick,fill]
(0,0) circle (2pt) node[above right] {\(3\)}
-- (1,0) circle (2pt)
(0,0) -- (120:1) circle (2pt)
-- (120:2) circle (2pt) node[right] {\(\sfv_1\)}
(0,0) -- (240:1) circle (2pt)
-- (240:2) circle (2pt) node[right] {\(\sfv_2\)};
\end{tikzpicture}
=(\sfL',\bfn').
\]
Here \(\sfL^{\mathrm{ext}}=\{\sfv_1,\sfv_2\}\), and the relation is witnessed by precisely two functions defined as follows:
\[\bfn^{\mathrm{ext}}(\sfv_1)=1,\quad\bfn^{\mathrm{ext}}(\sfv_2)=0, \qquad\text{or}\qquad \bfn^{\mathrm{ext}}(\sfv_1)=0,\quad \bfn^{\mathrm{ext}}(\sfv_2)=1.\]
Each of these witnesses is realized by a unique vertex configuration, namely \(\mathbf y_1=\{\sfv_1\}\) and \(\mathbf y_2=\{\sfv_2\}\), respectively. 
For each \(i\in\{1,2\}\), the configuration \(\mathbf y_i\) determines a cubical embedding \(\iota_{\mathbf y_i}\colon\UD_{\bfn}(\sfL)\hookrightarrow\UD_{\bfn'}(\sfL')\) which is a local isometry. Since \(\mathbf y_1\neq\mathbf y_2\), the two embeddings are distinct.
Thus, when the norms differ, the relation \((\sfL,\bfn)\preceq(\sfL',\bfn')\) need not determine a unique downstairs embedding; the witness and a realizing vertex configuration must be retained when this comparison is lifted.
\end{example}

\begin{lemma}\label{Lem:LocalisometrybetweenLegalPairs}
Let \((\sfL,\bfn)\preceq(\sfL',\bfn')\) be legal pairs, and let \(\bfn^{\mathrm{ext}}\) witness this relation. If \(\bfy\) is a vertex configuration in \(\sfL^{\mathrm{ext}}\) realizing \(\bfn^{\mathrm{ext}}\), then \(\iota_{\mathbf y}\colon \UD_{\bfn}(\sfL)\rightarrow\UD_{\bfn'}(\sfL')\), defined by \(\mathbf x\longmapsto\mathbf x\sqcup\mathbf y\), is a cubical embedding which is a local isometry.

If \(\|\bfn\|=\|\bfn'\|\), then every witness is identically zero and is realized uniquely by the empty configuration. In particular, the relation determines a unique embedding of the above form.
\end{lemma}
\begin{proof}
Let \((\sfL^\bfy,\bfn^\bfy)\coloneqq(\sfL\sqcup\bfy,\ \bfn\sqcup\bfone_{\bfy})\) be the legal pair obtained from
\((\sfL,\bfn)\) by adjoining every vertex of \(\bfy\) as a singleton component, where \(\bfone_{\bfy}\) assigns one particle to each of them.
Then \(\UD_{\bfn}(\sfL)\rightarrow\UD_{\bfn^\bfy}(\sfL^\bfy)\), defined by \(\bfx\mapsto\bfx\sqcup\bfy\),
is a cubical isomorphism.

The natural cubical embedding \(\UD_{\bfn^\bfy}(\sfL^\bfy)\hookrightarrow \UD_{\bfn'}(\sfL')\) is a local isometry. Indeed, a vertex of the link of \(\mathbf x\sqcup\mathbf y\) in \(\UD_{\bfn^\bfy}(\sfL^\bfy)\) corresponds to moving one particle of \(\mathbf x\) along an edge of \(\sfL\). The same move defines a vertex of its link in \(\UD_{\bfn'}(\sfL')\), so the induced map on links is injective.
If a collection of vertices in the image of this link map spans a simplex in the target link, then the corresponding closed edges of \(\sfL\) are pairwise disjoint and are disjoint from all stationary particles of \(\mathbf x\sqcup\mathbf y\). The same moves can be performed simultaneously in \(\UD_{\bfn^\bfy}(\sfL^\bfy)\) and thus the collection already spans a simplex in the source link. The full-link criterion now shows that this embedding is a local isometry. Hence so is \(\iota_{\bfy}\), being the composition of this embedding with a cubical isomorphism.

Finally, suppose that \(\|\bfn\|=\|\bfn'\|\). Summing the equalities in the definition of \(\preceq\) over all components of \(\sfL'\) gives
\[0=\|\bfn'\|-\|\bfn\|=\sum_{\sfL^{\mathrm{ext}}_k\in\pi_0(\sfL^{\mathrm{ext}})}\bfn^{\mathrm{ext}}(\sfL^{\mathrm{ext}}_k).\]
Since \(\bfn^{\mathrm{ext}}\) is nonnegative, it is identically zero. Its unique realizing configuration is therefore \(\mathbf y=\varnothing\). Hence the relation determines a unique embedding of the stated form.
\end{proof}

\begin{convention}
By the final assertion of \Cref{Lem:LocalisometrybetweenLegalPairs}, every legal pair \((\sfL,\bfn)\) with \(\|\bfn\|=n\) uniquely determines a cubical embedding \(\UD_{\bfn}(\sfL)\hookrightarrow\UD_n(\Graf)\) which is a local isometry. We identify \(\UD_{\bfn}(\sfL)\) with its image, and call this image the \emph{legal-pair subcomplex} associated to \((\sfL,\bfn)\).
\end{convention}

\begin{lemma}\label{Lem:Intersection_of_Legal_pairs}
Let \((\sfL,\bfn)\) and \((\sfL',\bfn')\) be legal pairs of norm \(n\). Then
the intersection \[\UD_{\bfn}(\sfL)\cap\UD_{\bfn'}(\sfL')\subseteq\UD_n(\Graf)\]
is a disjoint union of legal-pair subcomplexes.

More precisely, every connected component of the intersection is represented by the normalization of a pair \((\sfL'',\bfn'')\) with \(\|\bfn''\|=n\) satisfying the following conditions:
\begin{enumerate}[label=\textup{(\arabic*)}, ref=\arabic*]
\item
\(\sfL''\) is a union of connected components of \(\sfL\cap\sfL'\);
\item
\(\bfn''\colon\pi_0(\sfL'')\rightarrow\mathbb Z_{\geq1}\) satisfies
\(\bfn''(\sfL''_k)\leq\#V(\sfL''_k)\) for every \(\sfL''_k\in\pi_0(\sfL'')\), together with the following equalities:
\[\bfn(\sfL_i)=\sum_{\substack{\sfL''_k\in\pi_0(\sfL'')\\ \sfL''_k\subseteq\sfL_i}}\bfn''(\sfL''_k),\quad\text{\(\forall \sfL_i\in\pi_0(\sfL)\),}\qquad\text{and}\qquad
\bfn'(\sfL'_j)=\sum_{\substack{\sfL''_k\in\pi_0(\sfL'')\\ \sfL''_k\subseteq\sfL'_j}}\bfn''(\sfL''_k),\quad\text{\(\forall\sfL'_j\in\pi_0(\sfL')\).}\]
\end{enumerate}
\end{lemma}
\begin{proof}
A cell of \(\UD_n(\Graf)\), represented by an unordered collection \(\{\sigma_1,\dots,\sigma_n\}\), lies in \(\UD_{\bfn}(\sfL)\cap\UD_{\bfn'}(\sfL')\) if and only if every \(\sigma_j\) is contained in \(\sfL\cap\sfL'\) and both sets of particle-counting conditions determined by \(\bfn\) and \(\bfn'\) are satisfied.
Such a cell determines a union \(\sfL''\) of connected components of \(\sfL\cap\sfL'\), namely those containing at least one \(\sigma_j\), together with a positive particle distribution 
\[\bfn''\colon\pi_0(\sfL'')\rightarrow\mathbb Z_{\geq1}\quad\text{defined by}\quad\bfn''(\sfL''_k)=\#\{j\mid \sigma_j\subseteq\sfL''_k\}.\]
The cells determining the same pair \((\sfL'',\bfn'')\) form precisely the subcomplex \(\UD_{\bfn''}(\sfL'')\).

Compatibility with the two original particle distributions is equivalent to the equalities in \textup{(2)}. Conversely, every pair \((\sfL'',\bfn'')\) satisfying these conditions determines a subcomplex
\[\UD_{\bfn''}(\sfL'')\subseteq \UD_{\bfn}(\sfL)\cap\UD_{\bfn'}(\sfL'),\]
and every cell in the intersection belongs to exactly one such subcomplex.

The pair \((\sfL'',\bfn'')\) can fail to be legal only when some non-singleton component \(\sfL''_k\in\pi_0(\sfL'')\) satisfies \(\bfn''(\sfL''_k)=\#V(\sfL''_k)\). By \Cref{Rem:LegalPairNormalization}, normalization preserves both the norm and the associated subcomplex.
Thus the resulting pairs are legal pairs of norm \(n\), and the intersection is a disjoint union of legal-pair subcomplexes. Since each of these subcomplexes is connected, they are precisely the connected components of the intersection.
\end{proof}

\subsubsection{Undistorted product subgroups}

The following is an easy consequence of \Cref{Lem:LocalisometrybetweenLegalPairs}.

\begin{corollary}[Undistorted legal-pair subgroups]\label{Cor:UndistortedLegalPairSubgroups}
For every legal pair \((\sfL,\bfn)\) of norm \(n\), the inclusion \(\UD_{\bfn}(\sfL)\hookrightarrow\UD_n(\Graf)\), after choices of basepoints, induces an injective homomorphism with undistorted image. Moreover,
\[\pi_1\bigl(\UD_{\bfn}(\sfL)\bigr)\cong\prod_{\sfL_i\in\pi_0(\sfL)}\GBGdisc_{\bfn(\sfL_i)}(\sfL_i).\]
\end{corollary}
\begin{proof}
The product decomposition follows from \eqref{Eq:ProductofGBG}. By \Cref{Lem:LocalisometrybetweenLegalPairs}, the inclusion \(\UD_{\bfn}(\sfL)\hookrightarrow\UD_n(\Graf)\) is a local isometry, so it is injective on fundamental groups and every lift is convex by \Cref{Lem:IntersectionOfLifts}. The Milnor--{\v S}varc lemma then implies that its image is undistorted.
\end{proof}

\begin{example}[Free abelian and stable-range product subgroups]
\label{Ex:LegalPairProductSubgroups}
Assume that \(\Graf\) is sufficiently subdivided for \(n\); by \Cref{Thm:DiscreteConfigurationRetraction}, \(\GBGdisc_n(\Graf)\cong\GBG_n(\Graf)\).

\begin{enumerate}[label=\textup{(\arabic*)}, ref=\arabic*]
\item\label{Item:FreeAbelianLegalPair}
Let \(\ell_1,\dots,\ell_p\) and \(\sfv_1,\dots,\sfv_q\) be \(p\) embedded cycles and \(q\) essential vertices that are pairwise disjoint, where \(p+q>0\) and \(p+2q\leq n\).
Since \(\Graf\) is sufficiently subdivided for \(n\), we may choose pairwise disjoint tripods \(\tau_1,\dots,\tau_q\), centered respectively at \(\sfv_1,\dots,\sfv_q\), which are also disjoint from the cycles.

Assign one particle to each cycle and two particles to each tripod. If \(p>0\), assign the remaining \(n-p-2q\) particles to \(\ell_1\). This is possible because \(\ell_1\) has at least \(n+1\) vertices. If \(p=0\), fix the remaining \(n-2q\) particles at distinct vertices outside the tripods. The sufficient-subdivision conditions provide enough such vertices, as may be seen by following a bivalent path from one of the tripod centers to the next non-bivalent vertex, or back to the same center around a cycle.

Let \((\sfL_{\mathrm{ab}},\bfn_{\mathrm{ab}})\) be the resulting legal pair, and set \(a_i=\bfn_{\mathrm{ab}}(\ell_i)\).
Omitting the one-point factors, its legal-pair subcomplex decomposes as
\[\UD_{\bfn_{\mathrm{ab}}}(\sfL_{\mathrm{ab}})\cong\prod_{i=1}^{p}\UD_{a_i}(\ell_i)\times\prod_{j=1}^{q}\UD_2(\tau_j).\]
Each space \(\UD_{a_i}(\ell_i)\) has infinite cyclic fundamental group, as does each \(\UD_2(\tau_j)\). Hence
\[\pi_1\bigl(\UD_{\bfn_{\mathrm{ab}}}(\sfL_{\mathrm{ab}})\bigr)\cong\mathbb Z^{p+q}.\]
By \Cref{Cor:UndistortedLegalPairSubgroups}, this gives an undistorted subgroup \(\mathbb Z^{p+q}<\GBGdisc_n(\Graf)\cong\GBG_n(\Graf)\).

\item\label{Item:StableRangeLegalPair}
Let \(m(\Graf)\) be the number of essential vertices of \(\Graf\), and let \(m_3(\Graf)\) be the number of trivalent vertices. Suppose that \(m(\Graf)>0\) and \(n\geq2m(\Graf)+m_3(\Graf)\).
After passing to a further subdivision, if necessary, and continuing to denote it by \(\Graf\), choose pairwise disjoint subgraphs \(\sfS_{\sfv}\), indexed by the essential vertices \(\sfv\) of \(\Graf\), as follows:
\begin{itemize}
\item if \(\deg(\sfv)\geq4\), let \(\sfS_{\sfv}\) be the star consisting of \(\sfv\) and the first edge in every incident direction, and set \(r_{\sfv}=2\);
\item if \(\deg(\sfv)=3\), let \(\sfS_{\sfv}\) be a once-subdivided tripod centered at \(\sfv\), and set \(r_{\sfv}=3\).
\end{itemize}
We may also arrange that the remaining
\[n-\sum_{\sfv\text{ essential}}r_{\sfv}=n-\bigl(2m(\Graf)+m_3(\Graf)\bigr)\]
particles can be fixed at distinct vertices outside these subgraphs.

This determines a legal pair \((\sfL_{\mathrm{prod}},\bfn_{\mathrm{prod}})\) satisfying
\[\UD_{\bfn_{\mathrm{prod}}}(\sfL_{\mathrm{prod}})\cong \prod_{\sfv\in V_{\mathrm{ess}}(\Graf)}\UD_{r_{\sfv}}(\sfS_{\sfv}).\]
For each essential vertex \(\sfv\), set \(H_{\sfv}=\GBGdisc_{r_{\sfv}}(\sfS_{\sfv})\). By \cite[Examples~4 and~5]{JS25GBG}, the group \(H_{\sfv}\) is a finitely generated nonabelian free group in either case. Therefore \Cref{Cor:UndistortedLegalPairSubgroups} gives an undistorted subgroup  \(\prod_{\sfv\in V_{\mathrm{ess}}(\Graf)}H_{\sfv}<\GBGdisc_n(\Graf)\).
For each \(\sfv\), choose a rank-two free factor \(K_{\sfv}<H_{\sfv}\). Then \(\prod_{\sfv}K_{\sfv}\) is undistorted in \(\prod_{\sfv}H_{\sfv}\). Hence, we obtain an undistorted subgroup
\[\prod_{\sfv\in V_{\mathrm{ess}}(\Graf)}K_{\sfv}\cong\bbF_2^{\,m(\Graf)}<\GBG_n(\Graf).\]
\end{enumerate}

The second construction recovers
\cite[Theorem~2]{JS25GBG}.
\end{example}

\subsection{Construction of the factor system}
Let \(p\colon\widetilde{\UD}_n(\Graf)\rightarrow\UD_n(\Graf)\) be the universal covering map. If \(A\subseteq\UD_n(\Graf)\) is a legal-pair subcomplex, i.e., the subcomplex associated to a legal pair with norm equal to \(n\), then we call each connected component of \(p^{-1}(A)\) a \emph{legal-pair lift}; see \Cref{Fig:LegalPairTerminology}.
In this subsection, we prove that the collection of all legal-pair lifts forms a factor system in \(\widetilde{\UD}_n(\Graf)\).

\begin{figure}[ht]
\centering
\resizebox{\linewidth}{!}{%
\begin{tikzpicture}[
box/.style={draw,
    rounded corners=3pt,
    align=center,
    inner sep=7pt},
  every node/.style={font=\normalsize}
]
\node[box] (pair) at (0,0)
{a legal pair \((\sfL,\bfn)\) with an embedding \(\UD_{\bfn}(\sfL)\hookrightarrow\UD_n(\Graf)\)};
\node[box,text width=3cm] (lower) at (-4.5,-2.2){an embedding requires a witness};
\node[box,text width=4.1cm] (equal) at (0,-2.2)
  {unique embedding;\\  \emph{legal-pair subcomplex}};
\node[box,text width=4.9cm] (lifts) at (5.5,-2.2){components of \(p^{-1}\!\left(\UD_{\bfn}(\sfL)\right)\);\\ \emph{legal-pair lifts}};

\draw[->,thick]([xshift=6mm]pair.south west) -- node[pos=.55,right] {\(\|\bfn\|<n\)} (lower.north);
\draw[->,thick] (pair) --  node[pos=.58,right] {\(\|\bfn\|=n\)} (equal);
\draw[->,thick] (equal) -- (lifts);
\end{tikzpicture}%
}
\caption{Terminology for legal pairs relative to the fixed ambient particle number \(n\). The terms \emph{legal-pair subcomplex} and \emph{legal-pair lift} are reserved for legal pairs of norm \(n\).}
\label{Fig:LegalPairTerminology}
\end{figure}
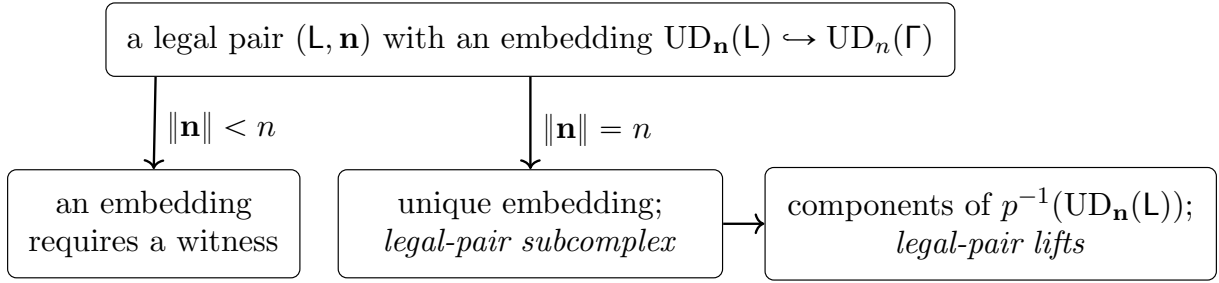

The description of hyperplane carriers in \Cref{Lem:Hyperplane_in_UD_n} immediately gives the following.

\begin{corollary}[Hyperplane carriers and combinatorial hyperplanes as legal-pair subcomplexes]\label{Cor:CombinatorialHyperplanesLegal}
Every hyperplane carrier in \(\UD_n(\Graf)\), as well as each of its
two combinatorial hyperplanes, is a legal-pair subcomplex.
\end{corollary}
\begin{proof}
Let \(H_{\bfE}\) be the hyperplane dual to an edge \(\bfE=\{\sfe,\sfv_1,\dots,\sfv_{n-1}\}\), and retain the notation \((\Graf_{\bfE},\bfn_{\bfE})\) introduced immediately before \Cref{Lem:Hyperplane_in_UD_n}. Then, by \Cref{Lem:Hyperplane_in_UD_n}, 
\[\mathcal N(H_{\bfE})\cong \UD_{\bfn_{\bfE}}(\Graf_{\bfE})\times\UD_1(\sfe).\]
Thus the carrier is the configuration-space subcomplex associated to the particle distribution on \(\Graf_{\bfE}\sqcup\sfe\) which restricts to \(\bfn_{\bfE}\) on \(\Graf_{\bfE}\) and assigns one particle to \(\sfe\). This distribution has norm \(n\), so after normalization the carrier is a legal-pair subcomplex. 

Replacing \(\sfe\) by either of its endpoints gives the two combinatorial hyperplanes. The same argument shows that they are legal-pair subcomplexes.
\end{proof}

We next record that legal-pair lifts are preserved under parallel transport across hyperplane carriers.

\begin{lemma}[Parallel transport across a hyperplane]\label{Lem:ParallelTransportLegalPair}
Let \(F\) be a legal-pair lift in \(\widetilde{\UD}_n(\Graf)\), and let \(\widetilde H\) be a hyperplane such that \(F\cap\mathcal N(\widetilde H)\neq\varnothing\). Then \(F\cap\mathcal N(\widetilde H)\) is also a legal-pair lift. 

Moreover, if \(F\cap\mathcal N(\widetilde H)\) is contained in one of the two combinatorial hyperplanes of \(\mathcal N(\widetilde H)\), then its parallel copy in the other combinatorial hyperplane is also a legal-pair lift.
\end{lemma}
\begin{proof}
Let \(H_{\bfE}\) be the hyperplane of \(\UD_n(\Graf)\) covered by \(\widetilde H\). By \Cref{Lem:Hyperplane_in_UD_n}, \(\mathcal N(\widetilde H)\) is a connected component of \(p^{-1}(\mathcal N(H_{\bfE}))\). Since \(\mathcal N(H_{\bfE})\) is a legal-pair subcomplex by \Cref{Cor:CombinatorialHyperplanesLegal}, \(\mathcal N(\widetilde H)\) is a legal-pair lift. Hence \Cref{Lem:Intersection_of_Legal_pairs,Lem:IntersectionOfLifts} shows that \(F\cap\mathcal N(\widetilde H)\) is a legal-pair lift.

Under the carrier product decomposition, the boundary swap replaces the particle at one endpoint of the edge of \(\Graf\) dual to \(\widetilde H\) by a particle at the other endpoint, while leaving all remaining cells unchanged. It therefore sends \(F\cap\mathcal N(\widetilde H)\) to a connected component of the full preimage of a configuration-space subcomplex determined by a particle distribution of norm \(n\) on a subgraph of \(\Graf\). After normalization, this parallel copy is a legal-pair lift.
\end{proof}

Consequently, by iterating the lemma along the interval factor, legal-pair lifts are preserved under parallel transport between the boundary fibers of a cubical product region.

\begin{proposition}[The full legal-pair factor system]\label{Prop:LCsubcomplexes_of_UD_n}
Let \(\frFLP^0\) be the collection of all legal-pair lifts in \(\widetilde{\UD}_n(\Graf)\). Then \(\frFLP^0\) is a \(\GBGdisc_n(\Graf)\)-invariant factor system.

Moreover, every gate projection between two factors is again a factor, so the constant in \Cref{Def:factor system}\eqref{Item:Projections} may be taken to be \(0\).
\end{proposition}
\begin{proof}
The deck invariance of \(\frFLP^0\) is immediate from its definition: a deck transformation sends legal-pair lifts to legal-pair lifts. We now verify the four axioms in \Cref{Def:factor system}.

Each element of \(\frFLP^0\) is a nonempty convex subcomplex of \(\widetilde{\UD}_n(\Graf)\). 
Indeed, by \Cref{Lem:LocalisometrybetweenLegalPairs}, the inclusion \(\UD_{\bfn}(\sfL)\hookrightarrow\UD_n(\Graf)\) is a local isometry, and hence its image is locally convex. Therefore every lift is convex by \Cref{Lem:IntersectionOfLifts}.
Moreover, \(\widetilde{\UD}_n(\Graf)\in\frFLP^0\), since \((\Graf,\bfn_{\Graf})\) is a legal pair, where \(\bfn_{\Graf}(\Graf)=n\). Thus \Cref{Def:factor system}\eqref{Item:Elements_of_FS} holds.

Every combinatorial hyperplane of \(\widetilde{\UD}_n(\Graf)\) is a connected component of the full preimage of a combinatorial hyperplane downstairs. Hence it is a legal-pair lift by \Cref{Cor:CombinatorialHyperplanesLegal}.
Let \(K\) be a combinatorial hyperplane, and let \(F\) be a non-singleton convex subcomplex with \(F\parallel K\). By \cite[Lemma~2.4]{BHS17I}, \(K\) and \(F\) occur as the boundary fibers of a cubical product region \(K\times I\). Parallel transport along \(I\), together with \Cref{Lem:ParallelTransportLegalPair}, shows that \(F\) is a legal-pair lift, proving \Cref{Def:factor system}\eqref{Item:CombinatorialHyperplane}.

Since \(\Graf\) is finite, there are only finitely many legal pairs of norm \(n\). For each legal-pair subcomplex downstairs, a vertex of \(\widetilde{\UD}_n(\Graf)\) belongs to at most one of its lifts. Consequently, there is a uniform bound on the number of elements of \(\frFLP^0\) containing any given vertex. This proves \Cref{Def:factor system}\eqref{Item:Delta}.

It remains to verify \Cref{Def:factor system}\eqref{Item:Projections}. Let \(F,F'\in\frFLP^0\). If \(F\cap F'\neq\varnothing\), then the gate projection of either factor to the other is their intersection: \(\frg_F(F')=F\cap F'\).
By \Cref{Lem:Intersection_of_Legal_pairs,Lem:IntersectionOfLifts}, this intersection is a legal-pair lift, and hence belongs to \(\frFLP^0\).

We argue by induction on the combinatorial distance \(d(F,F')\). The case \(d(F,F')=0\) is the intersection case established above. Suppose that \(d(F,F')>0\), and put \(A=\frg_{F'}(F)\) and \(B=\frg_F(F')\). By \cite[Lemmas~2.4 and~2.6]{BHS17I}, \(A\) and \(B\) are parallel and occur as the boundary fibers of a cubical product region \(A\times[0,d(F,F')]\). Let \(\widetilde H\) be the hyperplane dual to the first edge of the interval factor, starting from the \(A\)-side. The product structure gives
\[A\subseteq K:=F'\cap\mathcal N(\widetilde H).\]
Since \(\widetilde H\) separates \(F'\) from \(F\), the subcomplex \(K\) lies in the combinatorial boundary of the carrier on the \(F'\)-side. Let \(K^+\) be its parallel copy in the opposite combinatorial boundary. By \Cref{Lem:ParallelTransportLegalPair}, both \(K\) and \(K^+\) are legal-pair lifts.

Since \(A\subseteq K\subseteq F'\) and \(\frg_F(A)=B\), we have \(\frg_F(K)=B\). Moreover, paired vertices of \(K\) and \(K^+\) are separated only by \(\widetilde H\), which does not cross \(F\), and therefore have the same gate in \(F\). Hence
\[\frg_F(K^+)=\frg_F(K)=\frg_F(F').\]
Finally, \(d(F,K^+)=d(F,F')-1\). The induction hypothesis applied to \(F\) and \(K^+\) shows that \(\frg_F(K^+)\) is a legal-pair lift, and thus \(\frg_F(F')\in\frFLP^0\).
Therefore every gate projection between factors belongs to \(\frFLP^0\), so \Cref{Def:factor system}\eqref{Item:Projections} holds with \(\xi=0\).
\end{proof}

\subsection{Comparison with rich-family factor systems}
\label{Subsection:Comparison}

We briefly recall the rich-family construction. 
Let \(X\) be a compact special cube complex, and let \(\mathcal C_{0}X\) be its crossing graph. A collection \(\mathcal R\) of subgraphs of \(\mathcal C_{0}X\) is a \emph{rich family} if it contains \(\mathcal C_{0}X\) and every vertex link and is closed under intersections; the \emph{minimal} rich family is the family generated by these requirements, whereas the \emph{maximal} rich family consists of all subgraphs of \(\mathcal C_{0}X\). 
For \(\Omega\in\mathcal R\), two \(1\)-cubes \(e,e'\) of \(X\) are called \(\Omega\)-equivalent if there is an edge path \(e=e_0,\dots,e_k=e'\) such that every \(e_i\) is dual to a hyperplane represented by a vertex of \(\Omega\). We call a full subcomplex whose \(1\)-skeleton is an \(\Omega\)-equivalence class an \emph{\(\Omega\)-equivalence subcomplex}. The lifts of all such subcomplexes, as \(\Omega\) ranges over \(\mathcal R\), form a factor system on \(\widetilde X\); see \cite[Definitions~8.2 and~8.7 and Corollary~8.9]{BHS17I}.

\begin{remark}[Legal-pair and rich-family factor systems]\label{Rem:RichFamilyVersusLegalPairs}
The legal-pair factor system in \Cref{Prop:LCsubcomplexes_of_UD_n} need not coincide with the factor systems arising from either of the two canonical rich-family choices. In the rich-family construction for the compact special cube complex \(\UD_n(\Graf)\), a factor is a lift of an \(\Omega\)-equivalence subcomplex for some subgraph \(\Omega\) of the crossing graph \(\mathcal C_{0}(\UD_n(\Graf))\). A vertex of \(\mathcal C_{0}(\UD_n(\Graf))\) is a hyperplane of \(\UD_n(\Graf)\), and such a hyperplane is determined not only by an edge \(\sfe\) of \(\Graf\) but also by the distribution of the remaining particles among the components of \(\Graf\setminus\sfe\).

Berlyne used this rich-family construction and proposed a description of its factors in terms of graphical subgroups \cite[Theorem~5.1.10]{Ber21}. In the proof, a subgraph \(\Graf_\Omega\subset\Graf\) is associated to \(\Omega\) by taking the union of the underlying edge labels occurring among the hyperplanes represented in \(\Omega\). This passage retains the labels \(\sfe\) but forgets the component data of the remaining particles. Consequently, distinct subgraphs of \(\mathcal C_{0}(\UD_n(\Graf))\) may determine the same subgraph \(\Graf_\Omega\), while their equivalence subcomplexes can be different. In particular, the identification in that proof of an \(\Omega\)-equivalence subcomplex with a configuration-space component on \(\Graf_\Omega\), with all particles outside \(\Graf_\Omega\) fixed, does not hold in general; see \Cref{Ex:RichFamilyNotLegalPair}.

The rich-family construction itself still provides a factor system; the issue concerns only its proposed description by graphical subgroups. We therefore construct the legal-pair factor system directly from the subcomplexes \(\UD_{\bfn}(\sfL)\). Its factor-system axioms, including closure under gate projections, are verified in \Cref{Prop:LCsubcomplexes_of_UD_n}.
\end{remark}

The following example shows, within a single graph, how the legal-pair factor system differs from both canonical rich-family constructions.

\begin{example}[The two canonical rich-family constructions]\label{Ex:RichFamilyNotLegalPair}
Let \(\Graf=\sfK_{1,4}\), with center \(\sfv\), leaves \(\sfu_1,\dots,\sfu_4\), and edges \(\sfe_i=[\sfv,\sfu_i]\), and consider \(\UD_2(\Graf)\).
For a fixed edge \(\sfe_i\), the graph \(\Graf\setminus\sfe_i\) consists of the three isolated vertices \(\sfu_j\), for \(j\neq i\). Hence \(\UD_2(\Graf)\) has three distinct hyperplanes with underlying edge label \(\sfe_i\), represented by the edges \(\{\sfe_i,\sfu_j\}\), where \(j\neq i\).
Since no two edges of \(\Graf\) are disjoint, \(\UD_2(\Graf)\) is one-dimensional; in fact, it is the subdivision of the complete graph on the four vertices \(\{\sfv,\sfu_1\},\dots,\{\sfv,\sfu_4\}\). Thus each of its edges determines a distinct hyperplane. See \Cref{Fig:RichFamilyK14}.

Since \(\UD_2(\Graf)\) is one-dimensional, \(\mathcal C_{0}(\UD_2(\Graf))\) is discrete. Hence its canonical minimal rich family contains no proper nonempty subgraph of \(\mathcal C_{0}(\UD_2(\Graf))\). Let \(\tau\subset\Graf\) be the tripod obtained by deleting one of the four arms of \(\Graf\). Then \(\UD_2(\tau)\) is a proper legal-pair subcomplex of \(\UD_2(\Graf)\), and its lifts belong to the legal-pair factor system. They do not occur in the factor system associated to the canonical minimal rich family.

Now let \(\Omega\) be obtained from \(\mathcal C_{0}(\UD_2(\Graf))\) by deleting the single vertex corresponding to the hyperplane represented by \(\{\sfe_1,\sfu_2\}\) and retaining every other vertex. Since deleting this edge from \(\UD_2(\Graf)\) leaves a connected graph, all the remaining \(1\)-cubes form a single \(\Omega\)-equivalence class. Let \(Y_\Omega\) be the corresponding full subcomplex. Thus \(Y_\Omega\) is obtained from \(\UD_2(\Graf)\) by omitting the single edge \(\{\sfe_1,\sfu_2\}\). In particular, \(Y_\Omega\) is connected and proper, and its lifts occur in the factor system associated to the maximal rich family.

Every edge \(\sfe_i\) of \(\Graf\) still occurs as the underlying label of a hyperplane represented in \(\Omega\); in particular, \(\Graf_\Omega=\Graf\). Nevertheless, \(Y_\Omega\) is not a connected component of \(\UD_2(\Graf_\Omega)=\UD_2(\Graf)\), since the latter is connected and \(Y_\Omega\) is proper. Thus it is not the configuration-space component predicted by the graphical description in \cite[Theorem~5.1.10]{Ber21}.

The subcomplex \(Y_\Omega\) is also not a legal-pair subcomplex. Indeed, a legal-pair subcomplex whose edge labels contain all four edges of \(\Graf\) must have underlying subgraph \(\Graf\), and hence is the entire complex \(\UD_2(\Graf)\), rather than the proper subcomplex \(Y_\Omega\).

Thus the factor system arising from the canonical minimal rich family can omit legal-pair lifts, whereas the factor system arising from the maximal rich family can contain factors which are not legal-pair lifts.
\end{example}

\begin{figure}[ht]
\centering
\subcaptionbox{\(\Graf=\sfK_{1,4}\), with the tripod \(\tau\) in black%
  \label{Fig:K14Ambient}}[0.45\textwidth]{%
\begin{tikzpicture}[baseline=-.5ex]
  \coordinate (v)  at (0,0);
  \coordinate (u1) at (90:1.5);
  \coordinate (u2) at (180:1.5);
  \coordinate (u3) at (270:1.5);
  \coordinate (u4) at (0:1.5);
  \draw[lightgray, thick] (v) -- (u1);
  \draw[line width=1.2pt] (v) -- (u2);
  \draw[line width=1.2pt] (v) -- (u3);
  \draw[line width=1.2pt] (v) -- (u4);
  \foreach \i in {v,u2,u3,u4} {\draw[fill, thick] (\i) circle (2pt);}
  \draw[fill, thick, lightgray] (u1) circle (2pt);
  \draw (v)  node[above right=0.8ex] {\(\sfv\)};
  \draw (u1) node[above=0.6ex] {\(\sfu_1\)};
  \draw (u2) node[left=0.6ex]  {\(\sfu_2\)};
  \draw (u3) node[below=0.6ex] {\(\sfu_3\)};
  \draw (u4) node[right=0.6ex] {\(\sfu_4\)};
\end{tikzpicture}}
\hfill
\subcaptionbox{\(\UD_2(\Graf)\), with \(\UD_2(\tau)\) in black%
  \label{Fig:UD2K14}}[0.5\textwidth]{%
\begin{tikzpicture}[baseline=-.5ex, every node/.style={font=\scriptsize}]
  \coordinate (B1) at (0,0);
  \coordinate (B2) at (0,2);
  \coordinate (B3) at (210:2);
  \coordinate (B4) at (-30:2);
  \coordinate (M12) at (90:1);
  \coordinate (M13) at (210:1);
  \coordinate (M14) at (-30:1);
  \coordinate (M23) at (150:1);
  \coordinate (M24) at (30:1);
  \coordinate (M34) at (-90:1);
  \draw[line width=1.2pt] (B2) -- (M23) -- (B3) -- (M34) -- (B4) -- (M24) -- (B2);
  \draw[lightgray, thick] (B1) -- (M13) -- (B3);
  \draw[lightgray, thick] (B1) -- (M14) -- (B4);
  \draw[lightgray, thick] (B1) -- (M12);
  \draw[dashed, thick] (M12) -- (B2);
  \foreach \i in {B2,B3,B4,M23,M24,M34}
    {\draw[fill, thick] (\i) circle (2pt);}
  \foreach \i in {B1,M12,M13,M14}
    {\draw[fill, thick, lightgray] (\i) circle (2pt);}
  \draw (B2)  node[above=0.7ex]      {\(\{\sfv,\sfu_2\}\)};
  \draw (B3)  node[below left=0.2ex] {\(\{\sfv,\sfu_3\}\)};
  \draw (B4)  node[below right=0.2ex]{\(\{\sfv,\sfu_4\}\)};
  \draw (B1)  node[below=0.5ex]      {\(\{\sfv,\sfu_1\}\)};
  \draw (M12) node[below=0ex]      {\(\{\sfu_1,\sfu_2\}\)};
  \draw (M13) node[below=-0.2ex] {\(\{\sfu_1,\sfu_3\}\)};
  \draw (M14) node[below=-0.2ex]{\(\{\sfu_1,\sfu_4\}\)};
  \draw (M23) node[left=0.7ex]       {\(\{\sfu_2,\sfu_3\}\)};
  \draw (M24) node[right=0.7ex]      {\(\{\sfu_2,\sfu_4\}\)};
  \draw (M34) node[below=0.7ex]      {\(\{\sfu_3,\sfu_4\}\)};
\end{tikzpicture}}
\caption{The complex \(\UD_2(\sfK_{1,4})\) of \Cref{Ex:RichFamilyNotLegalPair}; the vertex subdividing the edge from \(\{\sfv,\sfu_i\}\) to \(\{\sfv,\sfu_j\}\) is \(\{\sfu_i,\sfu_j\}\). The gray and dashed cells belong to \(\UD_2(\Graf)\) but not to \(\UD_2(\tau)\), and the dashed edge is \(\{\sfe_1,\sfu_2\}\).}
\label{Fig:RichFamilyK14}
\end{figure}
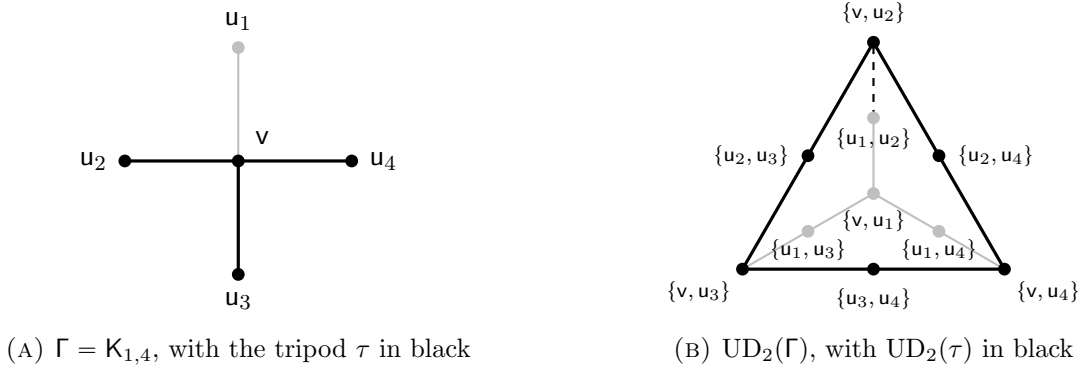

\section{The legal-pair hierarchy}\label{Section:LegalPairHierarchy}

We retain the assumptions on \(\Graf\) and \(n\) from \Cref{Section:LegalPairFactorSystem}. We now discard singleton factors, describe the resulting hierarchy and its domain relations, and prove that it belongs to \(\Xi\).

Set \(G=\GBGdisc_n(\Graf)\), and let \(\frFLP^0\) be the full legal-pair factor system from \Cref{Prop:LCsubcomplexes_of_UD_n}.

\begin{definition}[The legal-pair hierarchy]\label{Def:LegalPairHierarchy}
Set \(\frFLP=\{F\in\frFLP^0\mid F\text{ is not a singleton}\}\) and call \(\frFLP\) the \emph{reduced legal-pair factor system}. Since \(\frFLP^0\) is \(G\)-invariant and the \(G\)-action preserves singleton subcomplexes, \(\frFLP\) is \(G\)-invariant. By the reduction described in \Cref{Subsection:FactorSystems}, \(\frFLP\) is a factor system, and the constant in \Cref{Def:factor system}\eqref{Item:Projections} may be taken to be \(1\). 
The factor-system construction therefore gives an HHG structure on \(G\), whose domain set is
\[\frSLP=\frFLP/{\parallel}.\]
We call the HHG structure \((G,\frSLP)\) the \emph{legal-pair hierarchy}.
\end{definition}

\subsection{Standard paths and distinguished lifts}\label{Subsection:StandardPaths}

Choose a maximal tree \(\sfT\subseteq\Graf\), a leaf \(*\in\sfT\), and a planar embedding
\(\sfT\hookrightarrow\mathbb R^2\).
Following the standard discrete Morse-theoretic convention for graph braid groups, label the vertices of \(\Graf\) by positive integers by traversing counterclockwise along the boundary of a ribbon neighborhood of \(\sfT\), starting at \(*\). We write \(\sfv_m\) for the vertex labeled \(m\).

\begin{figure}[ht]
\[
\sfT=
\begin{tikzpicture}[baseline=-.5ex]
\foreach \i in {0,120,240} {
\begin{scope}[rotate=\i]
\draw[line width=8] (0,0) -- (90:1);
\draw[line width=8, line cap=round] (90:1) -- ++(60:1) ++(-1,0) -- ++(-60:1);
\end{scope}
}
\foreach \i in {0,120,240} {
\begin{scope}[rotate=\i]
\draw[line width=7, line cap=round, white] (0,0) -- (90:1) -- ++(60:1) ++(-1,0) -- ++(-60:1);
\end{scope}
}
\draw[fill, thick] (0,0) circle (2pt);
\foreach \i in {0,120,240} {
\begin{scope}[rotate=\i]
\draw[lightgray, thick] (0,0) ++(90:1) ++(60:1) -- ++(-1,0);
\draw[fill, thick] (0,0) -- (90:1) circle (2pt) -- ++(60:1) circle (2pt) ++(-1,0) circle (2pt) -- ++(-60:1);
\end{scope}
}
\draw[->] (0,-2) ++(120:1.1) arc (120:60:1.1);
\draw (0,0) node[below=1ex] {\(\sfv_3\)} ++ (-150:1) node[below right] {\(\sfv_2\)} + (-120:1) node[below=1ex] {\(*=\sfv_1\)} +(-1,0) node[left=1ex] {\(\sfv_{10}\)};
\draw (0,0) ++ (-30:1) node[below left] {\(\sfv_4\)} + (-60:1) node[below=1ex] {\(\sfv_5\)} +(1,0) node[right=1ex] {\(\sfv_6\)};
\draw (0,0) ++ (0,1) node[right=1ex] {\(\sfv_7\)} +(60:1) node[above right] {\(\sfv_8\)} +(120:1) node[above left] {\(\sfv_9\)};
\end{tikzpicture}
\]
\caption{A maximal tree \(\sfT\) and labels}\label{Fig:MaximalTreeAndLabels}
\end{figure}
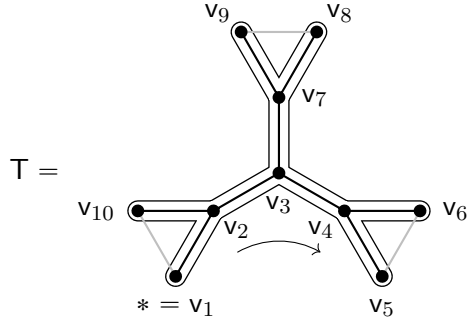

This labeling has the following property: every vertex \(\sfv\neq *\) of \(\Graf\) admits a unique neighbor \(\sfv'\) in \(\sfT\) whose label is smaller than that of \(\sfv\), namely the neighbor of \(\sfv\) lying on the geodesic in \(\sfT\) from \(\sfv\) to \(*\). The assignment \(\sfv\mapsto\sfv'\) may therefore be regarded as a \emph{(discrete) gradient vector field} on \(\Graf\).

This gradient vector field is used to reduce configurations, as follows. Let \(\mathbf x\) be a vertex of \(\UD_n(\Graf)\). Following the terminology of \cite{FS05}, a vertex \(\sfv\in\mathbf x\) is \emph{unblocked} if \(\sfv\neq *\) and \(\sfv'\notin\mathbf x\), and it is \emph{minimal unblocked} if, in addition, its label is the smallest among the labels of the unblocked vertices of \(\mathbf x\). If \(\sfv\) is unblocked, then the closed edge of \(\Graf\) joining \(\sfv\) to \(\sfv'\) meets \(\mathbf x\) only in \(\sfv\), so that
\[\mathbf x\longmapsto\left(\mathbf x\setminus\{\sfv\}\right)\cup\{\sfv'\}\]
is again a vertex of \(\UD_n(\Graf)\), and the two configurations span an edge of \(\UD_n(\Graf)\). We call this move, performed at the minimal unblocked vertex of \(\mathbf x\), the \emph{vertex-reduction} of \(\mathbf x\).

Each vertex-reduction strictly decreases the sum of the labels, so the procedure terminates at a configuration with no unblocked vertex, which we call a \emph{critical \(0\)-cell}.
Every critical \(0\)-cell contains \(*\), since otherwise its smallest-labeled vertex would be unblocked.
% Each vertex-reduction strictly decreases the sum of the labels of the vertices of a configuration, so iterating the procedure terminates after finitely many steps. The resulting configuration has no unblocked vertex; that is, no particle can be moved along the gradient vector field. We call such a configuration a \emph{critical \(0\)-cell}. Every critical \(0\)-cell contains \(*\). Indeed, if \(*\notin\mathbf x\), then the vertex \(\sfv\in\mathbf x\) with the smallest label satisfies \(\sfv\neq *\), and \(\sfv'\) has a strictly smaller label than \(\sfv\) and hence does not lie in \(\mathbf x\); so \(\sfv\) is unblocked.

Although the reduction procedure is deterministic for each initial configuration, its terminal critical \(0\)-cell may depend on that configuration.
For instance, for the tree \(\sfT\) of \Cref{Fig:MaximalTreeAndLabels} with \(n=3\), both \(\{\sfv_1,\sfv_2,\sfv_3\}\) and \(\{\sfv_1,\sfv_2,\sfv_{10}\}\) are critical \(0\)-cells.

\begin{remark}\label{Rem:UniqueCritical0Cell}
The configuration \(\{\sfv_1,\dots,\sfv_n\}\) is always a critical \(0\)-cell, since \(*=\sfv_1\) and, for every \(2\leq m\leq n\), the predecessor \(\sfv_m'\) belongs to \(\{\sfv_1,\dots,\sfv_{m-1}\}\).
If \(\Graf\) is sufficiently subdivided for \(n\) in the sense of \Cref{Def:SufficientlySubdivided}, then the maximal tree \(\sfT\) and the leaf \(*\) can be chosen so that \(\{\sfv_1,\dots,\sfv_n\}\) is the \emph{unique} critical \(0\)-cell.
For \(n=2\), the critical \(0\)-cell is unique for every choice of \(\sfT\) and \(*\): a critical configuration must contain \(*=\sfv_1\) and its unique neighbor \(\sfv_2\) in \(\sfT\).

% For \(n=2\), no such choice is needed: since \(*=\sfv_1\) is a leaf of \(\sfT\), its unique neighbor in \(\sfT\) is \(\sfv_2\), and a critical \(0\)-cell \(\{\sfv_1,\sfv\}\) must satisfy \(\sfv'=\sfv_1\). Hence \(\{\sfv_1,\sfv_2\}\) is the unique critical \(0\)-cell for every choice of \(\sfT\) and \(*\).

Sufficient subdivision is not necessary for the existence of such a choice. For example, let \(\Graf=\sfK_{2,3}\) and let \(\sfT\) be the maximal tree which is a path of length \(4\), with \(*\) one of its two endpoints, as in \Cref{Fig:K23SpanningPath}. 
Since the labels increase along \(\sfT\), a critical \(0\)-cell must consist of the \(n\) vertices with smallest labels. Hence \(\{\sfv_1,\dots,\sfv_n\}\) is the unique critical \(0\)-cell for every \(2\leq n\leq4\). On the other hand, \(\sfK_{2,3}\) is not sufficiently subdivided for \(n=4\).

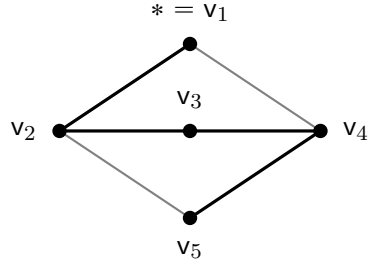
\begin{figure}[ht]
\[
\begin{tikzpicture}[baseline=-.5ex, scale=1.15]
\coordinate (u) at (0,0);
\coordinate (v) at (3,0);
\coordinate (w1) at (1.5,1);
\coordinate (w2) at (1.5,0);
\coordinate (w3) at (1.5,-1);
\draw[gray, thick] (u) -- (w3);
\draw[gray, thick] (v) -- (w1);
\draw[line width=1.2pt] (w1) -- (u) -- (w2) -- (v) -- (w3);
\foreach \i in {u,v,w1,w2,w3} {\draw[fill, thick] (\i) circle (2pt);}
\draw (w1) node[above=1ex] {\(*=\sfv_1\)};
\draw (u) node[left=1ex] {\(\sfv_2\)};
\draw (w2) node[above=1ex] {\(\sfv_3\)};
\draw (v) node[right=1ex] {\(\sfv_4\)};
\draw (w3) node[below=1ex] {\(\sfv_5\)};
\end{tikzpicture}
\]
\caption{A maximal tree \(\sfT\) of \(\sfK_{2,3}\) which is a path of length \(4\) (black), together with the resulting labels. The gray edges belong to \(\sfK_{2,3}\) but not to \(\sfT\).}
\label{Fig:K23SpanningPath}
\end{figure}

We refer to Farley--Sabalka \cite{FS05} for a detailed treatment of the discrete Morse theory underlying this discussion.
\end{remark}

Set \(\mathbf x^0=\{\sfv_1,\dots,\sfv_n\}\in \UD_n(\Graf)\), which is a critical \(0\)-cell by \Cref{Rem:UniqueCritical0Cell}.
For each critical \(0\)-cell \(\mathbf c\) of \(\UD_n(\Graf)\), fix once and for all an edge path
\(\rho_{\mathbf c}\) in \(\UD_n(\Graf)\) from \(\mathbf x^0\) to \(\mathbf c\); we take \(\rho_{\mathbf x^0}\) to be the constant path.

Iterated vertex-reduction for each vertex configuration \(\bfx\in\UD_n(\Graf)\) determines an edge path \(\mu_{\mathbf x}\) from \(\mathbf x\) to a critical \(0\)-cell \(\mathbf c(\mathbf x)\), each edge of which is a single vertex-reduction. 

\begin{definition}[Basepoints and standard paths]\label{Def:BasePointsStandardPaths}
We take \(\bfx^0\) as the global basepoint of \(\UD_n(\Graf)\), and fix a lift \(\widetilde{\bfx}^0\) of \(\bfx^0\) in \(\widetilde{\UD}_n(\Graf)\).

For each vertex configuration \(\bfx\in \UD_n(\Graf)\), we call
\[\gamma_{\bfx}=\rho_{\bfc(\bfx)}\cdot\mu_{\bfx}^{-1}\]
the \emph{standard path} to \(\bfx\); it is an edge path in \(\UD_n(\Graf)\) from \(\bfx^0\) to \(\bfx\). In particular, \(\gamma_{\bfx^0}\) is the constant path.

We denote by \(\widetilde\gamma_{\bfx}\) the lift of \(\gamma_{\bfx}\) starting at \(\widetilde{\bfx}^0\).
\end{definition}

\begin{remark}
The paths \(\rho_\bfc\) fixed above are chosen arbitrarily, but they too can be defined explicitly in terms of the labels induced by \(\sfT\).
\end{remark}

\begin{convention}
Let \((\sfL,\bfn)\) be a legal pair.
\begin{enumerate}[wide]
\item
For each component \(\sfL_i\in\pi_0(\sfL)\), set \(n_i=\bfn(\sfL_i)\).
Let \(\mathbf x^0_{\sfL_i,n_i}\in \UD_{n_i}(\sfL_i)\) be the set of the \(n_i\) vertices of \(\sfL_i\) with smallest labels and let
\[\mathbf x^0_{\sfL,\bfn}=\bigsqcup_{\sfL_i\in \pi_0(\sfL)}\mathbf x^0_{\sfL_i,n_i}\in \UD_{\bfn}(\sfL).\]
We use this configuration as the basepoint when writing
\[\GBGdisc_{n_i}(\sfL_i)=\pi_1\bigl(\UD_{n_i}(\sfL_i),\mathbf x^0_{\sfL_i,n_i}\bigr)\quad\text{and}\quad\GBGdisc_{\bfn}(\sfL)=\pi_1\bigl(\UD_{\bfn}(\sfL),\mathbf x^0_{\sfL,\bfn}\bigr) \cong \prod_{\sfL_i\in \pi_0(\sfL)}\GBGdisc_{n_i}(\sfL_i).\]

\item 
Suppose that \(\|\bfn\|=n\). We write
\[\gamma_{\sfL,\bfn}=\gamma_{\mathbf x^0_{\sfL,\bfn}}.\]
Then the inclusion \(\UD_{\bfn}(\sfL)\hookrightarrow \UD_n(\Graf)\), together with the standard path \(\gamma_{\sfL,\bfn}\), induces an injective homomorphism 
\[\Phi_{\sfL,\bfn}\colon \GBGdisc_{\bfn}(\sfL)\rightarrow\GBGdisc_n(\Graf),\quad\text{defined by}\quad[\alpha]\mapsto [\gamma_{\sfL,\bfn}\cdot\alpha\cdot\gamma_{\sfL,\bfn}^{-1}].\]
We identify \(\GBGdisc_{\bfn}(\sfL)\) with the image of \(\Phi_{\sfL,\bfn}\).

Let \(p\colon\widetilde{\UD}_n(\Graf)\rightarrow \UD_n(\Graf)\) be the universal covering map. Let \(\widetilde\gamma_{\sfL,\bfn}\) be the lift of \(\gamma_{\sfL,\bfn}\) starting at \(\widetilde{\mathbf x}^0\), and denote its endpoint by \(\widetilde{\mathbf x}^0_{\sfL,\bfn}\).
We define \(\langle\sfL,\bfn\rangle\) to be the connected component of \(p^{-1}(\UD_{\bfn}(\sfL))\) containing \(\widetilde{\mathbf x}^0_{\sfL,\bfn}\). Then every lift of \(\UD_{\bfn}(\sfL)\) is of the form \(g\cdot\langle\sfL,\bfn\rangle\) for some \(g\in\GBGdisc_n(\Graf)\).
\end{enumerate}
\end{convention}

\begin{remark}\label{Rem:AuxiliaryChoicesLegalPairs}
The global base configuration, standard paths, distinguished lifts, and subgroup representatives fixed above are auxiliary. After identifying fundamental groups by a change-of-basepoint isomorphism, changing these choices only changes distinguished lifts by deck translations and subgroup representatives by conjugation. The factor system and its parallelism classes are independent of these choices.
\end{remark}

\begin{example}[Legal pairs and their base configurations]\label{Ex:LegalPairBasePoints}
Let \(\Graf\) be the graph of \Cref{Fig:MaximalTreeAndLabels}, with the labels fixed there, and let \(n=2\).
The graph \(\Graf\) has exactly two embedded cycles which do not contain \(\sfv_1\), namely the triangles \(\ell_1\) on \(\sfv_4,\sfv_5,\sfv_6\) and \(\ell_2\) on \(\sfv_7,\sfv_8,\sfv_9\). For such a cycle \(\ell\), the complement \(\Graf\setminus\ell\) is connected, and we may take the subgraph \(\ell\sqcup(\Graf\setminus\ell)\) with one particle on each of its two components. Let \((\sfL,\bfn)\) and \((\sfL',\bfn')\) be the pairs obtained in this way from \(\ell_1\) and \(\ell_2\), respectively. Neither component is a single vertex and each carries fewer particles than it has vertices, so both pairs are legal of norm \(2\).

Following the convention above, the base configuration of each pair consists of the vertex of smallest label in each component. Hence
\[\bfx^0_{\sfL,\bfn}=\{\sfv_1,\sfv_4\}\qquad\text{and}\qquad \bfx^0_{\sfL',\bfn'}=\{\sfv_1,\sfv_7\},\]
as shown in \Cref{Fig:LegalPairBasePoints}.
\end{example}

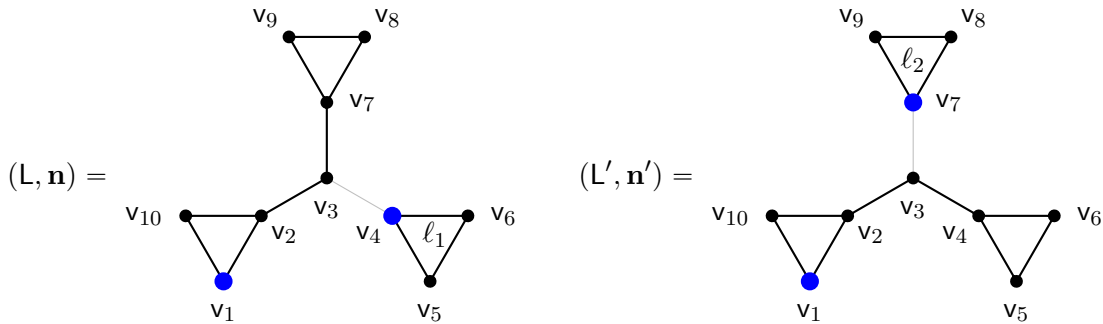
\begin{figure}[ht]
\begin{align*}
(\sfL,\bfn)&=
\begin{tikzpicture}[baseline=-.5ex]
\begin{scope}[rotate=-120]
\draw[lightgray] (0,0) -- (0,1);
\draw[fill, thick] (0,0) circle (2pt);
\draw[thick] (0,0) ++(90:1) ++(60:1) -- ++(-1,0) node[midway,above left=-1ex] {\(\ell_1\)};
\draw[fill, thick] (90:1) circle (2pt) -- ++(60:1) circle (2pt) ++(-1,0) circle (2pt) -- ++(-60:1);
\foreach \i in {120,240} {
\begin{scope}[rotate=\i]
\draw[thick] (0,0) ++(90:1) ++(60:1) -- ++(-1,0);
\draw[fill, thick] (0,0) -- (90:1) circle (2pt) -- ++(60:1) circle (2pt) ++(-1,0) circle (2pt) -- ++(-60:1);
\end{scope}
}
\end{scope}
\draw[thick,fill, blue] (0,0) ++ (-150:1) ++(-120:1) circle (3pt);
\draw[thick,fill, blue] (0,0) ++ (-30:1) circle (3pt);
\draw (0,0) node[below=1ex] {\(\sfv_3\)} ++ (-150:1) node[below right] {\(\sfv_2\)} + (-120:1) node[below=1ex] {\(\sfv_1\)} +(-1,0) node[left=1ex] {\(\sfv_{10}\)};
\draw (0,0) ++ (-30:1) node[below left] {\(\sfv_4\)} + (-60:1) node[below=1ex] {\(\sfv_5\)} +(1,0) node[right=1ex] {\(\sfv_6\)};
\draw (0,0) ++ (0,1) node[right=1ex] {\(\sfv_7\)} +(60:1) node[above right] {\(\sfv_8\)} +(120:1) node[above left] {\(\sfv_9\)};
\end{tikzpicture}&
(\sfL',\bfn')&=
\begin{tikzpicture}[baseline=-.5ex]
\draw[lightgray] (0,0) -- (0,1);
\draw[fill, thick] (0,0) circle (2pt);
\draw[thick] (0,0) ++(90:1) ++(60:1) -- ++(-1,0) node[midway,below] {\(\ell_2\)};
\draw[fill, thick] (90:1) circle (2pt) -- ++(60:1) circle (2pt) ++(-1,0) circle (2pt) -- ++(-60:1);
\foreach \i in {120,240} {
\begin{scope}[rotate=\i]
\draw[thick] (0,0) ++(90:1) ++(60:1) -- ++(-1,0);
\draw[fill, thick] (0,0) -- (90:1) circle (2pt) -- ++(60:1) circle (2pt) ++(-1,0) circle (2pt) -- ++(-60:1);
\end{scope}
}
\draw[thick,fill, blue] (0,0) ++ (-150:1) ++(-120:1) circle (3pt);
\draw[thick,fill, blue] (0,0) ++ (90:1) circle (3pt);
\draw (0,0) node[below=1ex] {\(\sfv_3\)} ++ (-150:1) node[below right] {\(\sfv_2\)} + (-120:1) node[below=1ex] {\(\sfv_1\)} +(-1,0) node[left=1ex] {\(\sfv_{10}\)};
\draw (0,0) ++ (-30:1) node[below left] {\(\sfv_4\)} + (-60:1) node[below=1ex] {\(\sfv_5\)} +(1,0) node[right=1ex] {\(\sfv_6\)};
\draw (0,0) ++ (0,1) node[right=1ex] {\(\sfv_7\)} +(60:1) node[above right] {\(\sfv_8\)} +(120:1) node[above left] {\(\sfv_9\)};
\end{tikzpicture}
\end{align*}
\caption{The legal pairs and chosen points \(\bfx^0_{\sfL,\bfn}\) and \(\bfx^0_{\sfL',\bfn'}\)}
\label{Fig:LegalPairBasePoints}
\end{figure}

\subsection{Domains and their relations}

We now describe the domains of the legal-pair hierarchy, together with their nesting and orthogonality relations. The domains will be represented by left cosets of discrete graph braid subgroups embedded using the standard paths fixed in \Cref{Subsection:StandardPaths}.

\begin{corollary}\label{Cor:FactorSystemElements}
The full legal-pair factor system \(\frFLP^0\) consists precisely of the legal-pair lifts, i.e., the subcomplexes \(g\cdot\langle\sfL,\bfn\rangle\), where \((\sfL,\bfn)\) is a legal pair with \(\|\bfn\|=n\) and \(g\in\GBGdisc_n(\Graf)\). 
The reduced legal-pair factor system \(\frFLP\) defining the legal-pair hierarchy consists precisely of those subcomplexes \(g\cdot\langle\sfL,\bfn\rangle\) for which \(\sfL\) has at least one non-singleton component.
\end{corollary}
\begin{proof}
Every lift of a legal-pair subcomplex is a deck translate of its distinguished lift. Moreover, \(\UD_{\bfn}(\sfL)\) is a point if and only if every component of \(\sfL\) is a singleton, by \Cref{Rem:ConfigurationSpaceDegenerate} and \eqref{Eq:ProductofGBG}.
The assertions follow from the definitions of \(\frFLP^0\) and \(\frFLP\).
\end{proof}

We next separate the part of a legal pair which determines the varying configuration-space factor from the part which records the choice of a particular parallel fiber.

\begin{definition}[Constrained and free components]\label{Def:ConstrainedandFree}
Let \((\sfL,\bfn)\) be a legal pair with \(\|\bfn\|=n\) such that \(\sfL\) has at least one non-singleton component, and let \(\sfL_i\) denote the connected components of \(\sfL\).
A component \(\sfL_i\) of \(\sfL\) is called \emph{constrained} if one of the following holds:
\begin{enumerate}
\item \(\sfL_i\) is not a single vertex;
\item \(\sfL_i\) is a single vertex, and \(V(\sfD)\subseteq V(\sfL)\), where \(\sfD\) is the connected component containing \(\sfL_i\) in the subgraph of \(\Graf\) obtained by removing all non-singleton components of \(\sfL\).
\end{enumerate}
A component which is not constrained is called \emph{free}.

We write \(\sfL^{\con}\) and \(\sfL^{\free}\) for the unions of the constrained and of the free components of \(\sfL\), respectively, and set
\[\bfn^{\con}=\bfn\big|_{\pi_0(\sfL^{\con})},\qquad
\bfn^{\free}=\bfn\big|_{\pi_0(\sfL^{\free})}.\]
See \Cref{figure:constrained_components}.
\end{definition}

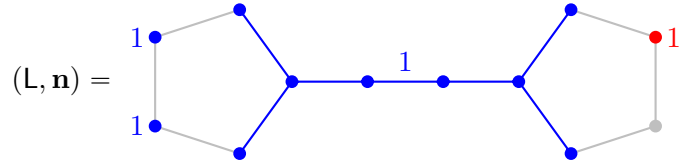
\begin{figure}[ht]
\[
(\sfL,\bfn)=\begin{tikzpicture}[baseline=-.5ex]
\draw[thick,fill,blue] (0,0) circle (2pt) -- (1,0) circle (2pt) -- node[midway,above] {\(1\)} (2,0) circle (2pt) -- (3,0) circle (2pt);
\draw[thick,lightgray] (4,0) + (108:1) -- +(36:1) -- +(-36:1) -- +(-108:1);
\draw[thick,fill,lightgray] (4,0) +(-36:1) circle (2pt);
\draw[thick,fill,red] (4,0) + (36:1) circle (2pt) node[right] {\(1\)};
\draw[thick,fill,blue] (4,0) ++ (108:1) circle (2pt) -- (3,0) (4,0) ++ (-108:1) circle (2pt) -- (3,0);
\draw[thick,lightgray] (-1,0) +(72:1) -- +(144:1) -- +(216:1) -- +(288:1);
\draw[thick,fill,blue] (-1,0) + (-72:1) circle (2pt) -- (0,0) (-1,0) +(1,0) -- +(72:1) circle (2pt);
\draw[thick,fill,blue] (-1,0) +(144:1) circle (2pt) node[left] {\(1\)} +(216:1) circle (2pt) node[left] {\(1\)};
\end{tikzpicture}
\]
\caption{Constrained (blue) and free (red) components of \((\sfL,\bfn)\).}
\label{figure:constrained_components}
\end{figure}

With this terminology, parallelism of legal-pair lifts can be described by moving only the free vertices in the complement of the constrained part.

\begin{lemma}\label{lemma:parallelism}
Let $F=g\cdot\langle\sfL,\bfn\rangle$ and $F'=g'\cdot\langle\sfL',\bfn'\rangle$ be elements of \(\frFLP\). Then \(F\) and \(F'\) are parallel if and only if the following conditions hold:
\begin{enumerate}[label=\textup{(\arabic*)}, ref=\arabic*]
\item \((\sfL^{\con},\bfn^{\con})=(\sfL'^{\con},\bfn'^{\con})\).
\item
There exists an edge path \(\eta\subset \UD_{\|\bfn^{\free}\|}(\Graf\setminus\sfL^{\con})\) from \(V(\sfL^{\free})\) to \(V(\sfL'^{\free})\) such that
\[g^{-1}g'\in \GBGdisc_{\bfn}(\sfL) \left[\gamma_{\sfL,\bfn}\cdot\bar\eta\cdot\gamma_{\sfL',\bfn'}^{-1} \right]\GBGdisc_{\bfn'}(\sfL')\]
inside \(\GBGdisc_n(\Graf)\), where
\(\bar\eta=\eta\times \mathbf x^0_{\sfL^{\con},\bfn^{\con}}\subset \UD_{\|\bfn^{\free}\|}(\Graf\setminus\sfL^{\con})\times \UD_{\bfn^{\con}}(\sfL^{\con}) \subset \UD_n(\Graf)\).
\end{enumerate}
\end{lemma}

\begin{proof}
Suppose first that the two conditions hold. Write \(\eta=\eta_1\cdots\eta_k\) as a concatenation of edges in \(\UD_{\|\bfn^{\free}\|}(\Graf\setminus\sfL^{\con})\).
Each edge \(\eta_i\) moves one free particle across an edge of \(\Graf\setminus\sfL^{\con}\), while leaving the constrained part fixed.
By \Cref{Lem:ParallelTransportLegalPair}, each such move determines a parallelism strip between the corresponding legal-pair subcomplexes. Concatenating these strips gives a parallelism between \(\UD_{\bfn}(\sfL)\) and \(\UD_{\bfn'}(\sfL')\).
The double-coset condition in \textup{(2)} ensures that the lift of this parallelism has boundary components precisely \(F\) and \(F'\).
Therefore \(F\) and \(F'\) are parallel.

Conversely, suppose that \(F\) and \(F'\) are parallel. Then they are joined by a sequence of adjacent parallelism strips. Projecting to \(\UD_n(\Graf)\), each strip leaves every non-singleton component unchanged. Moreover, a singleton lying in a constrained complementary component cannot be transported, since that component is completely occupied by singleton particles. Thus only free singleton particles move across these strips, and the constrained parts agree.
Hence the successive slides determine a path \(\eta\subset \UD_{\|\bfn^{\free}\|}(\Graf\setminus\sfL^{\con})\) from \(V(\sfL^{\free})\) to \(V(\sfL'^{\free})\). Comparing the distinguished lifts using the standard paths gives precisely the double-coset relation in \textup{(2)}.
\end{proof}

The parallelism criterion separates the relevant data: the constrained part determines the varying configuration-space factor, the free vertices determine a particular fiber, and the double-coset condition records the position of its lift. We now define the configuration space parametrizing the free-particle motion.

\begin{definition}[Closure and orthogonal pair]\label{Def:ClosureOrthogonalRegion}
Let \((\sfL,\bfn)\) be a legal pair with \(\|\bfn\|=n\) such that \(\sfL\) has at least one non-singleton component.
Let \(\sfL^\perp\) be the union of all connected components of \(\Graf\setminus\sfL^{\con}\) that intersect \(\sfL^{\free}\). 
We call \(\sfL^\perp\) the \emph{orthogonal subgraph} of \((\sfL,\bfn)\), and set \([\sfL]=\sfL^{\con}\sqcup\sfL^\perp\).

Since \(\pi_0([\sfL])=\pi_0(\sfL^{\con})\sqcup\pi_0(\sfL^{\perp})\), we may define
\([\bfn]\colon\pi_0([\sfL])\rightarrow\NZ\) by
\[[\bfn]\big|_{\pi_0(\sfL^{\con})}=\bfn^{\con}
\qquad\text{and}\qquad
[\bfn](\sfL^{\perp}_j)=\#\bigl(V(\sfL^{\perp}_j)\cap V(\sfL^{\free})\bigr),\]
where \(\sfL^{\perp}_j\) ranges over the connected components of \(\sfL^{\perp}\), and write \(\bfn^{\perp}=[\bfn]\big|_{\pi_0(\sfL^{\perp})}\).
Then we call
\begin{itemize}
\item the legal pair \(([\sfL],[\bfn])\) the \emph{closure} of \((\sfL,\bfn)\), and
\item the legal pair \((\sfL^\perp,\bfn^\perp)\) the \emph{orthogonal pair} of \((\sfL,\bfn)\).
\end{itemize}
We call \(\UD_{[\bfn]}([\sfL])\) the \emph{closure subcomplex} and \(\UD_{\bfn^\perp}(\sfL^\perp)\) the \emph{orthogonal configuration space} associated to \((\sfL,\bfn)\).
\end{definition}

\Cref{figure:orthogonal_closure} illustrates the closure \(([\sfL],[\bfn])\) and the orthogonal pair \((\sfL^\perp,\bfn^\perp)\).

The two pairs in \Cref{Def:ClosureOrthogonalRegion} are legal. Indeed, every component of \(\sfL^\perp\) carries at least one particle by construction, while a non-singleton component cannot be completely occupied by free vertices: otherwise all of its vertices would be constrained by \Cref{Def:ConstrainedandFree}. Moreover, 
\[\|[\bfn]\|=\|\bfn^{\con}\|+\|\bfn^{\free}\|=n,
\qquad\|\bfn^\perp\|=\|\bfn^{\free}\|.\]
Consequently, the closure determines a unique legal-pair subcomplex \(\UD_{[\bfn]}([\sfL])\subseteq\UD_n(\Graf)\). By contrast, the orthogonal configuration space \(\UD_{\bfn^\perp}(\sfL^\perp)\) is regarded as a cube complex in its own right, parametrizing the motion of the free particles. An embedding of this cube complex into \(\UD_n(\Graf)\) depends on the choice of a witness and need not be unique.
If \(\sfL^{\free}=\varnothing\), then the orthogonal pair is the empty legal pair \((\sfL^\perp,\bfn^\perp)=(\varnothing,\varnothing)\), and its orthogonal configuration space is the point \(\UD_0(\varnothing)\).

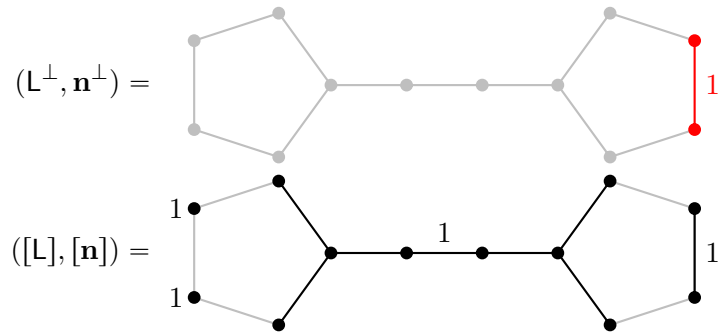
\begin{figure}[ht]
\begin{align*}
(\sfL^\perp,\bfn^\perp)&=\begin{tikzpicture}[baseline=-.5ex]
\draw[thick,fill,lightgray] (0,0) circle (2pt) -- (1,0) circle (2pt) -- (2,0) circle (2pt) -- (3,0) circle (2pt);
\draw[thick,lightgray] (4,0) + (108:1) -- +(36:1) +(-36:1) -- +(-108:1);
\draw[thick,fill,red] (4,0) + (36:1) circle (2pt) -- node[midway,right] {\(1\)} +(-36:1) circle (2pt);
\draw[thick,fill,lightgray] (4,0) ++ (108:1) circle (2pt) -- (3,0) (4,0) ++ (-108:1) circle (2pt) -- (3,0);
\draw[thick,lightgray] (-1,0) +(72:1) -- +(144:1) -- +(216:1) -- +(288:1);
\draw[thick,fill,lightgray] (-1,0) + (-72:1) circle (2pt) -- (0,0) (-1,0) +(1,0) -- +(72:1) circle (2pt);
\draw[thick,fill,lightgray] (-1,0) +(144:1) circle (2pt) node[left] {\(\hphantom{1}\)} +(216:1) circle (2pt) node[left] {\(\hphantom{1}\)};
\end{tikzpicture}\\
([\sfL],[\bfn])&=\begin{tikzpicture}[baseline=-.5ex]
\draw[thick,fill] (0,0) circle (2pt) -- (1,0) circle (2pt) -- node[midway,above] {\(1\)} (2,0) circle (2pt) -- (3,0) circle (2pt);
\draw[thick,lightgray] (4,0) + (108:1) -- +(36:1) -- +(-36:1) -- +(-108:1);
\draw[thick,fill] (4,0) + (36:1) circle (2pt) -- node[midway,right] {\(1\)} +(-36:1) circle (2pt);
\draw[thick,fill] (4,0) ++ (108:1) circle (2pt) -- (3,0) (4,0) ++ (-108:1) circle (2pt) -- (3,0);
\draw[thick,lightgray] (-1,0) +(72:1) -- +(144:1) -- +(216:1) -- +(288:1);
\draw[thick,fill] (-1,0) + (-72:1) circle (2pt) -- (0,0) (-1,0) +(1,0) -- +(72:1) circle (2pt);
\draw[thick,fill] (-1,0) +(144:1) circle (2pt) node[left] {\(1\)} +(216:1) circle (2pt) node[left] {\(1\)};
\end{tikzpicture}
\end{align*}
\caption{The orthogonal pair \((\sfL^\perp,\bfn^\perp)\) (top) and the closure \(([\sfL],[\bfn])\) (bottom) of the legal pair \((\sfL,\bfn)\) from \Cref{figure:constrained_components}.}
\label{figure:orthogonal_closure}
\end{figure}

The closure subcomplex admits the product decomposition
\begin{equation}\label{Eq:ProductDecomposition}
\UD_{[\bfn]}([\sfL])=\UD_{\bfn^{\con}}(\sfL^{\con})\times\UD_{\bfn^\perp}(\sfL^\perp).    
\end{equation}
Under this decomposition, the original subcomplex \(\UD_{\bfn}(\sfL)\) is the first-factor fiber \(\UD_{\bfn^{\con}}(\sfL^{\con})\times\{V(\sfL^{\free})\}\).
Thus the closure is obtained by allowing the free particles to vary in the orthogonal configuration space. To choose coherent representatives of the resulting parallelism classes, we now select a preferred free configuration in each such space.

\begin{definition}[Standard representative of a legal pair]\label{Def:StandardRepresentativeLegalPair}
Let \((\sfL,\bfn)\) be a legal pair with \(\|\bfn\|=n\) such that \(\sfL\) has at least one non-singleton component, and let \(([\sfL],[\bfn])\) be its closure. Since \(V([\sfL])=V(\sfL^{\con})\sqcup V(\sfL^{\perp})\), the base configuration of the closure decomposes as
\[\bfx^0_{[\sfL],[\bfn]}=\bfx^0_{\sfL^{\con},\bfn^{\con}}\sqcup \bfz_{\sfL,\bfn},
\qquad\text{where}\quad
\bfz_{\sfL,\bfn}:=\bfx^0_{[\sfL],[\bfn]}\cap V(\sfL^{\perp})\in \UD_{\bfn^\perp}(\sfL^\perp).\]
Explicitly, \(\bfz_{\sfL,\bfn}\) consists, in each connected component \(\sfL^\perp_i\) of \(\sfL^\perp\), of the \(\bfn^\perp(\sfL^\perp_i)\) vertices with smallest labels.
Then the \emph{standard representative} of \((\sfL,\bfn)\) is the legal pair \((\sfL^{\mathrm{st}},\bfn^{\mathrm{st}})\) defined by
\[\sfL^{\mathrm{st}}=\sfL^{\con}\sqcup\bfz_{\sfL,\bfn},\qquad
\bfn^{\mathrm{st}}=\bfn^\con\sqcup \bfone_{\bfz_{\sfL,\bfn}},\]
where the vertices of \(\bfz_{\sfL,\bfn}\) are regarded as singleton components of \(\sfL^{\mathrm{st}}\). 

Such a legal pair \((\sfL,\bfn)\) is called \emph{standard} if it agrees with its standard representative, or equivalently if \(\bfx^0_{\sfL,\bfn}=\bfx^0_{[\sfL],[\bfn]}\).
\end{definition}

The two conditions in \Cref{Def:StandardRepresentativeLegalPair} are equivalent. Indeed, every free component is a single vertex carrying one particle, so that \(\bfx^0_{\sfL,\bfn}=\bfx^0_{\sfL^{\con},\bfn^{\con}}\sqcup V(\sfL^{\free})\), while \(\bfx^0_{[\sfL],[\bfn]}=\bfx^0_{\sfL^{\con},\bfn^{\con}}\sqcup\bfz_{\sfL,\bfn}\) as displayed above. Since \(\sfL^{\con}\) and \(\sfL^{\perp}\) are disjoint, the two configurations agree precisely when \(V(\sfL^{\free})=\bfz_{\sfL,\bfn}\), that is, precisely when \((\sfL,\bfn)=(\sfL^{\mathrm{st}},\bfn^{\mathrm{st}})\).

\begin{lemma}[Standard representatives of parallelism classes]
\label{Lem:StandardRepresentativeParallelismClass}
Let \((\sfL,\bfn)\) be a legal pair with \(\|\bfn\|=n\) such that \(\langle\sfL,\bfn\rangle\in\frFLP\), and let \((\sfL^{\mathrm{st}},\bfn^{\mathrm{st}})\) be its standard representative.
Then there exists an element \(\theta_{\sfL,\bfn}\in\GBGdisc_n(\Graf)\) such that
\[\langle\sfL,\bfn\rangle\parallel\theta_{\sfL,\bfn}\cdot\langle\sfL^{\mathrm{st}},\bfn^{\mathrm{st}}\rangle.\]
Consequently, for every \(g\in\GBGdisc_n(\Graf)\), the parallelism class of \(g\cdot\langle\sfL,\bfn\rangle\) is represented by the left coset \(g\theta_{\sfL,\bfn}\GBGdisc_{\bfn^{\mathrm{st}}}(\sfL^{\mathrm{st}})\).
\end{lemma}

\begin{proof}
By construction, \((\sfL,\bfn)\) and \((\sfL^{\mathrm{st}},\bfn^{\mathrm{st}})\) have the same constrained part, and their free configurations lie in the same orthogonal configuration space \(\UD_{\bfn^\perp}(\sfL^\perp)\).
Since the two legal pairs have the same constrained part and \(\langle\sfL,\bfn\rangle\in\frFLP\), both distinguished lifts are non-singletons. Hence both belong to \(\frFLP\), and \Cref{lemma:parallelism} applies.
Choose an edge path in \(\UD_{\bfn^\perp}(\sfL^\perp)\) from \(V(\sfL^{\free})\) to \(\bfz_{\sfL,\bfn}\), and denote by \(\eta_{\sfL,\bfn}\) the corresponding path in the closure subcomplex \(\UD_{[\bfn]}([\sfL])\). Define
\[\theta_{\sfL,\bfn}=\left[\gamma_{\sfL,\bfn}\cdot\eta_{\sfL,\bfn}\cdot\gamma_{\sfL^{\mathrm{st}},\bfn^{\mathrm{st}}}^{-1}\right]\in\GBGdisc_n(\Graf).\]
Then \Cref{lemma:parallelism} implies that \(\langle\sfL,\bfn\rangle\parallel\theta_{\sfL,\bfn}\cdot\langle\sfL^{\mathrm{st}},\bfn^{\mathrm{st}}\rangle\). Multiplying by \(g\) gives the statement for \(g\cdot\langle\sfL,\bfn\rangle\).

Finally, the stabilizer of \(\langle\sfL^{\mathrm{st}},\bfn^{\mathrm{st}}\rangle\) in the deck group is precisely \(\GBGdisc_{\bfn^{\mathrm{st}}}(\sfL^{\mathrm{st}})\), by the definition of the subgroup using the standard path. Hence the parallelism class is represented by the stated left coset.
\end{proof}

\begin{lemma}[Closure subcomplexes and standard product regions]\label{Lem:LiftedClosureRegion}
Let \((\sfL,\bfn)\) be a standard legal pair such that \(F:=\langle\sfL,\bfn\rangle\in\frFLP\). Let \(g\in G\), and let \(U=[g\cdot F]\in\frSLP\). Then the standard product region associated to \(U\) is the convex product subcomplex
\[\mathbf P_U=g\cdot\bigl\langle[\sfL],[\bfn]\bigr\rangle
\cong(g\cdot F)\times\widetilde{\UD}_{\bfn^\perp}(\sfL^\perp),\]
where \(\widetilde{\UD}_{\bfn^\perp}(\sfL^\perp)\) denotes the universal cover of the orthogonal configuration space \(\UD_{\bfn^\perp}(\sfL^\perp)\).

Under this decomposition, the first-factor fibers
\[(g\cdot F)\times\{e\},\qquad e\in \bigl(\widetilde{\UD}_{\bfn^\perp}(\sfL^\perp)\bigr)^{(0)},\]
are precisely the factors representing \(U\). Consequently, \(\mathbf P_U\) is the convex hull of the factors representing \(U\). In particular, \(g\cdot\langle[\sfL],[\bfn]\rangle\) depends only on \(U\) and is preserved by \(\Stab_G(U)\).
\end{lemma}

\begin{proof}
By \(G\)-equivariance, it suffices to consider the case \(g=1\), so that \(U=[F]\). Since \((\sfL,\bfn)\) is standard, \(\bfx^0_{\sfL,\bfn}=\bfx^0_{[\sfL],[\bfn]}\), and hence \(F\subseteq\langle[\sfL],[\bfn]\rangle\).
The product decomposition \eqref{Eq:ProductDecomposition}, together with \Cref{Lem:LocalisometrybetweenLegalPairs}, lifts to
\[\bigl\langle[\sfL],[\bfn]\bigr\rangle\cong F\times\widetilde{\UD}_{\bfn^\perp}(\sfL^\perp).\]

By \Cref{lemma:parallelism}, the first-factor fibers in this product are precisely the factors in \(\frFLP\) parallel to \(F\). Hence, by the cubical factor-system model recalled in \Cref{Subsection:HHG},
\[\mathbf P_U=\bigl\langle[\sfL],[\bfn]\bigr\rangle.\]
It is therefore the convex hull of the factors representing \(U\), and is preserved by \(\Stab_G(U)\). Translating by \(g\) gives the result.
\end{proof}

In the cubical factor-system model recalled in \Cref{Subsection:HHG}, the factor \(g\cdot F\) is a cubical model for the standard nesting factor \(\mathbf F_U\), while \(\widetilde{\UD}_{\bfn^\perp}(\sfL^\perp)\) is a cubical model for the standard orthogonality factor \(\mathbf E_U\). Thus the coarse HHS product \(\mathbf P_U\asymp\mathbf F_U\times\mathbf E_U\) is realized here by an exact cubical product decomposition.

We now formulate nesting and orthogonality using standard representatives. The path condition records the relative position of the chosen deck translates.

\begin{definition}[Compatible standard legal pairs and comparison paths]\label{Def:CompatibleLegalPairs}
Let \((\sfL,\bfn)\) and \((\sfL',\bfn')\) be standard legal pairs with \(\|\bfn\|=\|\bfn'\|=n\).
We say that the two standard legal pairs are \emph{compatible} if their closures intersect, that is,
\[\UD_{[\bfn]}([\sfL])\cap\UD_{[\bfn']}([\sfL'])\neq\varnothing.\]

Suppose that \((\sfL,\bfn)\) and \((\sfL',\bfn')\) are compatible, and choose a vertex \(\mathbf y\in\UD_{[\bfn]}([\sfL])\cap\UD_{[\bfn']}([\sfL'])\). Since each closure subcomplex is connected and contains the corresponding standard base configuration, there are edge paths
\begin{itemize}
\item \(\eta_{\sfL}\subset\UD_{[\bfn]}([\sfL])\) from \(\mathbf x^0_{\sfL,\bfn}\) to \(\mathbf y\), and
\item \(\eta_{\sfL'}\subset\UD_{[\bfn']}([\sfL'])\) from \(\mathbf x^0_{\sfL',\bfn'}\) to \(\mathbf y\).
\end{itemize}
The concatenation \(\eta=\eta_{\sfL}\cdot\eta_{\sfL'}^{-1}\) is an edge path from \(\mathbf x^0_{\sfL,\bfn}\) to \(\mathbf x^0_{\sfL',\bfn'}\), obtained by first following a path in \(\UD_{[\bfn]}([\sfL])\) to \(\mathbf y\) and then a path in \(\UD_{[\bfn']}([\sfL'])\) from \(\mathbf y\).
We call any path obtained in this way a \emph{comparison path} from \((\sfL,\bfn)\) to \((\sfL',\bfn')\).
For such a comparison path \(\eta\), set
\begin{equation}\label{Eq:Delta_eta}
\delta_\eta=\left[\gamma_{\sfL,\bfn}\cdot\eta\cdot\gamma_{\sfL',\bfn'}^{-1}\right]\in G.
\end{equation}
\end{definition}

No independence from the choice of comparison path is asserted: the lifting criterion below only requires the existence of one comparison path satisfying the stated double-coset condition.

\begin{example}[A comparison path]\label{Ex:ComparisonPath}
We continue with the legal pairs \((\sfL,\bfn)\) and \((\sfL',\bfn')\) of \Cref{Ex:LegalPairBasePoints}. Both are standard: each of their components is constrained, so \(\sfL^{\free}=\sfL'^{\free}=\varnothing\), and hence \([\sfL]=\sfL\) and \([\sfL']=\sfL'\). Therefore
\[\UD_{[\bfn]}([\sfL])=\UD_1(\ell_1)\times\UD_1(\Graf\setminus\ell_1)\qquad\text{and}\qquad \UD_{[\bfn']}([\sfL'])=\UD_1(\ell_2)\times\UD_1(\Graf\setminus\ell_2),\]
that is, \(\UD_{[\bfn]}([\sfL])\) consists of the configurations with one particle on \(\ell_1\) and one particle off \(\ell_1\), and similarly for \(\UD_{[\bfn']}([\sfL'])\).

The configuration \(\mathbf y=\{\sfv_4,\sfv_7\}\) lies in both: indeed \(\sfv_4\in\ell_1\) and \(\sfv_7\in\Graf\setminus\ell_1\), while \(\sfv_7\in\ell_2\) and \(\sfv_4\in\Graf\setminus\ell_2\). Thus \((\sfL,\bfn)\) and \((\sfL',\bfn')\) are compatible.

The comparison path \(\eta=\eta_{\sfL}\cdot\eta_{\sfL'}^{-1}\) through \(\mathbf y\) drawn in \Cref{Fig:ComparisonPath} is given as follows.
\begin{itemize}
\item \(\eta_{\sfL}\) runs from \(\bfx^0_{\sfL,\bfn}=\{\sfv_1,\sfv_4\}\) to \(\mathbf y\); it keeps the particle on \(\ell_1\) fixed at \(\sfv_4\) and moves the other particle from \(\sfv_1\) to \(\sfv_7\) along the edge path with vertex sequence \((\sfv_1,\sfv_2,\sfv_3,\sfv_7)\).
\item \(\eta_{\sfL'}\) runs from \(\bfx^0_{\sfL',\bfn'}=\{\sfv_1,\sfv_7\}\) to \(\mathbf y\); it keeps the particle on \(\ell_2\) fixed at \(\sfv_7\) and moves the other particle from \(\sfv_1\) to \(\sfv_4\) along the edge path with vertex sequence \((\sfv_1,\sfv_2,\sfv_3,\sfv_4)\).
\end{itemize}
The moving particle stays off \(\ell_1\) along \(\eta_{\sfL}\) and off \(\ell_2\) along \(\eta_{\sfL'}\), so the two paths do lie in the respective closure subcomplexes.
\end{example}

\begin{figure}[ht]
\newcommand{\CPgraph}{%
\draw[thick] (-0.866,-0.5) -- (0,0) -- (0.866,-0.5);
\draw[thick] (0,0) -- (0,1);
\draw[thick] (-1.366,-1.366) -- (-0.866,-0.5) -- (-1.866,-0.5) -- (-1.366,-1.366);
\draw[thick] (1.366,-1.366) -- (0.866,-0.5) -- (1.866,-0.5) -- (1.366,-1.366);
\draw[thick] (-0.5,1.866) -- (0,1) -- (0.5,1.866) -- (-0.5,1.866);
\foreach \p in {{(0,0)},{(-0.866,-0.5)},{(-1.366,-1.366)},{(-1.866,-0.5)},{(0.866,-0.5)},{(1.366,-1.366)},{(1.866,-0.5)},{(0,1)},{(0.5,1.866)},{(-0.5,1.866)}}
{\fill \p circle (2pt);}
\draw (0,0) node[above right=-0.2ex] {\(\sfv_3\)};
\draw (-0.866,-0.5) node[above left=-0.4ex] {\(\sfv_2\)};
\draw (-1.366,-1.366) node[below=0.7ex] {\(\sfv_1\)};
\draw (-1.866,-0.5) node[left=0.7ex] {\(\sfv_{10}\)};
\draw (0.866,-0.5) node[above right=-0.2ex] {\(\sfv_4\)};
\draw (1.366,-1.366) node[below=0.7ex] {\(\sfv_5\)};
\draw (1.866,-0.5) node[right=0.7ex] {\(\sfv_6\)};
\draw (0,1) node[left=0.8ex] {\(\sfv_7\)};
\draw (0.5,1.866) node[right=0.7ex] {\(\sfv_8\)};
\draw (-0.5,1.866) node[left=0.7ex] {\(\sfv_9\)};
}
\begin{align*}
\eta_{\sfL}&=
\begin{tikzpicture}[baseline=-.5ex]
\CPgraph
\draw (1.366,-0.86) node {\(\ell_1\)};
\draw[blue, fill=blue] (0.866,-0.5) circle (3pt);
\draw[red, line width=1.4pt, ->] (-1.366,-1.366) -- (-0.866,-0.5) -- (0,0) -- (0,0.86);
\draw[red, fill=red, thick] (-1.366,-1.366) circle (3pt);
\draw[red, fill=red, thick] (0,1) circle (3pt);
\end{tikzpicture}&
\eta_{\sfL'}&=
\begin{tikzpicture}[baseline=-.5ex]
\CPgraph
\draw (0,1.62) node {\(\ell_2\)};
\draw[blue, fill=blue] (0,1) circle (3.4pt);
\draw[red, line width=1.4pt, ->] (-1.366,-1.366) -- (-0.866,-0.5) -- (0,0) -- (0.74,-0.43);
\draw[red, fill=red, thick] (-1.366,-1.366) circle (3pt);
\draw[red, fill=red, thick] (0.866,-0.5) circle (3pt);
\end{tikzpicture}
\end{align*}
\caption{The two halves of the comparison path of \Cref{Ex:ComparisonPath}. In each picture the blue particle is the one fixed on the constrained cycle, and the red arrow is the track of the moving particle, from \(\sfv_1\) to the other vertex of \(\mathbf y=\{\sfv_4,\sfv_7\}\).}
\label{Fig:ComparisonPath}
\end{figure}
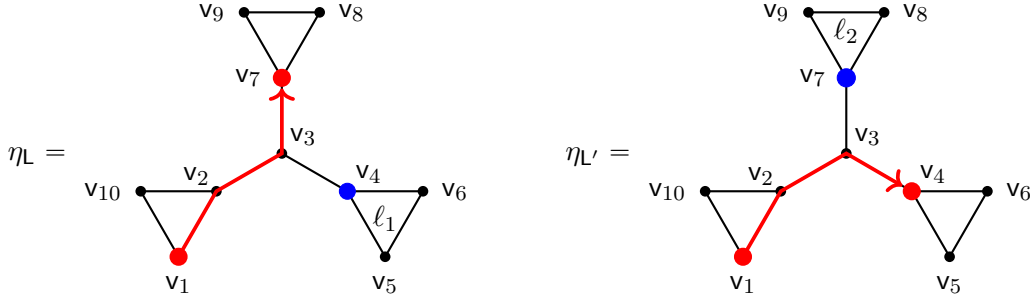

\begin{lemma}[Comparison paths and standard product regions]\label{Lem:ComparisonPathIntersection}
Let \((\sfL,\bfn)\) and \((\sfL',\bfn')\) be standard legal pairs, set \(H=\GBGdisc_{\bfn}(\sfL)\) and \(H'=\GBGdisc_{\bfn'}(\sfL')\), and let \(U,V\in\frSLP\) be represented by \(gH\) and \(g'H'\), respectively. Then \(\mathbf P_U\cap\mathbf P_V\neq\varnothing\) if and only if the two legal pairs are compatible and there exists a comparison path \(\eta\) from \((\sfL,\bfn)\) to \((\sfL',\bfn')\) such that \(g^{-1}g'\in H\delta_\eta H'\).
\end{lemma}
\begin{proof} 
Suppose first that the two legal pairs are compatible, and let \(\eta\) be a comparison path. Lifting \(\eta\) gives
\[\langle[\sfL],[\bfn]\rangle\cap\delta_\eta\cdot\langle[\sfL'],[\bfn']\rangle\neq\varnothing.\]
By \Cref{Lem:LiftedClosureRegion}, \(H\) and \(H'\) preserve \(\langle[\sfL],[\bfn]\rangle\) and \(\langle[\sfL'],[\bfn']\rangle\), respectively. Hence
\[g^{-1}g'\in H\delta_\eta H'\]
implies \(\mathbf P_U\cap\mathbf P_V\neq\varnothing\).

Conversely, if \(\mathbf P_U\cap\mathbf P_V\neq\varnothing\), projecting a vertex of the intersection gives a vertex in
\[\UD_{[\bfn]}([\sfL])\cap\UD_{[\bfn']}([\sfL']).\]
Thus the two legal pairs are compatible. Choosing paths to this vertex inside the two closure subcomplexes gives a comparison path \(\eta\), and the covering-space lifting criterion yields \(g^{-1}g'\in H\delta_\eta H'\).
\end{proof}

\begin{lemma}[Nesting and orthogonality]\label{Lem:NestingOrthogonality}
Let \(U,V\in\frSLP\) be domains represented respectively by \(g\GBGdisc_{\bfn}(\sfL)\) and \(g'\GBGdisc_{\bfn'}(\sfL')\), where \((\sfL,\bfn)\) and \((\sfL',\bfn')\) are standard legal pairs.

Then \(U\sqsubseteq V\) if and only if the following conditions hold:
\begin{enumerate}[label=\textup{(\arabic*)}, ref=\arabic*]
\item\label{Item:Compatible} The standard legal pairs \((\sfL,\bfn)\) and \((\sfL',\bfn')\) are compatible.
\item\label{Item:Relation} \((\sfL^{\con},\bfn^{\con})\preceq(\sfL'^{\con},\bfn'^{\con})\).
\item\label{Item:ExistPath} There exists a comparison path \(\eta\) from \((\sfL,\bfn)\) to \((\sfL',\bfn')\) such that
\[g^{-1}g'\in\GBGdisc_{\bfn}(\sfL)\,\delta_\eta\,\GBGdisc_{\bfn'}(\sfL').\]
\end{enumerate}

Moreover, \(U\bot V\) if and only if conditions~\textup{(1)} and~\textup{(3)} hold, and condition~\textup{(2)} is replaced by: 
\begin{enumerate}[label=\textup{(\arabic*${}'$)}, ref=\arabic*${}'$, start=2]
\item\label{Item:OrthogonalRelation} \((\sfL'^{\con},\bfn'^{\con})\preceq(\sfL^\perp,\bfn^\perp)\) and \((\sfL^{\con},\bfn^{\con})\preceq(\sfL'^\perp,\bfn'^\perp)\).
\end{enumerate}
\end{lemma}
\begin{proof}
By \Cref{Lem:ComparisonPathIntersection}, conditions~\textup{(1)} and~\textup{(3)} together are equivalent to \(\mathbf P_U\cap\mathbf P_V\neq\varnothing\). For \(\widetilde{\mathbf y}\in\mathbf P_U\cap\mathbf P_V\), let \(F_U(\widetilde{\mathbf y})\) and \(F_V(\widetilde{\mathbf y})\) denote the first-factor fibers of \(\mathbf P_U\) and \(\mathbf P_V\), respectively, through \(\widetilde{\mathbf y}\).

Suppose that conditions~\textup{(1)},~\textup{(2)}, and~\textup{(3)} hold, and choose \(\widetilde{\mathbf y}\in\mathbf P_U\cap\mathbf P_V\). The subconfiguration of the projection of \(\widetilde{\mathbf y}\) supported on \(\sfL'^{\con}\setminus\sfL^{\con}\) realizes a witness for \((\sfL^{\con},\bfn^{\con})\preceq (\sfL'^{\con},\bfn'^{\con})\). Consequently, \(F_U(\widetilde{\mathbf y})\subseteq F_V(\widetilde{\mathbf y})\), and hence \(U\sqsubseteq V\).

Conversely, suppose that \(U\sqsubseteq V\). Then there are factors \(F\in U\) and \(F'\in V\) such that \(F\subseteq F'\). In particular, \(F\subseteq\mathbf P_U\cap\mathbf P_V\), so the first paragraph gives conditions~\textup{(1)} and~\textup{(3)}. 
By \Cref{lemma:parallelism}, the projections of \(F\) and \(F'\) have constrained parts \((\sfL^{\con},\bfn^{\con})\) and \((\sfL'^{\con},\bfn'^{\con})\), respectively.
After projecting the inclusion \(F\subseteq F'\), the particles fixed on \(\sfL'^{\con}\setminus\sfL^{\con}\) provide a witness for \((\sfL^{\con},\bfn^{\con})\preceq (\sfL'^{\con},\bfn'^{\con})\).
Thus condition~\textup{(2)} holds, proving the nesting criterion.

Now suppose that conditions~\textup{(1)},~\textup{(2\('\))}, and~\textup{(3)} hold, and choose \(\widetilde{\mathbf y}\in\mathbf P_U\cap\mathbf P_V\).
At the projection of \(\widetilde{\mathbf y}\), the relevant complementary subconfigurations realize witnesses for the two relations in \textup{(2\('\))}.
Consequently, the motions in \(F_U(\widetilde{\mathbf y})\) and \(F_V(\widetilde{\mathbf y})\) are supported on disjoint constrained subgraphs and can be performed independently. They therefore determine a cubical product region in which \(F_U(\widetilde{\mathbf y})\) and \(F_V(\widetilde{\mathbf y})\) occur as distinct factors. Hence \(U\bot V\).

Conversely, suppose that \(U\bot V\). Choose factors \(F\in U\) and \(F'\in V\) which occur as distinct factors of a common cubical product region \(Q\).
Since \(Q\) is swept out by parallel copies of both factors, we have \(Q\subseteq\mathbf P_U\cap\mathbf P_V\).
The first paragraph therefore gives conditions~\textup{(1)} and~\textup{(3)}. After projecting the product decomposition of \(Q\), the particles fixed in the remaining coordinates provide witnesses for \((\sfL'^{\con},\bfn'^{\con})\preceq(\sfL^\perp,\bfn^\perp)\) and \((\sfL^{\con},\bfn^{\con})\preceq(\sfL'^\perp,\bfn'^\perp)\). Thus \textup{(2\('\))} holds.
\end{proof}

\begin{theorem}[Domains and relations in the legal-pair hierarchy]\label{Thm:GBGHHG}
Let \(\Graf\) be a connected finite graph and let \(2\leq n<\#V(\Graf)\). Then the legal-pair hierarchy \(\bigl(\GBGdisc_n(\Graf),\frSLP\bigr)\) has the following explicit description. 
\begin{itemize}
\item Its domains are represented by left cosets \(g\GBGdisc_{\bfn}(\sfL)\), where \(g\in\GBGdisc_n(\Graf)\) and \((\sfL,\bfn)\) ranges over the standard legal pairs of norm \(n\) such that \(\sfL\) has at least one non-singleton component.
\item The coset \(g\GBGdisc_{\bfn}(\sfL)\) represents the parallelism class of \(g\cdot\langle\sfL,\bfn\rangle\), and the associated hyperbolic space is the factored contact graph of any factor in this parallelism class.
\item Nesting and orthogonality are described by \Cref{Lem:NestingOrthogonality}. 
\end{itemize}
\end{theorem}
\begin{proof}
The domain set is \(\frSLP=\frFLP/{\parallel}\), where \(\frFLP\) is the reduced legal-pair factor system consisting of the non-singleton factors.
By \Cref{Cor:FactorSystemElements}, every factor is a deck translate of \(\langle\sfL,\bfn\rangle\) for some legal pair of norm \(n\). By \Cref{Lem:StandardRepresentativeParallelismClass}, every parallelism class has a representative associated to a standard legal pair. The translates of each standard factor are parametrized by the stated left cosets, and each such translate determines its parallelism class. The assertion about the associated hyperbolic spaces is part of the factor-system construction, and the final assertion is \Cref{Lem:NestingOrthogonality}.
\end{proof}

\subsection{Membership in \texorpdfstring{\(\Xi\)}{Xi}}

We now verify that the legal-pair hierarchy \((G,\frSLP)\) belongs to the class \(\Xi\) recalled in \Cref{Def:Xiclass}. We first identify the subgroups \(G_U\) and verify the \(\mathbf F_U\) stabilizers property. We then study common product regions for orthogonal domains. Their stabilizers give the commutative property directly and provide the product decomposition needed for the orthogonal decomposition property.

\begin{lemma}[The subgroups \(G_U\)]\label{Lem:LegalPairGU}
Let \(U\in\frSLP\), and let \(F=g\cdot\langle\sfL,\bfn\rangle\in\frFLP\) represent \(U\). Then
\[G_U=\Stab_G(F)=g\GBGdisc_{\bfn}(\sfL)g^{-1}\cong\GBGdisc_{\bfn^{\con}}(\sfL^{\con})
\cong\prod_{\sfL_i\in\pi_0(\sfL^{\con})}\GBGdisc_{\bfn(\sfL_i)}(\sfL_i).\]
\end{lemma}
\begin{proof}
In the cubical model fixed in \Cref{Subsection:HHG}, the \(\mathbf F_U\)-fibers are precisely the factors parallel to \(F\). By \cite[Lemma~3.6]{OPRAAG}, all these factors have the same stabilizer. Hence an element of \(\Stab_G(U)\) stabilizes every \(\mathbf F_U\)-fiber if and only if it stabilizes \(F\). Therefore \(G_U=\Stab_G(F)\).

By the convention on distinguished lifts and subgroup representatives, together with covering-space theory, \(\Stab_G\bigl(\langle\sfL,\bfn\rangle\bigr)=\GBGdisc_{\bfn}(\sfL)\).
Consequently, \(\Stab_G(F)=g\GBGdisc_{\bfn}(\sfL)g^{-1}\).

Finally, the decomposition in \eqref{Eq:ProductofGBG} gives \(\GBGdisc_{\bfn}(\sfL)\cong \prod_{\sfL_i\in\pi_0(\sfL)}\GBGdisc_{\bfn(\sfL_i)}(\sfL_i)\).
Every free component is a singleton, so its corresponding braid group is trivial. Removing these trivial factors gives
\[\GBGdisc_{\bfn}(\sfL)\cong\GBGdisc_{\bfn^{\con}}(\sfL^{\con})
\cong\prod_{\sfL_i\in\pi_0(\sfL^{\con})}\GBGdisc_{\bfn(\sfL_i)}(\sfL_i).\qedhere\]
\end{proof}

To verify the \(\mathbf F_U\) stabilizers property, it remains to show that \(G_U\) acts cofinitely on the induced nesting index set.

\begin{proposition}\label{Prop:LegalPairFUStabilizers}
The legal-pair hierarchy satisfies the \(\mathbf F_U\) stabilizers property.
\end{proposition}

\begin{proof}
Let \(U=[F]\in\frSLP\). By \Cref{Lem:LegalPairGU}, \(G_U=\Stab_G(F)\). The factor system induced on \(F\) is \(G_U\)-invariant and locally finite. The action \(G_U\curvearrowright F\) is proper, and its quotient is the compact legal-pair subcomplex covered by \(F\). Hence the action is cocompact.

As in the compact-special argument in the proof of \cite[Proposition~3.9]{OPRAAG}, local finiteness of the induced factor system together with cocompactness on \(F\) implies that \(G_U\) acts cofinitely on \(\frS_U\). Therefore the natural restricted action \(G_U\curvearrowright(\mathbf F_U,\frS_U)\) defines an HHG structure, as required.
\end{proof}

We next record the common product region associated to a finite collection of pairwise orthogonal domains. This gives the commutative property immediately and will also be the main input for the orthogonal decomposition property.

\begin{lemma}[Common product regions for orthogonal domains]\label{Lem:OrthogonalLegalPairProductRegions}
Let \(U_1,\dots,U_k\in\frSLP\), where \(k\geq1\), be pairwise orthogonal. Then the intersection of their standard product regions
\[Q:=\bigcap_{i=1}^k\mathbf P_{U_i}\]
is a nonempty legal-pair lift admitting a cubical product decomposition
\[Q\cong F_1\times\cdots\times F_k\times F_0,\]
where \(F_i\) represents \(U_i\) for every \(i\).

Let \(G_0\) be the product of the remaining component braid subgroups corresponding to the residual factor \(F_0\), with \(G_0=\{1\}\) when \(F_0\) is a point. Then multiplication induces an isomorphism
\[G_{U_1}\times\cdots\times G_{U_k}\times G_0\longrightarrow \Stab_G(Q).\]
In particular, the subgroups \(G_{U_1},\dots,G_{U_k},G_0\) commute pairwise. Moreover, if \(h\in G_0\), then \(V\bot U_i\) for every \(V\in\B(h)\) and every \(i\in\{1,\dots,k\}\).
\end{lemma}
\begin{proof}
By \Cref{Lem:NestingOrthogonality,Lem:ComparisonPathIntersection}, pairwise orthogonality implies \(\mathbf P_{U_i}\cap\mathbf P_{U_j}\neq\varnothing\) for \(i\neq j\).
The subcomplexes \(\mathbf P_{U_i}\) are convex. Hence the Helly property for convex subcomplexes of \(\CAT(0)\) cube complexes \cite{Ger98,Rol99} gives \(Q:=\bigcap_{i=1}^k\mathbf P_{U_i}\neq\varnothing\). 
Each \(\mathbf P_{U_i}\) is a legal-pair lift of a closure subcomplex. Repeated applications of \Cref{Lem:Intersection_of_Legal_pairs,Lem:IntersectionOfLifts} therefore show that \(Q\) is a legal-pair lift.

For each \(i\), choose a standard legal pair \((\sfL_i,\bfn_i)\) and a left coset representing \(U_i\), as in \Cref{Thm:GBGHHG}. Choose a vertex \(\widetilde{\mathbf y}\in Q\), and let \(F_i\) be the first-factor fiber of \(\mathbf P_{U_i}\) through \(\widetilde{\mathbf y}\). By \Cref{Lem:LiftedClosureRegion}, the factor \(F_i\) represents \(U_i\).

For \(i\neq j\), \Cref{Lem:NestingOrthogonality}\eqref{Item:OrthogonalRelation}, together with the witness supplied by the projection of \(\widetilde{\mathbf y}\), gives
\[(\sfL_i^{\con},\bfn_i^{\con})\preceq
(\sfL_j^\perp,\bfn_j^\perp).\]
By \Cref{Lem:LocalisometrybetweenLegalPairs}, lifting the corresponding cubical embedding through \(\widetilde{\mathbf y}\) gives \(F_i\subseteq\mathbf P_{U_j}\).
Hence \(F_i\subseteq Q\) for every \(i\).

Let \(A=p(Q)\). By \Cref{Lem:Intersection_of_Legal_pairs,Lem:IntersectionOfLifts}, it is a legal-pair subcomplex; write \(A=\UD_{\bfn'}(\sfL')\).
For each \(i\), the relation above shows that \(\sfL_i^{\con}\) is a union of connected components of \(\sfL'\). Indeed, intersection with the other closure subgraphs leaves every component of \(\sfL_i^{\con}\) unchanged. The particle-counting equalities in \Cref{Lem:Intersection_of_Legal_pairs} moreover give
\[\bfn'\big|_{\pi_0(\sfL_i^{\con})}=\bfn_i^{\con}.\]
Consequently, the componentwise product decomposition of \(A\) has the form
\[A\cong\UD_{\bfn_1^{\con}}(\sfL_1^{\con})\times\cdots\times\UD_{\bfn_k^{\con}}(\sfL_k^{\con})\times A_{\mathrm{res}},\]
where \(A_{\mathrm{res}}\) is the product of the remaining component configuration spaces and may be a point. 
Lifting this decomposition through \(\widetilde{\mathbf y}\) gives
\[Q\cong F_1\times\cdots\times F_k\times F_0,\]
because, for each \(i\), the subcomplex through \(\widetilde{\mathbf y}\) obtained by varying precisely the coordinates supported on \(\sfL_i^{\con}\) is \(F_i\).

Since \(Q\) contains the non-singleton factor \(F_1\), it belongs to \(\frFLP\). By \Cref{Lem:LegalPairGU}, its stabilizer is the product of the component braid subgroups associated to the legal-pair decomposition above. Since \(F_i\) represents \(U_i\), the subgroup acting on the \(F_i\)-coordinate is precisely \(G_{U_i}\). Therefore multiplication induces an isomorphism
\[G_{U_1}\times\cdots\times G_{U_k}\times G_0\longrightarrow\Stab_G(Q).\]
In particular, these subgroups commute pairwise.

It remains to prove the assertion about active domains. Write \(\widetilde{\mathbf y}=(\widetilde{\mathbf y}_1,\dots,\widetilde{\mathbf y}_k,\widetilde{\mathbf y}_0)\) with respect to the displayed product decomposition, and let
\[A_0=\{\widetilde{\mathbf y}_1\}\times\cdots\times\{\widetilde{\mathbf y}_k\}\times F_0.\]
This is again a legal-pair lift. Every \(h\in G_0\) preserves
\(A_0\), and the \(h\)-orbit of \(\widetilde{\mathbf y}\) is
contained in \(A_0\).

Let \(V\in\B(h)\), and choose a factor \(R\in\frFLP\) representing \(V\). Since \(V\) is active for \(h\), the projection of the \(h\)-orbit of \(\widetilde{\mathbf y}\), and hence that of \(A_0\), to \(\Chat R\) has infinite diameter. In particular, \(A_0\) is not a point and belongs to \(\frFLP\). By \cite[Lemma~8.19]{BHS17I}, the factor \(R\) is parallel to a subcomplex \(R_0\subseteq A_0\). Since legal-pair lifts are preserved under parallel transport, \(R_0\in\frFLP\) and \(R_0\) represents \(V\). Fixing all the remaining coordinates of \(Q\), the factors \(F_i\) and \(R_0\) occur as the two factors of a cubical product subcomplex \(F_i\times R_0\subseteq Q\). Hence \(V\bot U_i\) for every \(i\).
\end{proof}

In particular, \Cref{Lem:OrthogonalLegalPairProductRegions} gives \([G_U,G_V]=1\) whenever \(U\bot V\), and hence the legal-pair hierarchy satisfies the commutative property. 

It remains to verify the orthogonal decomposition property.

\begin{proposition}\label{Prop:LegalPairOrthogonalDecomposition}
The legal-pair hierarchy satisfies the orthogonal decomposition property.
\end{proposition}
\begin{proof}
Let \(a\in G\) have infinite order, and write \(\B(a)=\{U_1,\dots,U_k\}\). By \Cref{Prop:AxialEllipticDichotomy}, the domains \(U_1,\dots,U_k\) are pairwise orthogonal, and there exists \(M>0\) such that \(b:=a^M\) fixes every \(U_i\). By the power invariance of bigsets, \(\B(b)=\B(a)=\{U_1,\dots,U_k\}\). 

Since \(b\) fixes every \(U_i\), \Cref{Lem:LiftedClosureRegion} shows that \(b\) preserves each standard product region \(\mathbf P_{U_i}\). Hence \(b\) stabilizes \(Q=\bigcap_{i=1}^k\mathbf P_{U_i}\). By \Cref{Lem:OrthogonalLegalPairProductRegions}, there is a product decomposition \(Q\cong F_1\times\cdots\times F_k\times F_0\) such that
\[\Stab_G(Q)\cong G_{U_1}\times\cdots\times G_{U_k}\times G_0.\]
Thus we may write
\[b=h_1\cdots h_kh_0,\qquad h_i\in G_{U_i},\quad h_0\in G_0,\]
where \(h_1,\dots,h_k,h_0\) commute pairwise.

We first show that \(\B(h_0)=\varnothing\). Suppose otherwise, and let \(V\in\B(h_0)\). By the final assertion of \Cref{Lem:OrthogonalLegalPairProductRegions}, \(V\bot U_i\) for every \(i\).
By \Cref{Prop:AxialEllipticDichotomy}, there exists \(L>0\) such that \(h_0^L\) fixes \(V\) and acts loxodromically on \(\mathcal C V\). On the other hand, for every \(i\), we have \(h_i\in G_{U_i}\) and \(V\bot U_i\). Hence \cite[Lemma~3.3]{OPRAAG} implies that \(h_i\) fixes \(V\) and acts trivially on \(\mathcal C V\). Since the elements commute, we have
\[b^L=h_1^L\cdots h_k^Lh_0^L,\]
and therefore \(b^L\) fixes \(V\) and acts loxodromically on
\(\mathcal C V\). It follows that
\[V\in\B(b^L)=\B(b)=\{U_1,\dots,U_k\}.\]
This is impossible, since \(V\bot U_i\) for every \(i\). Thus \(\B(h_0)=\varnothing\).

By \Cref{Prop:AxialEllipticDichotomy}, \(h_0\) has finite order. On the other hand, \(G\) is torsion-free: every finite-order isometry of the complete \(\CAT(0)\) cube complex \(\widetilde{\UD}_n(\Graf)\) fixes a point \cite[Corollary~II.2.8]{BH}, whereas the deck action of \(G\) is free.
Hence \(h_0=1\), and therefore \(b=h_1\cdots h_k\).

It remains to show that each \(h_i\) is strongly fully supported on \(U_i\). Fix \(i\). For every \(j\neq i\), we have \(U_j\bot U_i\) and \(h_j\in G_{U_j}\), so \cite[Lemma~3.3]{OPRAAG} implies that \(h_j\) fixes \(U_i\) and acts trivially on \(\mathcal C U_i\). Consequently, \(h_i\) and \(b\) induce the same action on \(\mathcal C U_i\). Since \(b=a^M\) acts loxodromically on \(\mathcal C U_i\), the same is true of \(h_i\). Thus \(U_i\in\B(h_i)\).
In particular, \(h_i\) has infinite order and hence is axial by \Cref{Prop:AxialEllipticDichotomy}. Together with \(h_i\in G_{U_i}\), this shows that \(h_i\) is strongly fully supported on \(U_i\).
Therefore \(a^M=h_1\cdots h_k\) is the required orthogonal decomposition.
\end{proof}

We summarize the structural properties of the legal-pair hierarchy.

\begin{theorem}[Structural properties of the legal-pair hierarchy]\label{Thm:GBGLegalPairXi}
Let \(\Graf\) be a connected finite graph and let \(2\leq n<\#V(\Graf)\). Then \((\GBGdisc_n(\Graf),\frSLP)\in\Xi\).
\end{theorem}
\begin{proof}
The \(\mathbf F_U\) stabilizers property follows from \Cref{Prop:LegalPairFUStabilizers}, the commutative property from \Cref{Lem:OrthogonalLegalPairProductRegions}, and the orthogonal decomposition property from \Cref{Prop:LegalPairOrthogonalDecomposition}.
Thus the conclusion follows from \Cref{Def:Xiclass}.
\end{proof}

\section{General applications to graph braid groups}\label{Section:ApplicationtoGBG}
We now apply the legal-pair hierarchy to general graph braid groups. We classify its nesting-minimal unbounded domains, compute its rank, and obtain undistorted RAAG embeddings from generalized halo graphs.

\subsection{Nesting-minimal unbounded domains}
The coset description of domains in \Cref{Thm:GBGHHG} allows us to classify the nesting-minimal unbounded domains explicitly.

\begin{theorem}[Nesting-minimal unbounded domains]\label{Thm:MinimalUnboundedDomainsGBG}
Let \(\Graf\) be a connected finite graph and let \(2\leq n<\#V(\Graf)\). A domain \(U\in\frSLP\) of the legal-pair hierarchy on \(\GBGdisc_n(\Graf)\) is nesting-minimal unbounded if and only if it is represented by a left coset \(g\GBGdisc_{\bfn}(\sfL)\), where \(g\in\GBGdisc_n(\Graf)\) and \((\sfL,\bfn)\) is a standard legal pair of norm \(n\) whose constrained part has a unique non-singleton component \(\sfL_0\), where \(\sfL_0\) is one of the following:
\begin{enumerate}
\item an embedded cycle carrying a positive number of particles;
\item a tripod, that is, a subgraph isomorphic to \(\sfK_{1,3}\), carrying
two particles.
\end{enumerate}
\end{theorem}
\begin{proof}
Set \(G=\GBGdisc_n(\Graf)\). We shall repeatedly use the following observation. Let \(F\in\frFLP\) be a legal-pair lift of \(\UD_{\bfn}(\sfL)\). Replacing some component pairs \((\sfL_i,\bfn(\sfL_i))\) by legal pairs of the same respective norms whose configuration spaces are subcomplexes of \(\UD_{\bfn(\sfL_i)}(\sfL_i)\) produces a legal pair \((\sfL',\bfn')\) of norm \(n\) such that \(\UD_{\bfn'}(\sfL')\subset\UD_{\bfn}(\sfL)\). Lifting this inclusion to \(F\) gives a legal-pair lift \(F'\subset F\).
If \(F'\) is non-singleton, then \(V=[F']\) is a domain satisfying \(V\sqsubseteq[F]\). If the replacement changes the constrained part, then \(F'\not\parallel F\) by \Cref{lemma:parallelism}, and hence \(V\neq[F]\). The pair \((\sfL',\bfn')\) may subsequently be replaced by its standard representative using \Cref{Lem:StandardRepresentativeParallelismClass}.

Suppose, in addition, that \(G_V=\Stab_G(F')\) contains an infinite-order element. By \Cref{Prop:LegalPairFUStabilizers}, \(G_V\) is an HHG with underlying HHS \((\mathbf F_V,\frS_V)\). Applying \Cref{Prop:AxialEllipticDichotomy} to this induced HHG gives an unbounded domain \(W\sqsubseteq V\). Consequently, if \(U=[F]\) is nesting-minimal unbounded, no replacement can both change the constrained part and leave an infinite-order element in \(\Stab_G(F')\), since it would give \(W\sqsubseteq V\sqsubsetneq U\).

\smallskip

\noindent\emph{Forward implication.} Let \(U\) be a nesting-minimal unbounded domain. By \Cref{Thm:GBGHHG}, \(U\) is represented by a coset \(g\GBGdisc_{\bfn}(\sfL)\), where \((\sfL,\bfn)\) is a standard legal pair with \(\|\bfn\|=n\). 
Let \(F=g\cdot\langle\sfL,\bfn\rangle\) be the corresponding legal-pair lift. Since \(U\) is unbounded, so is \(F\). By \Cref{Lem:LegalPairGU},
\[\Stab_G(F)=g\GBGdisc_{\bfn}(\sfL)g^{-1}.\]
By \Cref{Prop:LegalPairFUStabilizers}, this group acts cocompactly on \(F\), and is therefore infinite. Hence some constrained component \(\sfL_i\) supports an infinite graph braid group.

We claim first that every constrained component \(\sfL_i\) supporting an infinite graph braid group is one of the two stated types. Write \(m=\bfn(\sfL_i)\). Suppose first that \(\sfL_i\) contains an embedded cycle \(\ell\). If \(\sfL_i\neq\ell\), let \(D_\ell=\#\bigl(V(\sfL_i)\setminus V(\ell)\bigr)\).
Since \((\sfL,\bfn)\) is legal and \(\sfL_i\) is not a singleton, we have \(m<\#V(\sfL_i)=\#V(\ell)+D_\ell\).
Hence we may choose an integer \(a\) satisfying
\[\max\{1,m-D_\ell\}\leq a\leq \min\{m,\#V(\ell)-1\}.\]
Replace the component \(\sfL_i\) by the embedded cycle \(\ell\) carrying \(a\) particles, together with \(m-a\) singleton components chosen from \(V(\sfL_i)\setminus V(\ell)\). 
Keeping all other components unchanged, the stabilizer of the new lift contains an infinite-order element, since \(\GBGdisc_a(\ell)\cong\mathbb Z\). The observation above thus contradicts the nesting-minimality of \(U\). Therefore, if \(\sfL_i\) contains a cycle, then \(\sfL_i\) itself must be an embedded cycle.

Now suppose that \(\sfL_i\) contains no cycle. Since \(\GBGdisc_m(\sfL_i)\) is infinite, \Cref{Lem:InfiniteGBG} implies that \(\sfL_i\) contains an essential vertex and \(m\leq \#V(\sfL_i)-2\). Choose a tripod \(\tau\subseteq\sfL_i\). If \(\sfL_i\neq\tau\), then
\[m-2\leq \#V(\sfL_i)-4=\#\bigl(V(\sfL_i)\setminus V(\tau)\bigr).\]
Replace \(\sfL_i\) by the tripod \(\tau\) carrying two particles, together with \(m-2\) singleton components chosen from \(V(\sfL_i)\setminus V(\tau)\). Since \(\GBGdisc_2(\tau)\cong\mathbb Z\), the stabilizer of the new lift contains an infinite-order element. The observation above again contradicts the nesting-minimality of \(U\).
Hence \(\sfL_i\) is a tripod. Since \(\GBGdisc_m(\tau)\) is infinite, \Cref{Lem:InfiniteGBG} gives \(2\le m\le \#V(\tau)-2=2\). Hence \(m=2\) and \(\sfL_i\) is a tripod carrying two particles.

Minimality now implies that there is a unique constrained component supporting an infinite graph braid group. Indeed, if there were two, one of them could be replaced by singleton components while retaining the other. Moreover, every other constrained component is a singleton: if one were non-singleton, it could likewise be replaced by singleton components while retaining the unique component supporting an infinite graph braid group. This completes the forward implication.

\smallskip

\noindent\emph{Reverse implication.} Suppose that \((\sfL,\bfn)\) is a standard legal pair satisfying the stated conditions, and let \(U\) be a domain represented by a coset \(g\GBGdisc_{\bfn}(\sfL)\). Let \(\sfL_0\) be the unique non-singleton component of \(\sfL^{\con}\), and set \(m=\bfn(\sfL_0)\).

We first observe that there exists an unbounded domain nested into \(U\). Let
\(F=g\cdot\langle\sfL,\bfn\rangle\).
By \Cref{Lem:LegalPairGU},
\[G_U=\Stab_G(F)\cong\GBGdisc_m(\sfL_0)\cong\mathbb Z.\]
Moreover, by \Cref{Prop:LegalPairFUStabilizers}, \(G_U\) is an HHG with underlying HHS \((\mathbf F_U,\frS_U)\). Choose a nontrivial element \(h\in G_U\). Since \(G_U\cong\mathbb Z\), the element \(h\) has infinite order. Applying \Cref{Prop:AxialEllipticDichotomy} to this HHG gives an unbounded domain \(V\in\frS_U\). In particular, \(V\sqsubseteq U\).

We now show that every unbounded domain nested into \(U\) is equal to \(U\). Let \(V\sqsubseteq U\) be unbounded, and choose legal-pair lifts \(F_V\in V\) and \(F_U\in U\) with \(F_V\subset F_U\). By \Cref{Thm:GBGHHG}, choose a coset \(g'\GBGdisc_{\bfn'}(\sfL')\) representing \(V\), where \((\sfL',\bfn')\) is a standard legal pair of norm \(n\). Since \(V\) is unbounded, some constrained component \(\sfL'_0\) supports an infinite graph braid group. Then \Cref{Lem:NestingOrthogonality} gives \(((\sfL')^{\con},(\bfn')^{\con})\preceq(\sfL^{\con},\bfn^{\con})\).
No singleton component of \(\sfL^{\con}\) can contain the non-singleton component \(\sfL'_0\). Therefore \((\sfL'_0,\bfn'(\sfL'_0))\preceq
(\sfL_0,\bfn(\sfL_0))\).

If \(\sfL_0\) is an embedded cycle, no proper subgraph of \(\sfL_0\) supports an infinite graph braid group, and the relation \(\preceq\) preserves the particle count when the entire cycle is retained. If \(\sfL_0\) is a tripod carrying two particles, no proper constrained subpair supports an infinite graph braid group by \Cref{Lem:InfiniteGBG}. Thus, in either case, \((\sfL'_0,\bfn'(\sfL'_0))=(\sfL_0,\bfn(\sfL_0))\).
Since \(\sfL'_0=\sfL_0\), every other component of \((\sfL')^{\con}\) is contained in a singleton component of \(\sfL^{\con}\), and is therefore a singleton.

Let \(A_V=p(F_V)\) and \(A_U=p(F_U)\). By \Cref{lemma:parallelism}, these legal-pair subcomplexes have the same constrained parts as \((\sfL',\bfn')\) and \((\sfL,\bfn)\), respectively. Hence there are vertex configurations \(\mathbf y_V\) and \(\mathbf y_U\) such that 
\[A_V=\UD_m(\sfL_0)\times\{\mathbf y_V\},\qquad A_U=\UD_m(\sfL_0)\times\{\mathbf y_U\}.\]
Since \(F_V\subset F_U\), we have \(A_V\subset A_U\), and therefore \(\mathbf y_V=\mathbf y_U\). Thus \(A_V=A_U\).
Now \(F_V\) and \(F_U\) are connected components of the full preimage of the same connected subcomplex, and they intersect. Hence \(F_V=F_U\), so \(V=U\).

Thus every unbounded domain nested into \(U\) is equal to \(U\). Since such an unbounded domain exists by the preceding paragraph, \(U\) itself is unbounded. Therefore \(U\) is nesting-minimal unbounded.
\end{proof}

The two types appearing in \Cref{Thm:MinimalUnboundedDomainsGBG} are illustrated in \Cref{Fig:MinimalUnboundedTypes}.

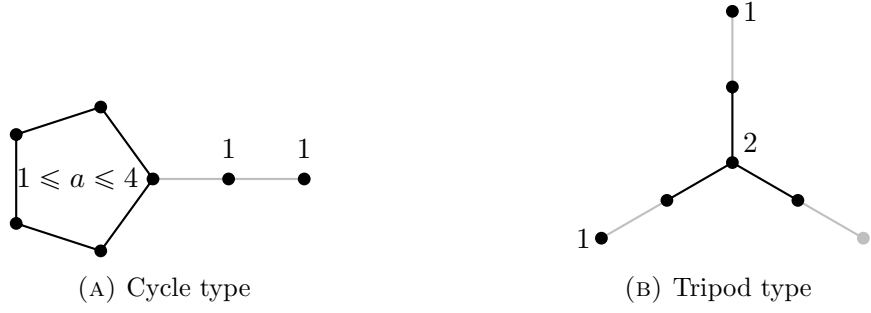
\begin{figure}[ht]
\subcaptionbox{Cycle type\label{Fig:MinimalCycleType}}[0.45\textwidth]{
\(
\begin{tikzpicture}[baseline=-.5ex]
\draw[thick,lightgray] (1,0) -- (2,0) -- (3,0);
\draw[thick,fill] (0,0) node {\(1\le a\le 4\)} (0:1) circle (2pt) -- (72:1) circle (2pt) -- (144:1) circle (2pt) -- (216:1) circle (2pt) -- (288:1) circle (2pt) -- (0:1);
\draw[thick,fill] (2,0) circle (2pt) node[above=0.8ex] {\(1\)};
\draw[thick,fill] (3,0) circle (2pt) node[above=0.8ex] {\(1\)};
\end{tikzpicture}
\)
}
\subcaptionbox{Tripod type\label{Fig:MinimalTripodType}}[0.45\textwidth]{
\(
\begin{tikzpicture}[baseline=-.5ex]
\draw[thick,lightgray] (0,1) -- (0,2);
\draw[thick,lightgray] (210:1) -- (210:2);
\draw[thick,lightgray] (-30:1) -- (-30:2);
\draw[thick,fill,lightgray] (-30:2) circle (2pt);
\draw[thick,fill] (0,0) circle (2pt) node[above right] {\(2\)};
\draw[thick,fill] (0,0) -- (0,1) circle (2pt);
\draw[thick,fill] (0,0) -- (210:1) circle (2pt);
\draw[thick,fill] (0,0) -- (-30:1) circle (2pt);
\draw[thick,fill] (0,2) circle (2pt) node[right] {\(1\)};
\draw[thick,fill] (210:2) circle (2pt) node[left] {\(1\)};
\end{tikzpicture}
\)
}
\caption{The two types of \Cref{Thm:MinimalUnboundedDomainsGBG}.}
\label{Fig:MinimalUnboundedTypes}
\end{figure}

\begin{corollary}\label{Cor:CoreEqualsExpandedCore}
Let \(U\) be a nesting-minimal unbounded domain of the legal-pair hierarchy. If \(F\in\frFLP\) satisfies \([F]=U\), then the projection \(F\rightarrow\mathcal C U\) is a quasi-isometry. In particular, \(\mathcal C U\) is a quasi-line.

Consequently, the expanded core graph is canonically isomorphic to the core graph: 
\[\EOLP\cong\mathcal G^{\frSLP}.\]
\end{corollary}
\begin{proof}
Let \(U\) be a nesting-minimal unbounded domain and let \(F\in\frFLP\) satisfy \([F]=U\). By \Cref{lemma:parallelism,Thm:MinimalUnboundedDomainsGBG}, \(F\) is quasi-isometric to a line. Since \(U\) is nesting-minimal unbounded and \(G_U\) acts cofinitely on the domains nested into \(U\) by \Cref{Prop:LegalPairFUStabilizers}, the spaces \(\mathcal C V\), for \(V\sqsubsetneq U\), have uniformly bounded diameter. The distance formula for the induced HHS structure on \(F\) therefore shows that \(F\to\mathcal C U\) is a quasi-isometry. Therefore, \(\mathcal C U\) is a quasi-line.

Consequently, every set \(V_U\) in \Cref{Def:ExpandedCoreGraph} consists of a single vertex.
\end{proof}

We record two consequences of the classification. The first detects non-isolated vertices from the orthogonal pair of a standard legal pair. The second shows that nesting-minimal unbounded domains and their adjacency are detected by their cyclic subgroups \(G_U\). Together, these reduce subdivision stability to transporting the cyclic subgroup supported on the unique cycle or tripod component.

\begin{corollary}[Non-isolated vertices]\label{Cor:NonisolatedVerticesGBG}
Let \(\Graf\) be a connected finite graph, let \(2\leq n<\#V(\Graf)\), and let \(U\) be a nesting-minimal unbounded domain of the legal-pair hierarchy on \(\GBGdisc_n(\Graf)\), represented by \(g\GBGdisc_{\bfn}(\sfL)\), where \((\sfL,\bfn)\) is standard.
Then the vertex corresponding to \(U\) is non-isolated if and only if \(\GBGdisc_{\bfn^\perp}(\sfL^\perp)\) is infinite, or equivalently, if \(\GBGdisc_{k_{\sfD}}(\sfD)\) is infinite for some \(\sfD\in\pi_0(\sfL^\perp)\), where \(k_{\sfD}=\bfn^\perp(\sfD)\).

If \(\Graf\) is sufficiently subdivided for \(n\), this is further equivalent to the condition that some \(\sfD\in\pi_0(\sfL^\perp)\) contains either
\begin{enumerate}
\item an embedded cycle; or
\item a vertex essential in \(\sfD\), with \(k_{\sfD}\geq2\).
\end{enumerate}
\end{corollary}
\begin{proof}
Suppose first that \(U\) is non-isolated, and choose a nesting-minimal unbounded domain \(W\bot U\). 
Represent \(W\) by a standard legal pair \((\sfL',\bfn')\). By \Cref{Lem:NestingOrthogonality}\eqref{Item:OrthogonalRelation},
\[((\sfL')^{\con},(\bfn')^{\con})\preceq(\sfL^\perp,\bfn^\perp).\]
The corresponding local isometry from \Cref{Lem:LocalisometrybetweenLegalPairs} induces an embedding
\[\GBGdisc_{(\bfn')^{\con}}((\sfL')^{\con})\longrightarrow\GBGdisc_{\bfn^\perp}(\sfL^\perp).\]
The group on the left is infinite cyclic by \Cref{Thm:MinimalUnboundedDomainsGBG}.

Conversely, suppose that \(\GBGdisc_{\bfn^\perp}(\sfL^\perp)\) is infinite.
Then the product decomposition in \eqref{Eq:ProductofGBG} and \Cref{Lem:InfiniteGBG} give, in some component of \(\sfL^\perp\), a legal subpair whose unique non-singleton component is either an embedded cycle carrying a positive number of particles or a tripod carrying two particles.
Completing this support inside the closure, fixing the remaining particles, and passing to the standard representative gives, by \Cref{Thm:MinimalUnboundedDomainsGBG,Lem:NestingOrthogonality}, a nesting-minimal unbounded domain orthogonal to \(U\). The componentwise formulation follows from \eqref{Eq:ProductofGBG}. 

Now suppose that \(\Graf\) is sufficiently subdivided. Let \(\sfL_0\) be the unique non-singleton constrained component, carrying \(m\) particles, and fix \(\sfD\in\pi_0(\sfL^\perp)\). Then \(1\leq k_{\sfD}\leq n-m\). Moreover,
\begin{itemize}
\item 
if \(\sfD\) contains an embedded cycle, then \(\#V(\sfD)\geq n+1>k_{\sfD}\), and
\item 
if it contains a vertex essential in \(\sfD\), the bivalent-path count from \Cref{Def:SufficientlySubdivided} gives
\[\#V(\sfD)\geq n-m+2\geq k_{\sfD}+2.\]
\end{itemize}
Thus the vertex-count conditions in \Cref{Lem:InfiniteGBG} are automatic, giving the stated criterion.
\end{proof}

\begin{corollary}[Stabilizers of nesting-minimal unbounded domains]\label{Cor:MinimalDomainStabilizers}
Let \(\Graf\) be a connected finite graph, let \(2\leq n<\#V(\Graf)\), and let \(U\) be a nesting-minimal unbounded domain of the legal-pair hierarchy on \(\GBGdisc_n(\Graf)\). Then \(G_U\cong\mathbb Z\), and every nontrivial element of \(G_U\) is strongly fully supported on \(U\). Consequently:
\begin{enumerate}
\item 
if \(U\) and \(V\) are nesting-minimal unbounded domains and \(G_U=G_V\), then \(U=V\);
\item 
if \(U\) and \(V\) are distinct nesting-minimal unbounded domains, then
\[U\bot V \qquad\Longleftrightarrow\qquad [G_U,G_V]=1.\]
\end{enumerate}
\end{corollary}
\begin{proof}
Set \(G=\GBGdisc_n(\Graf)\), and let \(F\) be a legal-pair lift representing \(U\). By \Cref{Lem:LegalPairGU,Prop:LegalPairFUStabilizers,Thm:MinimalUnboundedDomainsGBG}, \(G_U=\Stab_G(F)\cong\mathbb Z\), and \(G_U\) acts cocompactly on the quasi-line \(F\). By \Cref{Cor:CoreEqualsExpandedCore}, the projection \(F\rightarrow\mathcal C U\) is a quasi-isometry.
Thus every nontrivial \(h\in G_U\) has an unbounded orbit in \(F\). Since \(F\rightarrow\mathcal C U\) is a quasi-isometry, it follows that \(U\in\B(h)\).
Moreover, \(h\) is axial by \Cref{Prop:AxialEllipticDichotomy}. Hence \(h\) is strongly fully supported on \(U\), and the singleton bigset property following \Cref{Def:OurGUandFU} gives \(\B(h)=\{U\}\).

Suppose that \(G_U=G_V\), and choose a nontrivial element \(h\) of this common infinite cyclic subgroup. Applying the preceding conclusion to both \(U\) and \(V\) gives \(\{U\}=\B(h)=\{V\}\), and therefore \(U=V\). This proves \textup{(1)}.

Now let \(U\neq V\). If \(U\bot V\), then \Cref{Lem:OrthogonalLegalPairProductRegions} gives \([G_U,G_V]=1\).
Conversely, suppose that \([G_U,G_V]=1\), and choose nontrivial elements \(h\in G_U\) and \(k\in G_V\). Since \(\B(h)=\{U\}\) and \(\B(k)=\{V\}\), the pair \(\{h,k\}\) is geometrically irredundant. If \(U\not\bot V\), then \Cref{Thm:GBGLegalPairXi,Prop:RAAGEmbeddingTheorem} implies that sufficiently large powers of \(h\) and \(k\) generate a nonabelian free group. This contradicts \([G_U,G_V]=1\). Hence \(U\bot V\), proving \textup{(2)}.
\end{proof}

\begin{proposition}[Subdivision stability]\label{Prop:SubdivisionStabilityExpandedCore}
Fix \(n\geq2\), and let \(\Graf\) be a connected finite graph that is sufficiently subdivided for \(n\), and let \(\Graf'\) be a further subdivision of \(\Graf\). Let \(\EO\) and \(\EO'\) be the expanded core graphs of the legal-pair hierarchies on \(\GBGdisc_n(\Graf)\) and \(\GBGdisc_n(\Graf')\), respectively.

The full subgraphs of \(\EO\) and \(\EO'\) spanned by their non-isolated vertices are isomorphic. Consequently, the nontrivial connected components are independent, up to graph isomorphism, of the chosen sufficiently subdivided model.
\end{proposition}
\begin{proof}
Set \(G=\GBGdisc_n(\Graf)\) and \(G'=\GBGdisc_n(\Graf')\). By \Cref{Cor:CoreEqualsExpandedCore}, the vertices of \(\EO\) and \(\EO'\) are identified with the nesting-minimal unbounded domains of the respective legal-pair hierarchies.

The natural homeomorphism between the underlying topological graphs identifies their continuous configuration spaces. Hence \Cref{Thm:DiscreteConfigurationRetraction}, together with choices of change-of-basepoint paths, gives an isomorphism \(\phi\colon G\rightarrow G'\).

Let \(U\) be a non-isolated vertex of \(\EO\), represented by a coset \(gH\), where \(H=\GBGdisc_{\bfn}(\sfL)\) and \((\sfL,\bfn)\) is standard. By \Cref{Thm:MinimalUnboundedDomainsGBG}, the constrained part of this legal pair has a unique non-singleton component \(\sfL_0\), which is either an embedded cycle carrying a positive number of particles or a tripod carrying two particles.

Transfer this support to \(\Graf'\). In the cycle case, take its subdivided image. In the tripod case, take the tripod consisting of the same center and the first edge in each of the same three incident directions. Retain the particle number on this support and the particle numbers in the corresponding components of its complement, and then pass to the standard representative. Denote the resulting standard legal pair by \((\sfL',\bfn')\), and set \(H'=\GBGdisc_{\bfn'}(\sfL')\).

By \Cref{Thm:MinimalUnboundedDomainsGBG}, this legal pair represents a nesting-minimal unbounded domain for \(\Graf'\). Subdivision preserves embedded cycles, essential vertices in the corresponding complementary components, and the relevant particle numbers. Therefore \Cref{Cor:NonisolatedVerticesGBG} shows that this domain is non-isolated.

The same construction is reversible for non-isolated domains. Indeed, let \((\sfL'',\bfn'')\) be a standard legal pair representing a non-isolated domain for \(\Graf'\), let \(\sfL''_0\) be its unique non-singleton constrained component, and set \(m=\bfn''(\sfL''_0)\). Choose a component \(\sfD'\in\pi_0((\sfL'')^\perp)\) witnessing non-isolation as in \Cref{Cor:NonisolatedVerticesGBG}, and set \(k=(\bfn'')^\perp(\sfD')\).

The estimates in the proof of that corollary show that the corresponding component \(\sfD\) in \(\Graf\) can carry \(k\) particles. Every other occupied complementary component carries at most
\(n-m-k\leq n-2\) particles. Following a bivalent path from the transferred support to the next non-bivalent vertex, or back to the support, shows that each corresponding nonempty component in \(\Graf\) contains at least \(n-2\) vertices. Thus the entire particle distribution can be realized in \(\Graf\).

It remains to incorporate the coset data. The singleton constrained components contribute only trivial factors to \(H\) and \(H'\). If \(\sfL_0\) is a cycle, then \(H\) and \(H'\) are represented by the same cyclic particle motion in the common continuous configuration space. 
If \(\sfL_0\) is a tripod, then their generators are the elementary exchanges about the same essential vertex through the same three incident directions, and these exchange loops are homotopic in the common continuous configuration space. Hence, after the chosen changes of basepoint, there exists \(a\in G'\) such that \(\phi(H)=aH'a^{-1}\).

The coset \(\phi(g)aH'\) therefore represents a non-isolated nesting-minimal unbounded domain \(U'\) satisfying \(G_{U'}=\phi(G_U)\). By \Cref{Cor:MinimalDomainStabilizers}, \(U'\) is uniquely determined by this subgroup. We may consequently define \(G_{\Psi(U)}=\phi(G_U)\). The reverse construction, applied to \(\phi^{-1}\), shows that \(\Psi\) is a bijection between the non-isolated vertices of \(\EO\) and those of \(\EO'\).

For distinct non-isolated vertices \(U,V\), we have
\[\begin{aligned}
U\bot V &\Longleftrightarrow [G_U,G_V]=1\Longleftrightarrow[\phi(G_U),\phi(G_V)]=1\Longleftrightarrow[G_{\Psi(U)},G_{\Psi(V)}]=1\\
&\Longleftrightarrow\Psi(U)\bot\Psi(V),
\end{aligned}\]
where the first and last equivalences follow from \Cref{Cor:MinimalDomainStabilizers}. Thus \(\Psi\) is an isomorphism between the full subgraphs spanned by the non-isolated vertices.

Finally, any two sufficiently subdivided models of the same topological graph admit a common further subdivision. Applying the preceding result to the two maps into such a common subdivision proves the final assertion.
\end{proof}

\begin{example}[An isolated domain created by further subdivision]
\label{Ex:SubdivisionCreatesIsolatedDomain}
Let \(n=3\), and let \(\Graf=\sfK_{2,3}\) have bipartition \(\{\sfu,\sfv\}\sqcup\{\sfw_1,\sfw_2,\sfw_3\}\). Let \(\ell\) be the \(4\)-cycle through \(\sfu,\sfw_1,\sfv,\sfw_2\). The graph \(\Graf\) is sufficiently subdivided for \(3\), but \(\Graf\setminus\ell=\{\sfw_3\}\). Hence there is no legal pair of norm \(3\) whose unique non-singleton constrained component is \(\ell\) carrying one particle.

Subdivide the edge \([\sfu,\sfw_3]\) once, with new vertex \(\sfz\), and let \(\Graf'\) be the resulting graph. Then 
\[\sfL'=\ell\sqcup\sfz\sqcup\sfw_3,\qquad\bfn'(\ell)=\bfn'(\sfz)=\bfn'(\sfw_3)=1\]
is a standard legal pair. Indeed, the singleton components \(\sfz,\sfw_3\) exhaust the corresponding component of \(\Graf'\setminus\ell\), and are therefore constrained. Thus \((\sfL')^\perp=\varnothing\).

By \Cref{Thm:MinimalUnboundedDomainsGBG}, this pair represents a nesting-minimal unbounded domain, and \Cref{Cor:NonisolatedVerticesGBG} shows that the corresponding vertex is isolated. Denote this isolated domain by \(U'\). The natural identification of the continuous configuration spaces, together with \Cref{Thm:DiscreteConfigurationRetraction} and choices of change-of-basepoint paths, gives an isomorphism
\[\phi\colon\GBGdisc_3(\Graf)\longrightarrow\GBGdisc_3(\Graf').\]
The subgroup \(G_{U'}\) consists entirely of pure braids: its generator moves one particle around \(\ell\) while fixing the other two.

Every embedded cycle of \(\Graf=\sfK_{2,3}\) has four vertices and only one vertex outside it. Thus \Cref{Thm:MinimalUnboundedDomainsGBG} shows that every nesting-minimal unbounded domain \(U\) for \(\Graf\) has a cycle support carrying two or three particles, or a tripod support carrying two particles. In each case, \(G_U\) contains a braid with nontrivial particle permutation. Since \(\phi\) preserves the pure braid subgroups, \(G_{U'}\neq\phi(G_U)\) for every such \(U\). Hence the correspondence on non-isolated vertices need not extend to a bijection \(\Psi\) of the full vertex sets satisfying \(G_{\Psi(U)}=\phi(G_U)\).
\end{example}

\subsection{Rank}
Throughout this subsection, when working with \(\GBG_n(\Graf)\), we choose a subdivision \(\Graf^{\mathrm{sd}}\) that is sufficiently subdivided for \(n\), and suppress the superscript when no confusion can arise. By \Cref{Thm:DiscreteConfigurationRetraction} and the natural homeomorphism between a graph and its subdivisions,
\[\GBGdisc_n(\Graf^{\mathrm{sd}})\cong\GBG_n(\Graf^{\mathrm{sd}})\cong\GBG_n(\Graf).\]
We write \(\EOLP\) for the expanded core graph of the chosen legal-pair hierarchy. By \Cref{Prop:SubdivisionStabilityExpandedCore}, its nontrivial connected components are independent of this choice.

\begin{theorem}[Rank formula]\label{Thm:RankGBG}
Let \(\Graf\) be a connected finite graph, and let \(n\geq2\). Let \(r\) be the maximum of \(p+q\) such that \(p+2q\leq n\) and \(p\) embedded cycles and \(q\) essential vertices in \(\Graf\) can be chosen pairwise disjoint.

Let \(\Graf^{\mathrm{sd}}\) be any subdivision of \(\Graf\) that is sufficiently subdivided for \(n\). Then the legal-pair hierarchy on \(\GBGdisc_n(\Graf^{\mathrm{sd}})\) has rank
\[\operatorname{rank}\bigl(\GBGdisc_n(\Graf^{\mathrm{sd}}),\frSLP\bigr)=r.\]
Moreover, \(r\) is the maximal rank of a free abelian subgroup of \(\GBG_n(\Graf)\).
\end{theorem}
\begin{proof}
Fix a subdivision \(\Graf^{\mathrm{sd}}\) as in the statement and denote it again by \(\Graf\). Subdivision preserves embedded cycles, essential vertices, and their disjointness relations, so the number \(r\) is unchanged. If \(\Graf\) is a path, then \(r=0\) and \(\GBG_n(\Graf)\) is trivial. We may therefore assume that \(r\geq1\).

Let \(U_1,\dots,U_k\) be pairwise orthogonal unbounded domains. By finite complexity and inheritance of orthogonality under nesting, we may replace each \(U_i\) by a nesting-minimal unbounded domain nested into it. For each \(i\), choose a coset \(g_i\GBGdisc_{\bfn_i}(\sfL_i)\) representing \(U_i\), where \((\sfL_i,\bfn_i)\) is standard. By \Cref{Thm:MinimalUnboundedDomainsGBG}, its constrained part has a unique non-singleton component \(\sfL_{i,0}\), which is either an embedded cycle carrying at least one particle or a tripod carrying two particles.

Suppose that these components consist of \(p\) cycles and \(q\) tripods, so that \(k=p+q\). By \Cref{Lem:NestingOrthogonality}\eqref{Item:OrthogonalRelation}, these supports are pairwise disjoint. Let \(a_1,\dots,a_p\) be the particle numbers on the cycle components. Their common product region from \Cref{Lem:OrthogonalLegalPairProductRegions} gives
\[a_1+\cdots+a_p+2q\leq n.\]
In particular, \(p+2q\leq n\). Taking the centers of the tripod supports gives an admissible collection of \(p\) embedded cycles and \(q\) essential vertices. Therefore \(k=p+q\leq r\).

Conversely, choose pairwise disjoint embedded cycles \(\ell_1,\dots,\ell_p\) and essential vertices \(\sfv_1,\dots,\sfv_q\) realizing \(r=p+q\). Let \((\sfL_{\mathrm{ab}},\bfn_{\mathrm{ab}})\) be the legal pair constructed in \Cref{Ex:LegalPairProductSubgroups}\eqref{Item:FreeAbelianLegalPair}, and let \(F\) be a legal-pair lift of the legal-pair subcomplex associated to \((\sfL_{\mathrm{ab}},\bfn_{\mathrm{ab}})\). The product decomposition in \Cref{Ex:LegalPairProductSubgroups}\eqref{Item:FreeAbelianLegalPair} lifts to a product decomposition of \(F\).

For each of its \(r\) cycle or tripod factors, let \(F_i\subseteq F\) be the fiber obtained by varying that factor and fixing a vertex in every other factor. Each \(F_i\) is a legal-pair lift. By \Cref{Lem:StandardRepresentativeParallelismClass,Thm:MinimalUnboundedDomainsGBG}, the domain \(V_i=[F_i]\) is nesting-minimal unbounded. Since these fibers arise from distinct factors of the same lifted product, \Cref{Lem:NestingOrthogonality} gives \(V_i\bot V_j\) whenever \(i\neq j\). Thus the hierarchy has rank at least \(r\), and hence
\[\operatorname{rank}\bigl(\GBGdisc_n(\Graf),\frSLP\bigr)=r.\]

By \Cref{Ex:LegalPairProductSubgroups}\eqref{Item:FreeAbelianLegalPair}, the same legal-pair subcomplex also gives an undistorted subgroup \(\mathbb Z^r<\GBG_n(\Graf)\). By the general rank bound for free abelian subgroups of HHGs \cite[Proposition~2.17]{HRSS25}, every free abelian subgroup of \(\GBG_n(\Graf)\) has rank at most the hierarchical rank. Since this rank is \(r\), no free abelian subgroup has rank greater than \(r\). Therefore \(r\) is the maximal rank of a free abelian subgroup of \(\GBG_n(\Graf)\).
\end{proof}

The rank formula also gives a complete rank-two classification.

\begin{theorem}[Rank-two classification for graph braid groups]\label{Thm:RankTwoGBG}
Let \(\Graf\) be a connected finite graph, and let \(n\geq 2\). Then the legal-pair hierarchy associated to any sufficiently subdivided model of \(\Graf\) has rank at most \(2\) if and only if one of the following holds:
\begin{enumerate}[label=\textup{(\arabic*)}, ref=\arabic*]
\item \(n=2\);
\item \(n=3\), and \(\Graf\) does not contain three pairwise disjoint embedded cycles;
\item \(n=4\), and whenever \(\ell_1,\ell_2\subseteq\Graf\) are disjoint embedded cycles, the union \(\ell_1\cup\ell_2\) contains every essential vertex of \(\Graf\);
\item \(n=5\), and every embedded cycle in \(\Graf\), if one exists, contains all but at most one essential vertex of \(\Graf\);
\item \(n\geq 6\), and \(\Graf\) contains at most two essential vertices.
\end{enumerate}
Moreover, in each of these cases, for every finite graph \(\Lambda\), we have
\[\bbA(\Lambda)<\GBG_n(\Graf)\qquad\Longleftrightarrow\qquad\Lambda<\EOLP,\]
where \(\EOLP\) is the expanded core graph of the legal-pair hierarchy on any sufficiently subdivided model of \(\Graf\).
\end{theorem}
\begin{proof}
By \Cref{Thm:RankGBG}, the rank is the maximum of \(p+q\), where \(p\) embedded cycles and \(q\) essential vertices in \(\Graf\) can be chosen pairwise disjoint, subject to \(p+2q\leq n\).
Thus the rank is at most \(2\) if and only if there is no such collection with
\[p+q\geq 3\qquad\text{and}\qquad p+2q\leq n.\]
It is enough to consider the case \(p+q=3\), since any larger collection contains a subcollection with exactly three members, and the inequality \(p+2q\leq n\) remains true after passing to a subcollection.

If \(n=2\), then \(p+2q\leq 2\) forces \(p+q\leq 2\). Hence the rank is always at most \(2\).

If \(n=3\), the only possibility with \(p+q=3\) and \(p+2q\leq 3\) is \((p,q)=(3,0)\). Thus the rank is at most \(2\) precisely when \(\Graf\) does not contain three pairwise disjoint embedded cycles.

If \(n=4\), the possible pairs with \(p+q=3\) and \(p+2q\leq4\) are
\[(p,q)=(3,0)\qquad\text{and}\qquad(p,q)=(2,1).\]
The second case occurs precisely when two disjoint embedded cycles have an essential vertex outside their union. This condition also detects the first case: if three pairwise disjoint embedded cycles exist, then the third cycle contains an essential vertex outside the other two. Hence the rank is at most \(2\) precisely under the condition in \textup{(3)}.

If \(n=5\), the possible pairs are
\[(p,q)=(3,0),\qquad(p,q)=(2,1),\qquad(p,q)=(1,2).\]
These possibilities occur precisely when some embedded cycle misses at least two essential vertices. Indeed, a collection of type \((1,2)\) is exactly such a cycle together with two essential vertices outside it. Collections of types \((2,1)\) and \((3,0)\) also contain a cycle missing at least two essential vertices, obtained from the remaining cycle components and the selected essential vertex. Hence the rank is at most \(2\) precisely under the condition in \textup{(4)}.

Finally, suppose that \(n\geq6\). Three distinct essential vertices give an admissible collection of type \((0,3)\). Conversely, any admissible collection of three embedded cycles and essential vertices forces \(\Graf\) to have at least three essential vertices: every cycle in such a collection is proper and hence contains an essential vertex, and pairwise disjointness makes the resulting essential vertices distinct. This proves \textup{(5)}.

In each of these cases, the chosen legal-pair hierarchy has rank at most \(2\). Therefore \Cref{Thm:RankTwoCriterionHHG} gives
\[\bbA(\Lambda)<\GBG_n(\Graf)\qquad\Longleftrightarrow\qquad\Lambda<\EOLP.\]
Since the left-hand condition does not depend on the chosen subdivision, this equivalence holds for every sufficiently subdivided model.
\end{proof}

\subsection{Generalized halo graphs and undistorted RAAG embeddings}\label{Subsection:GeneralizedHaloGraphs}

We introduce the following weakening of Sabalka's halo graphs, adapted to the legal-pair hierarchy.

\begin{definition}[Generalized halo graph]\label{Def:GeneralizedHaloGraph}
Let \(n\geq1\), and let \(\Lambda\) be a finite simplicial graph equipped with a proper coloring \(\chi\colon V(\Lambda)\rightarrow\{1,\dots,n\}\).
A \emph{generalized \(\Lambda\)-halo graph} with respect to \(\chi\) is a connected finite graph \(\Graf\), together with distinct vertices \(\sfw_1,\dots,\sfw_n\in V(\Graf)\) and pairwise distinct embedded cycles \(\ell_v\subseteq\Graf\) for \(v\in V(\Lambda)\), satisfying the following conditions:
\begin{enumerate}
\item\label{Item:GeneralizedHaloColorVertex} For every \(v\in V(\Lambda)\), the cycle \(\ell_v\) contains \(\sfw_{\chi(v)}\) and contains no color vertex \(\sfw_j\) with \(j\neq\chi(v)\).
\item If \(v\) and \(w\) are adjacent in \(\Lambda\), then \(\ell_v\cap\ell_w=\varnothing\).
\item\label{Item:GeneralizedHaloNonedges} If \(v\) and \(w\) are distinct and non-adjacent in \(\Lambda\), then \(\ell_v\cap\ell_w\neq\varnothing\).
\end{enumerate}
The configuration
\(\mathbf x^0_{\mathrm{Artin}}=\{\sfw_1,\dots,\sfw_n\}\) is called the \emph{Artin base configuration}.
\end{definition}

We do not require \(\chi\) to be surjective. By condition~\textup{(1)}, every unused color vertex lies on none of the supporting cycles, while, for each \(v\in V(\Lambda)\), the particle at \(\sfw_{\chi(v)}\) may travel once around \(\ell_v\) with all the other particles fixed at the remaining color vertices. After choosing an orientation of \(\ell_v\), let \(\gamma_v\in\GBG_n(\Graf)\) denote the resulting \emph{Artin loop braid}.

\begin{example}[A generalized halo graph for \(C_4\)]
\label{Ex:GeneralizedHaloC4}
Let \(\Lambda=C_4\) have vertices \(a,b,c,d\) in cyclic order, with \(\chi(a)=\chi(c)=1\) and \(\chi(b)=\chi(d)=2\). Let \(\Graf\), its color vertices, and its four supporting cycles be as in \Cref{Fig:GeneralizedHaloGraph}.

Each of \(\ell_a,\ell_c\) is disjoint from each of \(\ell_b,\ell_d\), giving the four edges of \(C_4\). On the other hand, \(\ell_a\cap\ell_c\) is an edge and \(\ell_b\cap\ell_d=\{\sfw_2\}\). Thus the non-adjacent pairs have nonempty intersections, and the color conditions are immediate from \Cref{Fig:GeneralizedHaloGraph}. Hence \(\Graf\) is a generalized \(C_4\)-halo graph.
\end{example}

\begin{figure}[ht]
\begin{align*}
\Lambda&=
\vcenter{\hbox{
\begin{tikzpicture}[baseline=-.5ex]
\draw[thick] (0,1) -- (1,1) -- (1,0) -- (0,0) -- cycle;
\draw[fill,thick] (0,1) circle (2pt) node[above left] {\(a\)};
\draw[fill,thick] (1,1) circle (2pt) node[above right] {\(b\)};
\draw[fill,thick] (1,0) circle (2pt) node[below right] {\(c\)};
\draw[fill,thick] (0,0) circle (2pt) node[below left] {\(d\)};
\end{tikzpicture}}}
&
\Graf&=
\begin{tikzpicture}[baseline=-.5ex]
\draw[fill, thick] (0,0) -- (2,0) circle (2pt);
\draw[fill, thick] (0,0) -- (120:1) circle (2pt) -- (-1,0) circle (2pt) -- (0,0);
\draw[fill, thick] (0,0) -- (-120:1) circle (2pt) -- (-1,0);
\draw[fill, thick] (2,0) -- ++(15:1) circle (2pt) -- ++(135:1) circle (2pt) -- (2,0);
\draw[fill, thick] (2,0) -- ++(-15:1) circle (2pt) -- ++(-135:1) circle (2pt) -- (2,0);
\draw[fill,thick, blue] (0,0) circle (3pt) node[below right, black] {\(\sfw_1\)};
\draw[fill,thick, blue] (2,0) circle (3pt) node[below left, black] {\(\sfw_2\)};
\draw (-0.5,0.25) node {\(\ell_a\)};
\draw (-0.5,-0.25) node {\(\ell_c\)};
\draw (2.35,0.35) node {\(\ell_b\)};
\draw (2.35,-0.35) node {\(\ell_d\)};
\end{tikzpicture}
\end{align*}
\caption{A generalized \(\Lambda\)-halo graph for \(\Lambda=C_4\) and \(n=2\), with respect to
the proper coloring \(\chi(a)=\chi(c)=1\) and \(\chi(b)=\chi(d)=2\).}
\label{Fig:GeneralizedHaloGraph}
\end{figure}
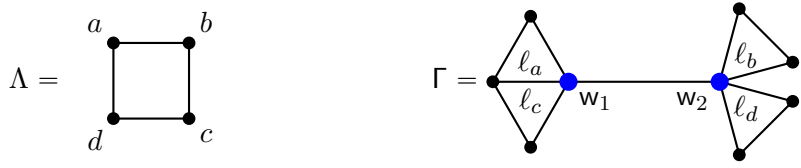

\begin{proposition}[Undistorted RAAGs from generalized halo graphs]\label{Prop:UndistortedRAAGsFromGeneralizedHaloGraphs}
Let \(n\geq2\), let \(\Lambda\) be a finite simplicial graph equipped with a proper coloring \(\chi\colon V(\Lambda)\rightarrow\{1,\dots,n\}\), and let \(\Graf\) be a generalized \(\Lambda\)-halo graph with respect to \(\chi\). For each \(v\in V(\Lambda)\), let \(\gamma_v\in\GBG_n(\Graf)\) be the corresponding Artin loop braid.

There exists \(D>0\) such that, for every integer \(d\geq D\), the assignment
\(v\mapsto\gamma_v^d\) extends to an injective homomorphism \(\bbA(\Lambda)\rightarrow\GBG_n(\Graf)\) whose image is undistorted.
\end{proposition}

\begin{proof}
Choose a sufficiently subdivided model of \(\Graf\) retaining the color vertices and supporting cycles, and continue to denote it by \(\Graf\). By \Cref{Thm:DiscreteConfigurationRetraction}, it suffices to work in \(G:=\GBGdisc_n(\Graf)\). By \Cref{Rem:AuxiliaryChoicesLegalPairs}, after a change of basepoint we may take \(\bfx^0_{\mathrm{Artin}}\) as the global base configuration and regard the Artin loop braids \(\gamma_v\) as elements of \(G\).

For each \(v\in V(\Lambda)\), let \((\sfL_v,\bfn_v)\) be the legal pair whose unique non-singleton component is \(\ell_v\), carrying one particle, and whose remaining components are the singleton color vertices \(\{\sfw_j\}\), \(j\neq\chi(v)\). Let \(F_v\) be the lift of its legal-pair subcomplex through the chosen lift of \(\bfx^0_{\mathrm{Artin}}\), and set \(U_v=[F_v]\).
By \Cref{Lem:StandardRepresentativeParallelismClass,Thm:MinimalUnboundedDomainsGBG}, each \(U_v\) is nesting-minimal unbounded. Since the supporting cycles are pairwise distinct, \Cref{lemma:parallelism} implies that the domains \(U_v\) are pairwise distinct.

Since \(\UD_{\bfn_v}(\sfL_v)\cong\UD_1(\ell_v)\cong\ell_v\), the Artin loop represents a generator of its fundamental group.
By \Cref{Cor:UndistortedLegalPairSubgroups}, its image \(\gamma_v\in G\) is therefore nontrivial.
Since the loop representing \(\gamma_v\) lies in the corresponding legal-pair subcomplex, its lift starting at the chosen basepoint is contained in \(F_v\). Hence
\[\gamma_v\in\Stab_G(F_v)=G_{U_v}\]
by \Cref{Lem:LegalPairGU}. Therefore \Cref{Cor:MinimalDomainStabilizers} shows that \(\gamma_v\) is strongly fully supported on \(U_v\), with \(\B(\gamma_v)=\{U_v\}\).

For distinct \(v,w\in V(\Lambda)\), the generalized halo conditions, together with \Cref{Lem:NestingOrthogonality}, give
\[U_v\bot U_w
\qquad\Longleftrightarrow\qquad
\ell_v\cap\ell_w=\varnothing
\qquad\Longleftrightarrow\qquad
\{v,w\}\in E(\Lambda).\]
Indeed, when the cycles are disjoint, properness of the coloring places each cycle in the orthogonal subgraph of the other. Thus the \(U_v\) span an induced copy of \(\Lambda\) in \(\EOLP\), and the collection \(\{\gamma_v\mid v\in V(\Lambda)\}\) is geometrically irredundant.

Finally, \((G,\frSLP)\in\Xi\) by \Cref{Thm:GBGLegalPairXi}. The conclusion now follows from \Cref{Prop:RAAGEmbeddingTheorem} and \Cref{Thm:DiscreteConfigurationRetraction}.
\end{proof}

\begin{remark}[Comparison with Sabalka's halo graphs]\label{Rem:ComparisonWithSabalkaHalo}
Sabalka's halo graphs satisfy stronger intersection conditions. For each color, there is a distinguished color vertex lying on every supporting cycle of that color and on no supporting cycle of another color. Supporting cycles associated to adjacent vertices are disjoint, whereas those associated to distinct non-adjacent vertices meet in exactly one vertex. Consequently, every halo graph in Sabalka's sense is a generalized halo graph.

Our definition retains only the intersection--disjointness pattern detected by the legal-pair hierarchy and permits arbitrary nonempty intersections between the supporting cycles of non-adjacent vertices. For example, the cycles \(\ell_a\) and \(\ell_c\) in \Cref{Ex:GeneralizedHaloC4} meet in an edge, so that example is not a halo graph in Sabalka's sense.

Sabalka proved that the squares of the Artin loop braids generate a subgroup isomorphic to \(\bbA(\Lambda)\) \cite[Proposition~3.1]{Sab07Embedding}. By contrast, \Cref{Prop:UndistortedRAAGsFromGeneralizedHaloGraphs} shows that sufficiently large common powers generate an undistorted copy of \(\bbA(\Lambda)\), but does not determine whether the subgroup generated by the squares is undistorted.
\end{remark}

\begin{corollary}[Universal undistorted RAAG embeddings]\label{Cor:EveryRAAGUndistortedGBG}
Let \(\Lambda\) be a finite simplicial graph. For every integer \(n\geq\chi(\Lambda)\), there exists a connected finite graph \(\Graf\) such that \(\bbA(\Lambda)\) is isomorphic to an undistorted subgroup of \(\GBG_n(\Graf)\).
\end{corollary}
\begin{proof}
Suppose first that \(n=1\). Then \(\Lambda\) is discrete. Writing \(N=\#V(\Lambda)\), let \(\mathsf R_N\) be a subdivided bouquet of \(N\) circles. Then \(\bbA(\Lambda)\cong\bbF_N\cong\pi_1(\mathsf R_N)\cong\GBG_1(\mathsf R_N)\).

Now suppose that \(n\geq2\), and set \(c=\chi(\Lambda)\). Choose a proper coloring \(\chi\colon V(\Lambda)\rightarrow\{1,\dots,c\}\). Sabalka's construction gives a \(\Lambda\)-halo graph \(\Graf_0\) with color vertices \(\sfw_1,\dots,\sfw_c\) and supporting cycles satisfying the halo intersection conditions \cite[Section~2 and Theorem~1.1]{Sab07Embedding}.
If \(n>c\), attach \(n-c\) new pendant vertices \(\sfw_{c+1},\dots,\sfw_n\) to \(\Graf_0\), and denote the resulting connected graph by \(\Graf\). If \(n=c\), set \(\Graf=\Graf_0\). Regard the original coloring as a map into \(\{1,\dots,n\}\).

The new color vertices lie on none of the supporting cycles, and the supporting cycles retain their original intersection pattern. Therefore \(\Graf\), with color vertices \(\sfw_1,\dots,\sfw_n\), is a generalized \(\Lambda\)-halo graph in the sense of \Cref{Def:GeneralizedHaloGraph}. Applying \Cref{Prop:UndistortedRAAGsFromGeneralizedHaloGraphs} gives an undistorted subgroup of \(\GBG_n(\Graf)\) isomorphic to \(\bbA(\Lambda)\).
\end{proof}

\section{Graph \texorpdfstring{\(2\)}{2}-braid groups}\label{Section:Graph2BraidGroups}

Throughout this section, let \(\Graf\) be a connected finite simplicial graph. If \(\Graf\) is a path, then \(\GBG_2(\Graf)\) is trivial, and the conclusions below are immediate. We therefore assume that \(\Graf\) is not a path, set \(G=\GBGdisc_2(\Graf)\), and equip \(G\) with its legal-pair hierarchy \((G,\frSLP)\). We write \(\EO\) for the corresponding expanded core graph. Since \(\Graf\) is not a path, it is sufficiently subdivided for two particles, and hence
\[G=\GBGdisc_2(\Graf)\cong\GBG_2(\Graf).\]
Moreover, \Cref{Thm:RankTwoGBG} shows that this hierarchy has rank at most \(2\), so \Cref{Thm:RankTwoCriterionHHG} identifies the RAAG-embedding problem for \(\GBG_2(\Graf)\) with the induced-subgraph problem for \(\EO\).

\subsection{Cycle coordinates}

We first determine which nesting-minimal unbounded domains correspond to non-isolated vertices of \(\EO\).

\begin{lemma}[Non-isolated vertices]\label{Lem:NonisolatedCycleDomains}
Let \(v\in V(\EO)\), and let \(U\in\frSLP\) be its underlying domain. Then \(v\) is non-isolated if and only if \(U\) is represented by a coset \(g\GBGdisc_{\bfn}(\sfL)\), where \((\sfL,\bfn)\) is a standard legal pair 
whose constrained part is an embedded cycle \(\sfC\) carrying one particle and 
whose free part is a single vertex lying in a component \(\sfC^\perp\in\pi_0(\Graf\setminus\sfC)\) that contains an embedded cycle.
\end{lemma}
\begin{proof}
By \Cref{Thm:MinimalUnboundedDomainsGBG}, the unique non-singleton constrained component of \((\sfL,\bfn)\) is either an embedded cycle or a tripod. By \Cref{Cor:NonisolatedVerticesGBG}, \(v\) is non-isolated exactly when \(\GBGdisc_{\bfn^\perp}(\sfL^\perp)\) is infinite.
Since \(n=2\), this excludes a tripod and a cycle carrying two particles, and forces an embedded cycle \(\sfC\) carrying one particle together with one free particle in a component \(\sfC^\perp\) of \(\Graf\setminus\sfC\). Thus \((\sfL^\perp,\bfn^\perp)=(\sfC^\perp,\bfone)\).
Since \(\Graf\) is sufficiently subdivided for two particles, the final criterion in \Cref{Cor:NonisolatedVerticesGBG} says precisely that \(\sfC^\perp\) contains an embedded cycle.
\end{proof}

In view of \Cref{Lem:NonisolatedCycleDomains}, we introduce coordinates for the non-isolated vertices of \(\EO\).

\begin{definition}[Cycle data and coordinates]\label{Def:CycleCoordinates}
A \emph{cycle datum} in \(\Graf\) is a pair \((\sfC,\sfC^\perp)\), where \(\sfC\subseteq\Graf\) is an embedded cycle and \(\sfC^\perp\) is a component of \(\Graf\setminus\sfC\) containing an embedded cycle.

Let \(\sfv_{\sfC}\) be the vertex of \(\sfC\) with smallest label, let \(\sfv_{\sfC^\perp}\) be the vertex of \(\sfC^\perp\) with smallest label, and set
\[\bfx_{\sfC,\sfC^\perp}=\{\sfv_{\sfC},\sfv_{\sfC^\perp}\}.\]
Let \((\sfL_{\sfC,\sfC^\perp},\bfn_{\sfC,\sfC^\perp})\) be the standard legal pair defined by
\[\sfL_{\sfC,\sfC^\perp}=\sfC\sqcup\sfv_{\sfC^\perp},\qquad
\bfn_{\sfC,\sfC^\perp}(\sfC)=1,\qquad
\bfn_{\sfC,\sfC^\perp}(\sfv_{\sfC^\perp})=1.\]
We write
\[H_{\sfC,\sfC^\perp}=\GBGdisc_{\bfn_{\sfC,\sfC^\perp}}(\sfL_{\sfC,\sfC^\perp})<G\]
for the corresponding standard discrete graph braid subgroup, and write
\[\gamma_{\sfC,\sfC^\perp}=\gamma_{\sfL_{\sfC,\sfC^\perp},\bfn_{\sfC,\sfC^\perp}}\]
for the corresponding standard path.

For \(g\in G\), the triple \((g,\sfC,\sfC^\perp)\) is called a \emph{cycle coordinate}. It represents the non-isolated vertex of \(\EO\) associated to the coset \(gH_{\sfC,\sfC^\perp}\).
\end{definition}

By \Cref{Lem:NonisolatedCycleDomains}, every non-isolated vertex of \(\EO\) admits a cycle coordinate, although such a coordinate need not be unique.

For later use, we record explicitly when two cycle coordinates determine adjacent vertices. The first condition below is the downstairs orthogonality condition, while the second records the relative positions of the chosen lifts.

\begin{lemma}[Adjacency in cycle coordinates]\label{Lem:CycleCoordinateOrthogonality}
Let \(v_1=(g_1,\sfC_1,\sfC_1^\perp)\) and \(v_2=(g_2,\sfC_2,\sfC_2^\perp)\) be vertices of \(\EO\) represented by cycle coordinates. Then \(v_1\) and \(v_2\) are adjacent if and only if the following conditions hold:
\begin{enumerate}[label=\textup{(\arabic*)}, ref=\arabic*]
\item\label{Item:CycleCoordinateDownstairs}
\(\sfC_1\cap\sfC_2=\varnothing\), \(\sfC_1\subseteq \sfC_2^\perp\) and \(\sfC_2\subseteq \sfC_1^\perp\).
\item\label{Item:CycleCoordinatePath}
There exists a comparison path \(\eta\) from \((\sfL_{\sfC_1,\sfC_1^\perp},\bfn_{\sfC_1,\sfC_1^\perp})\) to \((\sfL_{\sfC_2,\sfC_2^\perp},\bfn_{\sfC_2,\sfC_2^\perp})\) such that
\[g_1^{-1}g_2\in H_{\sfC_1,\sfC_1^\perp}\,\delta_\eta\,H_{\sfC_2,\sfC_2^\perp},\]
where \(\delta_\eta\) is defined in \eqref{Eq:Delta_eta}.
\end{enumerate}

Moreover, suppose that two cycle data \((\sfC_1,\sfC_1^\perp)\) and \((\sfC_2,\sfC_2^\perp)\) satisfy \textup{(1)}. Then, for every \(g_1\in G\), there exists \(g_2\in G\) such that the vertices represented by \((g_1,\sfC_1,\sfC_1^\perp)\) and \((g_2,\sfC_2,\sfC_2^\perp)\) are adjacent.
\end{lemma}
\begin{proof}
The vertices are adjacent precisely when their underlying domains are orthogonal. For \(i=1,2\), the corresponding standard legal pair has \(\sfC_i\) as its only constrained component, carrying one particle, and has orthogonal pair \((\sfC_i^\perp,\bfone)\). Hence \Cref{Lem:NestingOrthogonality}\eqref{Item:OrthogonalRelation} is equivalent to \(\sfC_1\subseteq\sfC_2^\perp\) and \(\sfC_2\subseteq\sfC_1^\perp\).
These containments imply \(\sfC_1\cap\sfC_2=\varnothing\), so this is precisely \textup{(1)}.

Under this condition, the two closure subcomplexes intersect: for \(p_i\in V(\sfC_i)\), the configuration \(\{p_1,p_2\}\) belongs to both. Thus the two standard legal pairs are compatible and comparison paths between them exist. The remaining condition in \Cref{Lem:NestingOrthogonality} is exactly \textup{(2)}, proving the adjacency criterion.

For the final assertion, choose a comparison path \(\eta\) and set \(g_2=g_1\delta_\eta\). Then \textup{(2)} holds, so the resulting vertices are adjacent.
\end{proof}

See \Cref{Ex:ColorSwap} below for an explicit pair of adjacent cycle coordinates.

\subsection{The bipartite obstruction}
We next define a two-coloring of the non-isolated vertices. 
Choose the ordering \(\bar\bfx^0=(\sfv_1,\sfv_2)\) of the base configuration \(\bfx^0\).
Lifting a loop in \(\UD_2(\Graf)\) to \(\D_2(\Graf)\) defines the particle-permutation homomorphism 
\[\sgn\colon G\longrightarrow\mathfrak S_2\cong\{1,-1\}.\]
For a vertex configuration \(\bfx=\{\sfv,\sfw\}\in \UD_2(\Graf)\), let \(\bar\gamma_{\bfx}\) be the lift of the standard path \(\gamma_{\bfx}\) to \(\D_2(\Graf)\) starting at \(\bar\bfx^0\).
Define \(\sgn_{\bfx}\colon\bfx\rightarrow\{1,-1\}\) by declaring, for \(\sfz\in\bfx\),
\[\sgn_{\bfx}(\sfz)=
\begin{cases}
1 & \text{if the particle starting at \(\sfv_1\) ends at \(\sfz\),}\\
-1 & \text{if the particle starting at \(\sfv_2\) ends at \(\sfz\).}
\end{cases}
\]

Let \(v=(g,\sfC,\sfC^\perp)\) be a non-isolated vertex of \(\EO\). We define
\[\varepsilon(v)=\sgn_{\bfx_{\sfC,\sfC^\perp}}(\sfv_\sfC)\sgn(g)\in\{1,-1\}.\]
This is independent of the choice of the representative \(g\) of the coset \(gH_{\sfC,\sfC^\perp}\), since every element of \(H_{\sfC,\sfC^\perp}\) is pure, i.e., \(H_{\sfC,\sfC^\perp}<\PGBGdisc_2(\Graf)\).
We also verify that \(\varepsilon(v)\) is independent of the choice of cycle coordinate. Suppose that \((g_1,\sfC_1,\sfC_1^\perp)\) and \((g_2,\sfC_2,\sfC_2^\perp)\) represent the same vertex.
By \Cref{lemma:parallelism}, 
\[\sfC_1=\sfC_2=:\sfC\qquad\text{and}\qquad\sfC_1^\perp=\sfC_2^\perp=:\sfC^\perp.\]
Moreover, there is a loop \(\eta\) in \(\sfC^\perp\), based at \(\sfv_{\sfC^\perp}\), such that
\[g_1^{-1}g_2\in H_{\sfC,\sfC^\perp}\left[\gamma_{\sfC,\sfC^\perp}\cdot\bar\eta\cdot\gamma_{\sfC,\sfC^\perp}^{-1}\right]H_{\sfC,\sfC^\perp},\]
where, as in \Cref{lemma:parallelism}, \(\bar\eta(t)=\{\sfv_{\sfC},\eta(t)\}\). Thus \(\bar\eta\) keeps the constrained particle fixed and moves only the free particle. Both endpoint subgroups consist of pure braids, and the middle term is also pure. Hence \(\sgn(g_1)=\sgn(g_2)\), so the two coordinates determine the same value of \(\varepsilon(v)\).

\begin{proposition}[Bipartiteness of the expanded core graph]
\label{Prop:GBG2CoreGraphBipartite}
The expanded core graph \(\EO\) is bipartite.
\end{proposition}
\begin{proof}
It suffices to color the non-isolated vertices, since isolated vertices may be assigned either color.
Let \(v_1=(g_1,\sfC_1,\sfC^{\perp}_1)\) and \(v_2=(g_2,\sfC_2,\sfC^{\perp}_2)\) be adjacent non-isolated vertices. By \Cref{Lem:CycleCoordinateOrthogonality}, there exists a comparison path \(\eta\) such that
\[g_1^{-1}g_2\in H_{\sfC_1,\sfC^{\perp}_1}\,\delta_\eta\,H_{\sfC_2,\sfC^{\perp}_2}.\]
Since the two subgroups appearing on the right consist of pure braids, \(\sgn(g_1^{-1}g_2)=\sgn(\delta_\eta)\).

The inclusions \(\sfC_1\subseteq \sfC^{\perp}_2\) and \(\sfC_2\subseteq \sfC^{\perp}_1\) imply that the comparison path interchanges the roles of the constrained and free particles. More precisely, the particle lying on \(\sfC_1\) at the initial configuration becomes the free particle at the terminal configuration, while the initial free particle becomes the particle lying on \(\sfC_2\). Therefore
\[\sgn(\delta_\eta)=-\sgn_{\bfx_{\sfC_1,\sfC^{\perp}_1}}(\sfv_{\sfC_1})\sgn_{\bfx_{\sfC_2,\sfC^{\perp}_2}}(\sfv_{\sfC_2}).\]
Combining the two equalities gives \(\varepsilon(v_1)=-\varepsilon(v_2)\). Thus every edge between non-isolated vertices joins vertices with opposite \(\varepsilon\)-values. Assigning isolated vertices arbitrarily to either side gives a bipartition of \(\EO\).
\end{proof}

\begin{example}[The sign change in an explicit adjacency]\label{Ex:ColorSwap}
Let \(\Graf\) be the graph of \Cref{Fig:MaximalTreeAndLabels}, with the maximal tree \(\sfT\), the leaf \(*=\sfv_1\) and the labels fixed there, and let \(n=2\).
Let \(\sfC_1\) be the triangle on \(\sfv_4,\sfv_5,\sfv_6\) and let \(\sfC_2\) be the triangle on \(\sfv_7,\sfv_8,\sfv_9\), as in \Cref{Fig:ColorSwap}.
Both \(\Graf\setminus\sfC_1\) and \(\Graf\setminus\sfC_2\) are connected, so the selected components \(\sfC_1^\perp\) and \(\sfC_2^\perp\) are determined
by the cycles alone, and each \(\sfC_i^\perp\) contains an embedded cycle. Since \(\sfC_1\cap\sfC_2=\varnothing\), \(\sfC_1\subseteq \sfC^{\perp}_2\) and \(\sfC_2\subseteq \sfC^{\perp}_1\), \Cref{Lem:CycleCoordinateOrthogonality}\eqref{Item:CycleCoordinateDownstairs} holds.

Taking in each case the vertices of smallest label gives
\[\bfx^0=\{\sfv_1,\sfv_2\},\qquad
\bfx_{\sfC_1,\sfC^{\perp}_1}=\{\sfv_1,\sfv_4\},\qquad
\bfx_{\sfC_2,\sfC^{\perp}_2}=\{\sfv_1,\sfv_7\}.\]
Since \(n=2\), the configuration \(\bfx^0\) is the unique critical \(0\)-cell by \Cref{Rem:UniqueCritical0Cell}, so every standard path is the reverse of a vertex-reduction path. Explicitly, both \(\gamma_{\sfC_1,\sfC^{\perp}_1}\) and \(\gamma_{\sfC_2,\sfC^{\perp}_2}\) leave one particle fixed at \(\sfv_1\) and move the other along the edge paths with vertex sequences \((\sfv_2,\sfv_3,\sfv_4)\) and \((\sfv_2,\sfv_3,\sfv_7)\), respectively.
The closure subcomplexes associated to these two cycle data, as well as the comparison path through \(\mathbf y=\{\sfv_4,\sfv_7\}\), are precisely those described in \Cref{Ex:ComparisonPath}.

\Cref{Fig:ColorSwap} follows the two particles along \(\delta_\eta\).
The particle which starts at \(\sfv_1\) ends at \(\sfv_2\), and the particle which starts at \(\sfv_2\) ends at \(\sfv_1\). Thus this loop is not pure, and \(\sgn(\delta_\eta)=-1\), in accordance with the computation in the proof of \Cref{Prop:GBG2CoreGraphBipartite}: here, 
\[\sgn_{\bfx_{\sfC_1,\sfC^{\perp}_1}}(\sfv_4)=\sgn_{\bfx_{\sfC_2,\sfC^{\perp}_2}}(\sfv_7)=-1.\]
Setting \(g_2=\delta_\eta\), the vertices \(v_1=(1,\sfC_1,\sfC^{\perp}_1)\) and \(v_2=(g_2,\sfC_2,\sfC^{\perp}_2)\) are adjacent, and
\(\varepsilon(v_1)=-1\) while \(\varepsilon(v_2)=+1\).
\end{example}

\begin{figure}[ht]
\newcommand{\GBGedges}{%
\draw[thick] (-0.866,-0.5) -- (0,0) -- (0.866,-0.5);
\draw[thick] (0,0) -- (0,1);
\draw[thick] (-1.366,-1.366) -- (-0.866,-0.5) -- (-1.866,-0.5) -- (-1.366,-1.366);
}
\newcommand{\GBGdots}[1]{%
\foreach \p in {{(0,0)},{(-0.866,-0.5)},{(-1.366,-1.366)},{(-1.866,-0.5)},{(0.866,-0.5)},{(1.366,-1.366)},{(1.866,-0.5)},{(0,1)},{(0.5,1.866)},{(-0.5,1.866)}}
{\fill \p circle (#1);}
}
\newcommand{\GBGpanel}[2]{%
\begin{tikzpicture}[baseline=-.5ex, scale=1]
\GBGedges
\draw[thick] (1.366,-1.366) -- (0.866,-0.5) -- (1.866,-0.5) -- (1.366,-1.366);
\draw[thick] (-0.5,1.866) -- (0,1) -- (0.5,1.866) -- (-0.5,1.866);
\GBGdots{2pt}
\fill[red] #1 circle (3pt);
\fill[blue] #2 circle (3pt);
\end{tikzpicture}}
\newcommand{\pA}{(-1.366,-1.366)}
\newcommand{\pB}{(-0.866,-0.5)}
\newcommand{\pD}{(0.866,-0.5)}
\newcommand{\pG}{(0,1)}
\[
\begin{tikzcd}
\GBGpanel{\pA}{\pB}\ar[d,"\gamma_{\sfC_1,\sfC_1^\perp}"']\ar[rr,equal,yshift=1cm,"\text{in }\UD_2(\Graf)"] & & \GBGpanel{\pB}{\pA}\\
\GBGpanel{\pA}{\pD}\ar[r,"\eta_1"] & \GBGpanel{\pG}{\pD} \ar[r,"\eta_2^{-1}"] & \GBGpanel{\pG}{\pA} \ar[u,"\gamma_{\sfC_2,\sfC_2^\perp}^{-1}"]
\end{tikzcd}
\]
\caption{The sign-changing loop of \Cref{Ex:ColorSwap}. The two top panels represent the same unordered configuration, but the red and blue particles have been interchanged.}
\label{Fig:ColorSwap}
\end{figure}
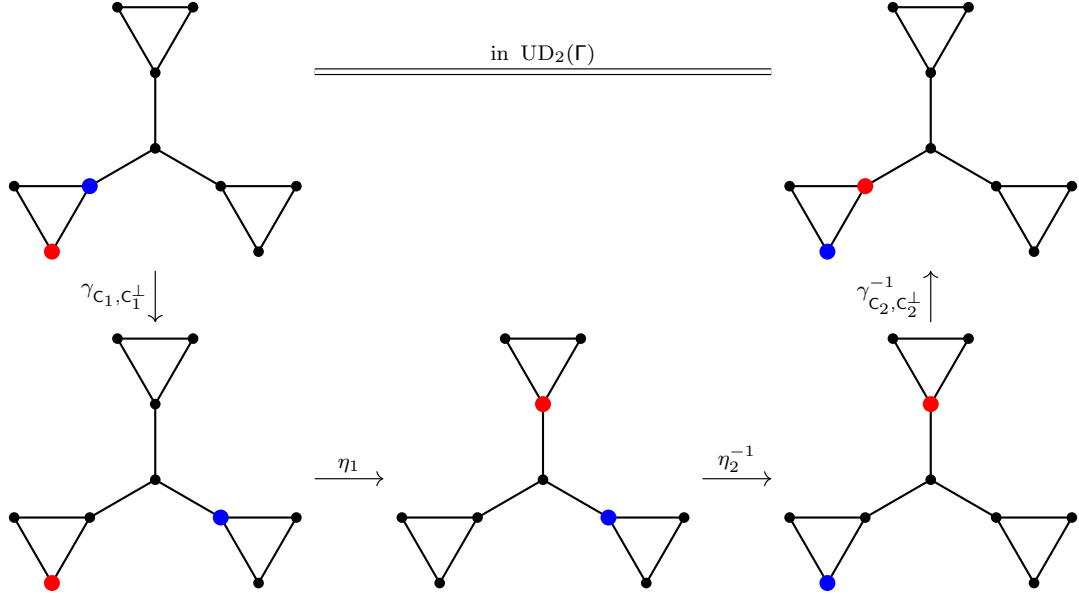

\begin{theorem}[Graph \(2\)-braid obstruction]\label{Thm:GBG2BipartiteObstruction}
Let \(\Lambda\) be a finite graph. If \(\bbA(\Lambda)<\GBG_2(\Graf)\), then \(\Lambda\) is bipartite.
\end{theorem}
\begin{proof}
By \Cref{Thm:RankTwoGBG}, if \(\bbA(\Lambda)<\GBG_2(\Graf)\), then \(\Lambda<\EO\). The conclusion follows from \Cref{Prop:GBG2CoreGraphBipartite}.
\end{proof}

Thus, for example, if \(C_{2k+1}\) is an odd cycle, then \(\bbA(C_{2k+1})\) does not embed into any graph \(2\)-braid group.
The obstruction is sharp when the ambient graph is allowed to vary: Sabalka's halo graph construction realizes every bipartite RAAG inside some graph \(2\)-braid group, and \Cref{Prop:UndistortedRAAGsFromGeneralizedHaloGraphs} shows that sufficiently large powers of the same Artin loop braids generate an undistorted copy.

\subsection{\texorpdfstring{The \(P_4\)-criterion}{The P4-criterion}}
Recall that \(P_4\) denotes the path graph of length \(3\), and label its vertices consecutively by \(a,b,c,d\), in this order. We now fix the ambient graph and characterize when \(\EO\) contains an induced copy of \(P_4\).

\begin{definition}[\(P_4\)-cycle configuration]\label{Def:P4CycleConfiguration}
A \emph{\(P_4\)-cycle configuration} in \(\Graf\) consists of four distinct embedded cycles
\(\sfC_a,\sfC_b,\sfC_c,\sfC_d\subseteq\Graf\) satisfying
\[\sfC_a\cap\sfC_b=\sfC_b\cap\sfC_c=\sfC_c\cap\sfC_d=\varnothing,\qquad
\sfC_a\cap\sfC_c\neq\varnothing,\quad \sfC_a\cap\sfC_d\neq\varnothing,\quad
\sfC_b\cap\sfC_d\neq\varnothing.\]
\end{definition}

Equivalently, with respect to the coloring \(\chi(a)=\chi(c)=1\) and \(\chi(b)=\chi(d)=2\), the
four cycles satisfy the conditions of \Cref{Def:GeneralizedHaloGraph} in the ambient graph
\(\Graf\), with any \(\sfw_1\in\sfC_a\cap\sfC_c\) and \(\sfw_2\in\sfC_b\cap\sfC_d\) as color
vertices; the disjointness conditions force \(\sfw_1\notin\sfC_b\cup\sfC_d\) and
\(\sfw_2\notin\sfC_a\cup\sfC_c\).

\begin{example}[A \(P_4\)-cycle configuration in \(\sfK_7\)]\label{Ex:P4inK7}
\Cref{figure:P_4-cycle_configurations} exhibits a \(P_4\)-cycle configuration in \(\sfK_7\).
Seven vertices are needed: no graph on fewer than seven vertices admits such a configuration.
Indeed, since \(\Graf\) is simplicial, every embedded cycle has at least three vertices, so \(\sfC_b\cap\sfC_c=\varnothing\) already forces \(\#V(\Graf)\geq6\). If \(\#V(\Graf)=6\), then \(\sfC_b\) and \(\sfC_c\) are triangles and \(V(\Graf)=V(\sfC_b)\sqcup V(\sfC_c)\); since \(\sfC_a\cap\sfC_b=\varnothing\), this gives \(V(\sfC_a)\subseteq V(\sfC_c)\), hence \(V(\sfC_a)=V(\sfC_c)\) and therefore \(\sfC_a=\sfC_c\), contrary to the distinctness of the four cycles.
\end{example}

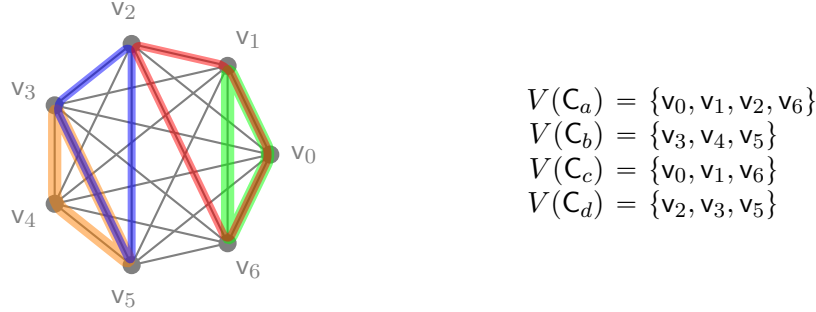
\begin{figure}[ht]
\begin{align*}
&
\begin{tikzpicture}[baseline=-.5ex, scale=1.5]
\foreach \i in {0,...,6} {
\draw[gray,thick,fill] ({\i*360/7}:1) circle (2pt);
\draw[gray,thick,fill] ({\i*360/7}:1.3) node {\(\sfv_\i\)};
\draw[gray,thick] ({\i*360/7}:1) -- ({(\i+1)*360/7}:1);
\draw[gray,thick] ({\i*360/7}:1) -- ({(\i+2)*360/7}:1);
\draw[gray,thick] ({\i*360/7}:1) -- ({(\i+3)*360/7}:1);
}
\draw[line width=5, rounded corners, green, opacity=0.5] ({0*360/7}:1) -- ({1*360/7}:1) -- ({-1*360/7}:1) -- cycle;
\draw[line width=5, rounded corners, orange, opacity=0.5] ({3*360/7}:1) -- ({4*360/7}:1) -- ({5*360/7}:1) -- cycle;
\draw[line width=3, rounded corners, blue, opacity=0.5] ({2*360/7}:1) -- ({3*360/7}:1) -- ({5*360/7}:1) -- cycle;
\draw[line width=3, rounded corners, red, opacity=0.5] ({2*360/7}:1) -- ({-1*360/7}:1) -- ({0*360/7}:1) -- ({1*360/7}:1) -- cycle;
\end{tikzpicture}
&
\begingroup
\setlength{\arraycolsep}{2pt}
\begin{array}{rcl}
V(\sfC_a)&=&\{\sfv_0,\sfv_1,\sfv_2,\sfv_6\}\\
V(\sfC_b)&=&\{\sfv_3,\sfv_4,\sfv_5\}\\
V(\sfC_c)&=&\{\sfv_0,\sfv_1,\sfv_6\}\\
V(\sfC_d)&=&\{\sfv_2,\sfv_3,\sfv_5\}
\end{array}
\endgroup
\end{align*}
\caption{A \(P_4\)-cycle configuration in \(\sfK_7\).}
\label{figure:P_4-cycle_configurations}
\end{figure}

\begin{proposition}[Two graphical sources of induced copies of \(P_4\)]\label{Prop:P4SourcesGBG2}
\(P_4<\EO\) if \(\Graf\) contains one of the following:
\begin{enumerate}
\item a \(P_4\)-cycle configuration;
\item three pairwise disjoint embedded cycles.
\end{enumerate}
\end{proposition}

\begin{proof}
Suppose first that \(\Graf\) contains a \(P_4\)-cycle configuration. By the observation following \Cref{Def:P4CycleConfiguration}, \(\Graf\) is a generalized \(P_4\)-halo graph with respect to the coloring \(\chi(a)=\chi(c)=1\) and \(\chi(b)=\chi(d)=2\), so it satisfies the hypotheses of \Cref{Prop:UndistortedRAAGsFromGeneralizedHaloGraphs}. Hence \(\bbA(P_4)<\GBG_2(\Graf)\). It follows from \Cref{Thm:RankTwoGBG} that \(P_4<\EO\).

Now suppose that \(\Graf\) contains three pairwise disjoint embedded cycles \(\sfC_1,\sfC_2,\sfC_3\).
After relabeling them, there exist components \(\sfC^{\perp}_1\in\pi_0(\Graf\setminus\sfC_1)\) and \(\sfC^{\perp}_3\in\pi_0(\Graf\setminus\sfC_3)\) such that \(\sfC_2,\sfC_3\subseteq \sfC^{\perp}_1\) and \(\sfC_1,\sfC_2\subseteq \sfC^{\perp}_3\).
Indeed, contract the three cycles to vertices, choose a minimal tree joining the resulting vertices, and take \(\sfC_1\) and \(\sfC_3\) to correspond to two leaves of this tree.

Let \(\sfC^{\perp}_{2,1}\) and \(\sfC^{\perp}_{2,3}\) be the components of \(\Graf\setminus\sfC_2\) containing \(\sfC_1\) and \(\sfC_3\), respectively. They are allowed to coincide.
Start with any vertex \(v_1=(g_1,\sfC_1,\sfC^{\perp}_1)\). By the final assertion of \Cref{Lem:CycleCoordinateOrthogonality}, choose vertices \(v_2=(g_2,\sfC_2,\sfC^{\perp}_{2,1})\) and \(v_3=(g_3,\sfC_3,\sfC^{\perp}_3)\) which are both adjacent to \(v_1\). Applying the same lemma once more, choose \(v_4=(g_4,\sfC_2,\sfC^{\perp}_{2,3} )\) adjacent to \(v_3\).

Since \(v_2\) and \(v_3\) are both adjacent to \(v_1\), they have the same color in the bipartition from \Cref{Prop:GBG2CoreGraphBipartite}, and hence are not adjacent. Similarly, \(v_1\) and \(v_4\) are both adjacent to \(v_3\), so they have the same color and are not adjacent.
On the other hand, \Cref{Lem:CycleCoordinateOrthogonality}\eqref{Item:CycleCoordinateDownstairs} shows that \(v_2\) and \(v_4\) are not adjacent, since their constrained cycles are both equal to \(\sfC_2\). Moreover, \(v_2\) and \(v_4\) are distinct: \(v_2\) has the same color as \(v_3\), whereas \(v_4\) is adjacent to \(v_3\) and hence has the opposite color. Consequently, \(v_2,v_1,v_3,v_4\) span an induced \(P_4\).
\end{proof}

We now prove the remaining implication.

\begin{lemma}[Cycle extraction from an induced \(P_4\)]\label{Lem:P4CycleExtraction}
Suppose that \(\Graf\) contains no three pairwise disjoint embedded cycles. If \(P_4<\EO\), then \(\Graf\) contains a \(P_4\)-cycle configuration.
\end{lemma}
\begin{proof}
Let \(v_a,v_b,v_c,v_d\) be the vertices of an induced copy of \(P_4\) in \(\EO\), corresponding respectively to the consecutively ordered vertices \(a,b,c,d\) of \(P_4\).
Since none of these vertices is isolated, choose cycle coordinates \(v_x=(g_x,\sfC_x,\sfC^{\perp}_x)\) for \(x\in\{a,b,c,d\}\). By \Cref{Lem:CycleCoordinateOrthogonality}\eqref{Item:CycleCoordinateDownstairs}, consecutive cycles are disjoint, and \(\sfC_a,\sfC_c\subseteq \sfC_b^{\perp}\), \(\sfC_b,\sfC_d\subseteq \sfC_c^{\perp}\).

Let \(U_x\in\frSLP\) be the domain underlying \(v_x\). Since \(\EO\cong\mathcal G^{\frSLP}\) by \Cref{Cor:CoreEqualsExpandedCore}, the four distinct vertices \(v_a,v_b,v_c,v_d\) have pairwise distinct underlying domains \(U_a,U_b,U_c,U_d\).
Since \(U_b\bot U_c\), \Cref{Lem:OrthogonalLegalPairProductRegions} gives \(\mathbf P_{U_b}\cap\mathbf P_{U_c}\neq\varnothing\).
Choose a vertex \(\widetilde{\mathbf y}\in\mathbf P_{U_b}\cap\mathbf P_{U_c}\) and write its projection to \(\UD_2(\Graf)\) as \(\mathbf y=\{p_b,p_c\}\) with \(p_b\in\sfC_b\) and \(p_c\in\sfC_c\).
Choose an edge path \(\widetilde\beta\) from \(\widetilde{\bfx}^0\) to \(\widetilde{\mathbf y}\), and let \(\beta\) be its projection to \(\UD_2(\Graf)\).
We use the resulting change-of-basepoint isomorphism
\[\pi_1\bigl(\UD_2(\Graf),\mathbf y\bigr)\longrightarrow G=\pi_1\bigl(\UD_2(\Graf),\bfx^0\bigr),\qquad[\omega]\longmapsto[\beta\cdot\omega\cdot\beta^{-1}].\]
Under this identification, the maps
\[j_b\colon\sfC_b^\perp\longrightarrow\UD_2(\Graf),\quad q\longmapsto\{p_b,q\},
\qquad
j_c\colon\sfC_c^\perp\longrightarrow\UD_2(\Graf),\quad q\longmapsto\{q,p_c\},\]
are cubical embeddings which are local isometries by \Cref{Lem:LocalisometrybetweenLegalPairs}. They therefore induce injective homomorphisms
\[\iota_b\colon\pi_1(\sfC_b^{\perp},p_c)\longrightarrow G,\qquad \iota_c\colon\pi_1(\sfC_c^{\perp},p_b)\longrightarrow G.\]

Let \(E_b(\widetilde{\mathbf y})\) denote the second-factor fiber of \(\mathbf P_{U_b}\) through \(\widetilde{\mathbf y}\). By \Cref{Lem:LiftedClosureRegion}, it is the lift through \(\widetilde{\mathbf y}\) of the image of \(j_b\). Define \(E_c(\widetilde{\mathbf y})\) similarly. 
Since \(U_a\bot U_b\), \Cref{Lem:OrthogonalLegalPairProductRegions} gives a factor representing \(U_a\) in some second-factor fiber of \(\mathbf P_{U_b}\). Parallel transport across the first factor, together with \Cref{Lem:LiftedClosureRegion,Lem:ParallelTransportLegalPair}, therefore gives a factor \(F_a\subseteq E_b(\widetilde{\mathbf y})\) representing \(U_a\). Choose a path in \(E_b(\widetilde{\mathbf y})\) from \(\widetilde{\mathbf y}\) to a vertex of \(F_a\). Conjugating a nontrivial cycle translation stabilizing \(F_a\) along the projection of this path gives an axial element \(h_a\in\iota_b(\pi_1(\sfC_b^\perp,p_c))\) strongly fully supported on \(U_a\).
Similarly, \(U_d\bot U_c\) gives an axial element \(h_d\in\iota_c(\pi_1(\sfC_c^{\perp},p_b))\) strongly fully supported on \(U_d\).
Since \(U_a\neq U_d\) and \(U_a\not\bot U_d\), the pair \(\{h_a,h_d\}\) is geometrically irredundant. Hence \Cref{Prop:RAAGEmbeddingTheorem} gives \(\langle h_a^N,h_d^N\rangle\cong\bbF_2\)
for all sufficiently large \(N\). In particular, \([h_a,h_d]\neq1\).

Let \(\mathsf{A}_b\subseteq \sfC_b^{\perp}\) be the union of all embedded cycles contained in \(\sfC_b^{\perp}\), and define \(\mathsf{A}_c\subseteq \sfC_c^{\perp}\) similarly.

\begin{claim}\label{Claim:CoreIntersection}
Both \(\mathsf{A}_b\) and \(\mathsf{A}_c\) are connected, and \(\mathsf{A}_b\cap \mathsf{A}_c\neq\varnothing\).
\end{claim}
\begin{proof}[Proof of \Cref{Claim:CoreIntersection}]
Every embedded cycle in \(\sfC_b^{\perp}\) meets \(\sfC_c\). Indeed, a cycle in \(\sfC_b^{\perp}\) disjoint from \(\sfC_c\) would, together with \(\sfC_b\) and \(\sfC_c\), give three pairwise disjoint embedded cycles. Since \(\sfC_c\subseteq \sfC_b^{\perp}\), it follows that \(\mathsf{A}_b\) is connected. Moreover, \(\sfC_b^{\perp}\) is obtained from \(\mathsf{A}_b\) by attaching a forest. Hence the inclusion induces an isomorphism \(\pi_1(\mathsf{A}_b,p_c)\rightarrow\pi_1(\sfC_b^{\perp},p_c)\).
Similarly, \(\mathsf{A}_c\) is connected, \(\sfC_c^{\perp}\) is obtained from \(\mathsf{A}_c\) by attaching a forest, and the inclusion induces an isomorphism \(\pi_1(\mathsf{A}_c,p_b)\rightarrow\pi_1(\sfC_c^{\perp},p_b)\).
Here the basepoints lie in the corresponding subgraphs since \(p_c\in\sfC_c\subseteq \mathsf{A}_b\) and \(p_b\in\sfC_b\subseteq \mathsf{A}_c\).
We may therefore choose based loops
\[\rho_a\colon([0,1],\{0,1\})\longrightarrow(\mathsf{A}_b,p_c),\qquad\rho_d\colon([0,1],\{0,1\})\longrightarrow(\mathsf{A}_c,p_b)\]
such that \(\iota_b([\rho_a])=h_a\) and \(\iota_c([\rho_d])=h_d\).

Suppose, to the contrary, that \(\mathsf{A}_b\cap \mathsf{A}_c=\varnothing\). Then there is a product subcomplex \(\mathsf{A}_b\times \mathsf{A}_c\subseteq\D_2(\Graf)\) with ordered base configuration \((p_c,p_b)\). The loops \(t\mapsto(\rho_a(t),p_b)\) and \(t\mapsto(p_c,\rho_d(t))\) lie in the two different factors of \(\mathsf{A}_b\times \mathsf{A}_c\), and hence commute in its fundamental group. After projecting to \(\UD_2(\Graf)\) and applying the common change-of-basepoint isomorphism, they represent \(h_a\) and \(h_d\), respectively. It follows that \([h_a,h_d]=1\), contradicting the noncommutation established above.
Therefore \(\mathsf{A}_b\cap \mathsf{A}_c\neq\varnothing\).
\end{proof}

By \Cref{Claim:CoreIntersection}, choose a vertex \(z\in \mathsf{A}_b\cap \mathsf{A}_c\). Since \(\mathsf{A}_b\) is the union of the embedded cycles contained in \(\sfC_b^{\perp}\), there is an embedded cycle \(\alpha\subseteq \sfC_b^{\perp}\) containing \(z\). Similarly, there is an embedded cycle \(\delta\subseteq \sfC_c^{\perp}\) containing \(z\). Consequently, \(z\in\alpha\cap\delta\), and in particular \(\alpha\cap\delta\neq\varnothing\).
These cycles need not coincide with the original endpoint cycles \(\sfC_a\) and \(\sfC_d\).
Since \(\alpha\subseteq \sfC_b^{\perp}\) and \(\delta\subseteq \sfC_c^{\perp}\), we have 
\[\alpha\cap\sfC_b=\sfC_b\cap\sfC_c=\sfC_c\cap\delta=\varnothing.\]
The absence of three pairwise disjoint embedded cycles now implies that 
\begin{equation}\label{Eq:Intersection}
\alpha\cap\sfC_c\neq\varnothing\qquad\text{and}\qquad\sfC_b\cap\delta\neq\varnothing.  
\end{equation}
Indeed, otherwise either \(\alpha,\sfC_b,\sfC_c\) or \(\sfC_b,\sfC_c,\delta\) would be a collection of three pairwise disjoint embedded cycles.

The four cycles \(\alpha,\sfC_b,\sfC_c,\delta\) are distinct.
Consecutive cycles are distinct by the displayed disjointness relations. Moreover, if \(\alpha=\sfC_c\), then \(\alpha\cap\delta\neq\varnothing\) would contradict \(\sfC_c\cap\delta=\varnothing\). Similarly, if \(\delta=\sfC_b\), then \(\alpha\cap\delta\neq\varnothing\) would contradict \(\alpha\cap\sfC_b=\varnothing\).
Finally, if \(\alpha=\delta\), then \(\alpha,\sfC_b,\sfC_c\) would be three pairwise disjoint embedded cycles, contrary to the hypothesis.

Thus the ordered cycles \(\alpha,\sfC_b,\sfC_c,\delta\), corresponding respectively to
\(a,b,c,d\), form a \(P_4\)-cycle configuration.
\end{proof}

\begin{theorem}[Graphical \(P_4\)-criterion]\label{Thm:GraphicalP4Criterion}
For a connected finite simplicial graph \(\Graf\), the following are equivalent:
\begin{enumerate}[label=\textup{(\arabic*)}, ref=\arabic*]
\item \(\bbA(P_4)<\GBG_2(\Graf)\);
\item \(P_4<\EO\);
\item \(\Graf\) contains either
\begin{enumerate}
\item three pairwise disjoint embedded cycles, or
\item a \(P_4\)-cycle configuration.
\end{enumerate}
\end{enumerate}
\end{theorem}

\begin{proof}
The equivalence of \textup{(1)} and~\textup{(2)} follows from \Cref{Thm:RankTwoGBG}.
The implication \(\textup{(3)}\Rightarrow\textup{(2)}\) is \Cref{Prop:P4SourcesGBG2}.
Conversely, suppose that \(P_4<\EO\). If \(\Graf\) contains three pairwise disjoint embedded cycles, then \textup{(3)(a)} holds. Otherwise, \Cref{Lem:P4CycleExtraction} gives a \(P_4\)-cycle configuration, so \textup{(3)(b)} holds.
\end{proof}

\begin{corollary}[Algorithmic \(P_4\)-embeddability]\label{Cor:AlgorithmicP4Embeddability}
There is an algorithm which, given a connected finite simplicial graph \(\Graf\), decides whether \(\bbA(P_4)<\GBG_2(\Graf)\).
\end{corollary}
\begin{proof}
A finite graph has only finitely many embedded cycles, and these can be enumerated. One may therefore check, by finite enumeration, whether \(\Graf\) contains three pairwise disjoint embedded cycles or four distinct embedded cycles forming a \(P_4\)-cycle configuration. The conclusion follows from \Cref{Thm:GraphicalP4Criterion}.
\end{proof}

\begin{corollary}[Path and forest RAAGs]\label{Cor:PathForestRAAGsGBG2}
For \(m\geq4\), let \(P_m\) denote the path graph on \(m\) vertices. Then the following are equivalent:
\begin{enumerate}
\item \(\bbA(P_m)<\GBG_2(\Graf)\);
\item \(\Graf\) contains either a \(P_4\)-cycle configuration or three pairwise disjoint embedded cycles.
\end{enumerate}
Moreover, these conditions are equivalent to \(\GBG_2(\Graf)\) containing \(\bbA(T)\) for every finite forest \(T\).
\end{corollary}
\begin{proof}
Since \(P_4\) is an induced subgraph of \(P_m\), one has \(\bbA(P_4)<\bbA(P_m)\). Hence
\[\bbA(P_m)<\GBG_2(\Graf)\quad\Longrightarrow\quad\bbA(P_4)<\GBG_2(\Graf).\]
Conversely, \(P_m\) is a finite forest, so \cite[Theorem~1.8]{KK13} gives \(\bbA(P_m)<\bbA(P_4)\). Thus
\[\bbA(P_m)<\GBG_2(\Graf)\qquad\Longleftrightarrow\qquad\bbA(P_4)<\GBG_2(\Graf).\]
The first assertion now follows from \Cref{Thm:GraphicalP4Criterion}.

More generally, every RAAG defined by a finite forest embeds into \(\bbA(P_4)\) by \cite[Theorem~1.8]{KK13}. Hence \(\bbA(P_4)<\GBG_2(\Graf)\) implies that \(\GBG_2(\Graf)\) contains \(\bbA(T)\) for every finite forest \(T\). The converse follows by taking \(T=P_4\).
\end{proof}

\printbibliography

\end{document}